\documentclass[11pt]{article} 
\usepackage{amsmath,amssymb,amsthm}
\usepackage{a4wide}
 
\newtheorem{theorem}{Theorem}

\newtheorem{corollary}[theorem]{Corollary}
\newtheorem{proposition}{Proposition}
\newtheorem{lemma}{Lemma} 
\newtheorem{claim}{Claim}
\newtheorem{remark}{Remark}

\numberwithin{claim}{section}
\numberwithin{equation}{section}
\numberwithin{lemma}{section}
\numberwithin{proposition}{section}

\newcommand{\RR}{\mathbb{R}}

\newcommand{\px}{\partial_x}
\newcommand{\sud}{\sum_{j=1,2}}
\newcommand{\qun}{R_1}
\newcommand{\qde}{R_2}
\newcommand{\qtud}{\widetilde R_j}
\newcommand{\qtun}{\widetilde R_1}
\newcommand{\qtde}{\widetilde R_2}
\newcommand{\qud}{R_j}
\newcommand{\xun}{x_1}
\newcommand{\xde}{x_2}
\newcommand{\qb}{\overline{R}}
\newcommand{\FF}{F}
\newcommand{\FTF}{\widetilde F}
\newcommand{\OO}{\mathcal{O}}
\newcommand{\SCS}{\mathcal{S}}
\newcommand{\SSr}{S}
\newcommand{\SSt}{\widetilde S}

\newcommand{\GGr}{G}
\newcommand{\HH}{H}
\newcommand{\ETE}{\widetilde E}

\newcommand{\YY}{\mathcal{Y}}
\newcommand{\YYzz}{Y_0}
\newcommand{\YYrr}{Y}
\newcommand{\pyy}{y}
\newcommand{\pzz}{z}
\newcommand{\pyyb}{{\bar y}}
\newcommand{\summu}{{\bar \mu}}
\newcommand{\TSR}{T}
\newcommand{\TPP}{T''}
\newcommand{\TPS}{T'}
\newcommand{\TTPP}{T'''}
\newcommand{\yy}{y}
\newcommand{\JJ}{\mathcal{J}}
\newcommand{\iscal}{J}
\newcommand{\llb}{[}
\newcommand{\rrb}{]}

\begin{document}

\title{Inelastic interaction of nearly equal solitons
	 for the \\ quartic gKdV equation\footnote{
This research was supported in part
by the Agence Nationale de la Recherche (ANR ONDENONLIN).}}
\author{Yvan Martel$^{(1)}$ \and Frank Merle$^{(2)}$}
\date{\small
(1) Universit\'e de Versailles Saint-Quentin-en-Yvelines and IUF
\\
(2) Universit\'e de Cergy-Pontoise and IHES
}
\maketitle
\begin{abstract}
This paper presents a complete description of the interaction of two solitons with nearly equal speeds for the quartic (gKdV) equation
\begin{equation} \label{eq:KDVi} 
\partial_t u + \partial_x (\partial_x^2 u + u^4) =0,\quad t,x\in \RR. \end{equation}
For $c>0$, $y_0\in \RR$, we call soliton a solution of \eqref{eq:KDVi} of the form
$
	R_{c,y_0}(t,x)=Q_{c}(x-c t -y_0),
$
where 
$Q_c''+Q_c^4=c Q_c.$
Since \eqref{eq:KDVi} is not an integrable model, the general question of the  collision of two given solitons $ {R}_{c_1,y_1}$, $ {R}_{c_2,y_2}$ with $c_1\neq c_2$ is an open problem.

We focus on the special case where the two solitons have nearly equal speeds~: let
$U(t)$ be the solution of \eqref{eq:KDVi} satisfying
$$
\lim_{t\to -\infty} \|{U}(t) -  Q_{c_1^-}(.-c_1^- t) - Q_{c_2^-}(.-c_2^- t )\|_{H^1} = 0,
$$
for $\mu_0= (c_2^--c_1^-)/(c_1^-+c_2^-) >0 $ small. By constructing an approximate solution of   \eqref{eq:KDVi},
we  prove in particular that, for all time $t\in \RR$, 
$${U}(t) = {Q}_{c_1(t)}(x-y_1(t)) + {Q}_{c_2(t)}(x-y_2(t)) + {w}(t) \quad \text{where} \quad
 \|w(t)\|_{H^1}\leq |\ln \mu_0| \mu_0^2,$$
with
$y_1(t)-y_2(t)> 2 |\ln  \mu_0| + O(1).
$
These estimates mean that the two solitons are preserved by the interaction and that for all time they are separated by a large distance, as in the case of the integrable KdV equation in this regime.

However,  unlike in the integrable case, we prove that the collision is not perfectly elastic, in the following sense:
$$
\lim_{t\to +\infty} c_1(t) > c_2^- , \quad  \lim_{t\to +\infty} c_2(t)< c_1^-\quad \text{and}\quad {w}(t) \not \to 0\text{ in $H^1$ as $t\to +\infty$.}
$$
The arguments developed in this paper are general and can be applied to different models.

\end{abstract} 

\section{Introduction}

We consider the generalized KdV equation
\begin{equation}\label{eq:gKdV} 
	\partial_t u + \partial_x( \partial_x^2 u + u^p)=0,\quad t,x\in \RR.
\end{equation}
Recall that $p=2$ and $3$ correspond respectively to the (KdV) and (mKdV) equations, which are completely integrable models.
In this paper, we focus on the nonintegrable case $p=4$.

As usual, we call \emph{solitons} solutions of \eqref{eq:gKdV} of the form
$R_{c,y_0}(t,x)= Q_c(x-ct-y_0)$, for $c>0$, $y_0\in \RR$, where $Q_c(x)=c^{\frac 1{p-1}} Q(\sqrt{c} x)$ and    $Q$ satisfies
$$Q'' + Q^p = Q, \quad 
Q(x)=\Bigg(\frac {p+1} {2 \cosh^2\big(\frac {p-1} 2 x\big)} \Bigg)^{\frac 1{p-1}}
.$$

\subsection{Review of the collision problem for the (gKdV) model}

We start with the classical integrable (KdV) equation
\begin{equation}\label{eq:KdV} 
	\partial_t u + \partial_x( \partial_x^2 u + u^2)=0.
\end{equation}

First, it is very well-known that  the (KdV)  equation has explicit pure
$N$-soliton solutions (see   \cite{HIROTA},   \cite{WT},  \cite{Miura}).
Namely, for any given $c_1>\ldots>c_N>0$, $y_1^-,\ldots,y_N^-\in \mathbb{R}$, there exists an explicit multi-soliton solution $u(t,x)$ of \eqref{eq:KdV} which satisfies
\[ 
\lim_{t\to \pm \infty}	\biggr\|u(t)- \sum_{j=1}^N Q_{c_j} (.-c_jt-y_j^\pm)\biggr\|_{H^1(\RR)}   =  0,
\]
for some $y_j^+$ such that the shifts $\Delta_j=y_j^+-y_j^-$ depends on the $(c_k)$.
Recall that  explicit formulas for such solutions were derived using the inverse scattering transform.

Recall also that  before the discovery of these explicit solutions, Fermi, Pasta and Ulam \cite{FPU} and Zabusky and Kruskal  \cite{KZ} discovered numerically 
remarkable phenomena related to solitons collision. 
Lax (\cite{LAX1}) has developed a mathematical framework to study these problems, known now as complete integrability and then  other decisive developments
appeared, such as the inverse scattering transform (for a review on this theory, we refer  for example to
Miura \cite{Miura}).

The $N$-solitons are fundamental in studying the properties of general solutions of  equation \eqref{eq:KdV} in particular because of the so-called decomposition property  
(Kruskal \cite{KRUSKAL}, Eckhaus and Schuur \cite{EcSc}, \cite{Schuur},
Cohen \cite{Cohen}), which states that  the asymptotic behavior in large time of any sufficiently regular and decaying solution is governed by a finite number of solitons.

Stability and asymptotic stability of $N$-solitons were studied by Maddocks and Sachs \cite{MS} in $H^N$ by variational techniques   and in the energy space $H^1$ by Martel, Merle and Tsai \cite{MMT}.

\medskip
 Second, recall that LeVeque \cite{Le}   investigated the behavior of the explicit $2$-soliton  $U_{c_1^-,c_2^-}$ satisfying
\begin{equation}\label{2sol}
	\lim_{t\to \pm \infty}	\biggr\|U_{c_1^-,c_2^-}(t)- Q_{c_1^-}(.-c_1^- t-y_1^\pm) - Q_{c_2^-}(.-c_2^- t-y_2^\pm)\biggr\|_{H^1(\RR)}   =  0,
\end{equation}
in the asymptotic $\mu_0 = \frac {c_2^--c_1^-} {c_1^-+c_2^-}>0$  small i.e. for nearly equal solitons. 
In \cite{Le}, the following estimate
\begin{equation}\label{eq:Le}
\sup_{t,x\in \RR} \left| U_{c_1^-,c_2^-}(t,x) - 
  Q_{c_1(t)}(x-y_1(t)) -    Q_{c_2(t)}(x- y_2 (t))\right| \leq C \mu_0^2,
\end{equation}
is proved for some explicit functions $c_j(t)$, $y_j(t)$. Moreover, it is proved that
\begin{equation}\label{eq:Le2}
	\min_{t\in \RR} (y_1(t)-y_2(t)) =   2 | \ln \mu_0 | + O(1),
\end{equation}
which means that the minimum separation between the two solitons goes to $\infty$ as $\mu_0\to 0$.
See  also  \cite{EiOh}, where  Ei and Ohta investigated formally the dynamics of interacting pulses for several models.
\medskip

 Apart from integrability theory,
questions on interaction of solitons have been studied since the 60's from both numerical and experimental  points of view.
Fermi, Pasta and Ulam \cite{FPU}, Zabusky and Kruskal \cite{KZ} and Zabusky \cite{Z} have
introduced nonlinear systems and computed interaction of nonlinear objects by numerics.
Since then, many other systems have been studied numerically. Bona et al. \cite{BPS}, and
Kalisch and Bona \cite{KB}, focused on the problem of collision of two solitary waves for the Benjamin   and the BBM equations.
Shih \cite{SHIH} studied the case of the generalized KdV equations with  nonlinearities $|u|^p$  for 
some non integer values of $p$. Recently, Craig et al. \cite{Craig} presented  new numerical and experimental works for the water wave problem. These works give evidence  that in general, unlike for
the pure $N$-solitons of the integrable case, the collision of two solitary waves fails to be elastic
by a small dispersion.
For experimental literature, see e.g. Weidman and Maxworthy \cite{WM},
Hammack et al. \cite{HHGY}, Craig et al. \cite{Craig}.

\medskip

We now recall two recent mathematical works related to the interaction of solitons for the  (gKdV) equations, which is otherwise a widely open question.

\medskip

First, Mizumachi \cite{Mi}   studied rigorously the interaction of two solitons of nearly equal speeds for \eqref{eq:gKdV} both for  integrable ($p=3$) and nonintegrable ($p=4$) cases. Consider $u_0$ close to the sum of two solitons $Q(x)+ Q_c(x+Y_0)$, where
$Y_0>0$ is large and $c$, close to $1$, satisfies $c-1 \leq  e^{-\frac 12 {Y_0}}$. Let $u(t)$ be  the corresponding solution of \eqref{eq:gKdV}. If $c-1 = e^{-\frac 12 {Y_0}} >0 $, the quicker soliton is initially on the left of the other soliton:  one could think that the two solitons have to cross at some positive finite time.
However, Mizumachi proved (see Theorem 1.1 in \cite{Mi}) that the interaction of the two solitons being repulsive, for $c-1$ small enough, the two solitons remain separated for all positive time and eventually $u(t)$ behaves as 
\begin{equation}\label{eq:mi}
u(t)=  Q_{c_1^+}(.-c_1^+ t -y_1^+) + Q_{c_2^+}(.-c_2^+ t -y_2^+) + w(t,x),
\end{equation}
for large time, for some $c_1^+>c_2^+$ close to $1$ and $w$ small in some space.
The analysis part in \cite{Mi} relies on scattering techniques due to Hayashi and Naumkin \cite{HN1,HN2} and on the use of spaces of exponentially decaying functions (introduced in this context  by Pego and Weinstein \cite{PW}).  

Interestingly, note that using Mizumachi's result for $t<0$, backwards in time (i.e. using the symmetry $x\to -x$, $t\to -t$ of  \eqref{eq:gKdV}), one can construct a class of global solutions $u(t)$ of \eqref{eq:gKdV} such that for all $t\in \RR$, $u(t)$ is close to the sum of two separated solitons with nearly equal speeds, where the minimal separation between the two solitons is large, and satisfying \eqref{eq:mi} both at $t \sim - \infty$ and $t\sim + \infty$.

The situation is thus at the main order similar to the one described in the integrable case by LeVeque \cite{Le}.
However, after Mizumachi's work, two important questions remain open in this regime for the nonintegrable case: 
the global   stability in the energy space $H^1$ of the $2$-soliton structure and the existence or nonexistence of pure $2$-soliton solutions. By analogy with the integrable case, the expression \textit{pure $2$-soliton} denotes a solution of   \eqref{eq:gKdV}   which satisfies
\begin{equation}\label{eq:pure}
\lim_{t\to \pm \infty} \|u(t) -   Q_{c_1^\pm}(.-c_1^\pm t -y_1^\pm) -  Q_{c_2^\pm}(.-c_2^\pm t -y_2^\pm)\|_{H^1} = 0.
\end{equation}
Note that if \eqref{eq:pure} holds both at $-\infty$ and $+\infty$, then necessarily
$c_j^-=c_j^+$ for $j=1,2$ (see Lemma \ref{le:am} and \cite{MMcol2}, pp. 68, 69).

\medskip

Second, in a series of recent works (\cite{MMcol1}, \cite{MMcol2}), the authors of the present paper have addressed the problem of collision of two solitons of \eqref{eq:gKdV}  for a general nonlinearity in the case where one soliton is supposed to be large with respect to the other one, i.e. assuming that the ratio of the speeds $c_1$, $c_2$ of the solitons satisfies $c=c_1/c_2\ll 1$. 
 
In this regime, we were able to compute an approximate solution as a series of powers of $c$ (in some sense) describing the collision up to any prescribed order of $c$, which allowed us to understand the collision phenomenon.

In \cite{MMcol2}, under general assumptions on the nonlinearity, it is proved that the two solitons   are preserved at the main order by the collision. While it is natural for the large soliton   to be preserved by perturbation using standard $H^1$ stability results, it is surprizing that in a general nonintegrable situation, the small soliton   is also preserved. 

In \cite{MMcol1}, concerning   the special case of quartic nonlinearity $p=4$ in \eqref{eq:gKdV}, still for   $c\ll 1$,  the perturbation due to the collision could be computed in more details. 
In particular, an explicit lower bound on the defect due to the collision was obtained. As a consequence, there exists no pure $2$-soliton solution in this context.

These results have been later extended to the case of the BBM equation in \cite{MMM}.

\subsection{Main results}
In this paper, we focus on the quartic (gKdV) equation
\begin{equation}\label{eq:KDVt}
 \partial_t u + \partial_x (\partial_x^2 u   + u^4) =0,\quad t,x\in \RR. 
\end{equation}
Recall  that  the Cauchy problem for \eqref{eq:KDVt} is globally well-posed in $H^1$ (see Kenig, Ponce and Vega \cite{KPV}), and that any $H^1$ solution $u(t,x)$ of \eqref{eq:KDVt} satisfies for all $t\in \RR$,
\begin{align}
	&	\int u^2(t) = M(u(t)) = M(u(0)) \qquad \text{(mass)}	\label{eq:i5}\\
	&	\int (\partial_x u)^2 (t)  -\frac 25 u^5(t)  = \mathcal{E}(u(t)) = \mathcal{E}(u(0))\qquad \text{(energy)}	\label{eq:i6}
\end{align}
Our objective is to  describe the interaction of two solitons with almost equal speeds for   \eqref{eq:KDVt}   and in particular to answer in this specific regime the two main questions raised above: Is the $2$-soliton structure stable globally in time  in $H^1$? Does there exist a pure $2$-soliton solution? 

\medskip

For our first result, Theorem \ref{TH:1}, we focus on special solutions of   \eqref{eq:KDVt}   which behave as a $2$-soliton asymptotically as $t\to -\infty$. Recall  that the existence (as well as uniqueness properties) of such solutions was proved in   \cite{Ma2} (see also \cite{MMT}).
For $c_2^- - c_1^- >0$ small, and any $x_1^-$, $x_2^-$, let $u(t)$ be the unique solution of \eqref{eq:KDVt} such that
\begin{equation}\label{eq:rd0}
	\lim_{t\to -\infty} \|u(t)- Q_{c_1^{-}} (. - c_1^{-} t -x_1^-) -Q_{c_2^{-}} (. - c_2^{-}t -x_2^-)\|_{H^1} =0.
\end{equation}
Let
\begin{equation}\label{eq:rd1}
	 c_0=\frac  {c_1^-+c_2^-}{2},\quad \mu_0 	=\frac {c_2^- - c_1^-}{c_1^-+c_2^-},\quad y_1^-=x_1^- \sqrt{c_0},  
	\quad y_2^-= x_2^- \sqrt{c_0}.\end{equation}
Then
\begin{equation}\label{eq:rd2}
 U(t,x)= c_0^{-1/3} u\left(c_0^{-3/2} t , c_0^{-1/2} (x+t)\right)
\end{equation}
solves
\begin{equation}\label{eq:KDV}
\partial_t U + \partial_x (\partial_x^2 U - U + U^4) =0,\quad t,x\in \RR,
\end{equation}
 and it is the unique solution of \eqref{eq:KDV} satisfying
\begin{equation}\label{eq:rd3}
	\lim_{t\to -\infty} \left\|U(t)- Q_{1-\mu_0} (. +\mu_0 t -y_1^-) -Q_{1+\mu_0} (. -\mu_0 t -y_2^-)\right\|_{H^1} =0.
\end{equation}
This means that from the general case \eqref{eq:rd0}, we can reduce ourselves to a symmetric situation for the asymptotic speeds at $-\infty$.

In this context, we now state our main results.

\begin{theorem}[Inelastic interaction of two solitons with nearly equal speeds]\label{TH:1} 
There exist $C,c, \sigma, \mu_*>0$ such that the following holds.
For $0<\mu_0<\mu_*$, let $U(t)$ be the unique solution of {\eqref{eq:KDV}} such that
\begin{equation}\label{eq:th1}
	\lim_{t\to -\infty}\left\|U(t) - Q_{1-\mu_0}(. + \mu_0 t + \tfrac 12 Y_0 + \ln 2) - Q_{1+\mu_0}(.-  \mu_0 t -\tfrac 12 Y_0-  \ln 2)\right\|_{H^1(\RR)} = 0,
\end{equation}
where 
$Y_0 = |\ln(\mu_0^2/\alpha)|$ and $\alpha=  {12 (10)^{2/3}}(\int Q^2)^{-1}.$
Then 
\begin{enumerate}
	\item [\rm (i)] Global behavior of $2$-solitons. There exist $\mu_1(t),\mu_2(t),y_1(t),y_2(t)$ of class $C^1$ such that
	$$
	w(t,x)=U (t) -  Q_{1+\mu_1(t)}(.-y_1(t)) - Q_{1+\mu_2(t)}(.-y_2(t))
	$$
	satisfies,  for all $t\in \RR$,
	\begin{align}
	&	 \|w(t)\|_{H^1(\RR)} \leq C |\ln \mu_0|^{1/2} \mu_0^{2},\quad	\Big| \min_{t\in \RR} (y_1(t)-y_2(t)) -  Y_0 \Big|\leq   C |\ln \mu_0|^\sigma \mu_0^{3/2},\label{eq:th5} \\
	& \sum_{j=1,2}| \mu_j(t) + (-1)^j \mu_0 \tanh(\mu_0 t)|+\sum_{j=1,2} | \dot y_j(t) - \mu_j(t)|   \leq  C |\ln \mu_0|^2 \mu_0^{2}.\label{eq:th6}
	\end{align}
	\item [\rm (ii)] Asymptotics and defect.
	The limits $\mu_1^+=\displaystyle\lim_{ +\infty} \mu_1$, $\displaystyle \mu_2^+=\lim_{+\infty} \mu_2$ exist and
		\begin{align} 
			&	\lim_{t\to +\infty}\| w(t)\|_{H^1(x>-(99/100) t)}=0,
			\quad \liminf_{t\to +\infty }  \|w(t)\|_{H^1(\RR)}\geq c \mu_0^3,\label{eq:th8}\\  
					&		 c  \mu_0^5	 \leq   {\mu_1^+}  - \mu_0 \leq  C |\ln \mu_0|^{2\sigma} \mu_0^4 , \quad  c  \mu_0^5	 \leq   -{\mu_2^+}  - \mu_0 \leq  C |\ln \mu_0|^{2\sigma} \mu_0^4 . \label{eq:th9}
	\end{align}
	\end{enumerate}
\end{theorem}
It follows immediately from the lower bound \eqref{eq:th8} that
\emph{no pure $2$-soliton exists}, which is a new result in this regime.

As a consequence of the proof of Theorem \ref{TH:1}, the $2$-soliton structure is globally stable in the energy space $H^1$.

\begin{theorem}[Stability result in the energy space]\label{TH:2}
There exist $C,\sigma, \mu_*>0$, such that the following holds.
Let $\tilde \mu_0\in \RR$ and $\tilde Y_0>0$ be such that
\begin{equation}\label{eq:th2a}
	\mu_0  = \left(\tilde \mu_0^2 + 4 \alpha e^{-\tilde Y_0}\right)^{1/2} <\mu_*,
\end{equation}
where $\alpha$ is defined in Theorem \ref{TH:1}.
Let $  u_0\in H^1$ be such that
\begin{equation}\label{eq:th2b}
	\|   u_0 - Q_{1-\tilde \mu_0}(.- \tfrac 12 \tilde Y_0) - Q_{1+\tilde \mu_0}(. +\tfrac 12  \tilde Y_0) \|_{H^1(\RR)}\leq \omega \mu_0,
\end{equation}
where $0<\omega<|\ln \mu_0|^{-2}$, and let $u(t)$ be the solution of {\eqref{eq:KDV}} such that $u(0)=  u_0$. Then,
there exist $T(t)$, $X(t)$ of class $C^1$ such that, for all $t\in \RR$,
\begin{equation}\label{eq:th2c}
	\|u(t + T(t), . + X(t)) - U(t)\|_{H^1(\RR)} + |\dot X(t)|+ \mu_0 |\dot T(t)| \leq C \omega \mu_0 + C |\ln \mu_0|^\sigma \mu_0^{ 3 / 2 },
\end{equation}
where $U(t)$ is the solution defined in Theorem \ref{TH:1}.
\end{theorem}

\noindent\emph{Comments on the results:}
\medskip

1. For the specific solution $U(t)$ considered in Theorem \ref{TH:1}, the dynamics of the  parameters $\mu_j(t),$ $y_j(t)$  are closely related to
 the function 
\begin{equation}\label{eq:HHi}
\YYrr(t)=\YYzz + 2 \ln(\cosh(\mu_0 t))     \text{ which solves }  	\ddot \YYrr = 2 \alpha e^{-\YYrr}, \ \lim_{\pm \infty} \dot \YYrr= \pm 2 \mu_0,\ \dot \YYrr(0)=0.
\end{equation}
This choice is motivated by the symmetry in time of $Y(t)$. But by time and space translations and by the scaling argument \eqref{eq:rd1}-\eqref{eq:rd2}, we can extend Theorem \ref{TH:1} to any solution of \eqref{eq:KDVt} satisfying \eqref{eq:rd0}. 

More detailed information on the behavior of $U(t)$ and the parameters $\mu_j(t)$, $y_j(t)$ is available in Proposition \ref{pr:st}. 
Using refined asymptotic techniques (see \cite{MMas2} and \cite{MMcol1}),  one   can prove in the context of Theorem \ref{TH:1} that  
$\displaystyle\lim_{t\to +\infty} (y_1(t)-c_1^+ t)$ and $\displaystyle\lim_{t\to +\infty} (y_2(t)-c_2^+ t)$ exist.

Finally, from the critical Cauchy theory developed for \eqref{eq:KDVt} by Tao \cite{Tao}, one expects that in the context of Theorem \ref{TH:1}, $w(t)$ scatters as $t\to +\infty$. In particular, apart from the two main solitons, the solution $U(t)$ should not contain any other soliton but only dispersion.

\medskip

2. Theorems \ref{TH:1} and \ref{TH:2} completely answer the two questions raised before concerning the interaction of two solitons of almost equal speeds.

Note in particular that the lower bounds in estimates \eqref{eq:th8} and \eqref{eq:th9} measur the defect of $U(t)$ at $+\infty$; in other words, they quantify in the energy space $H^1$ the inelastic character of the collision of $2$ solitons of \eqref{eq:KDV} in the regime where $\mu_0$ is small.
Comparing \eqref{eq:th5} and \eqref{eq:th8}, and lower and upper bounds in \eqref{eq:th9}, we see that there is a gap between the lower and the upper bounds for the size of the defect. It is an open problem. A similar open question related to the size of the defect appears in \cite{MMcol1}.

The information on the limiting values of the scaling parameters $\mu_1^+$ and $\mu_2^+$
(i.e. \eqref{eq:th9}) is more precise that the information one could deduce from \eqref{eq:th6}.
To prove \eqref{eq:th9}, we use energy and mass conservations (see \eqref{eq:i5}--\eqref{eq:i6}) and the fact that $w(t)$ goes to zero locally around each soliton. In this way, we obtain sharp information on the scaling parameters, and in particular the monotonicity formulas:
$\mu_1^+ > \mu_0$, $\mu_2^+< - \mu_0$ (see Lemma \ref{le:am}).

Theorem \ref{TH:2} is a global  stability result concerning solutions which are close to the sum of two solitons of nearly equal speeds. A sharper result  is presented in Proposition \ref{PR:SHARP}.

\medskip

3. The proof of Theorem \ref{TH:1} relies on new computations, in particular for the construction of a relevant approximate solution.
The  strategy developed in this paper is expected to be quite general:  
it will be extended to the BBM equation in \cite{MMprepa} and can also be extended to the (gKdV) equation with general nonlinearity. The proof of  the lower bounds on $\|w(t)\|_{H^1}$ 
as $t\to +\infty$ in \eqref{eq:th8} is in the spirit of Liouville theorem for the (gKdV) equation -- see e.g. \cite{MM1}.

\subsection{Strategy of the  proofs}

Recall that for the integrable KdV equation, the interaction of two solitons of almost  equal speeds can be completely described using the  explicit formulas for $2$-solitons. From \cite{Le}, considering $u(t)$ a typical $2$-soliton solution of \eqref{eq:KdV} with almost equal speeds,  the two solitons remain well separated  for all time and eventually exchange their speeds:
\begin{equation}\label{2solii}\begin{split}
	&  u(t) \sim  Q_{c_1(t)}(.-y_1(t))+ Q_{c_2(t)}(.-y_2(t)),\\
	&   y_1(t)-y_2(t)= y(t) \sim  Y(t),\quad 
	\lim_{- \infty} c_1 = \lim_{+ \infty} c_2 = c_1^-,\  \lim_{- \infty} c_2=   \lim_{+ \infty} c_1 = c_2^-, 
\end{split}\end{equation}
where $Y(t)$ is solution of $\ddot Y = 2 \bar \alpha e^{-Y}$ for some $\bar \alpha$.
Their interaction is repulsive  and since solitons have exponentially decay in space, they interact weakly. Moreover, the solution has a symmetry with respect to the transformation
$x\to -x$, $t\to -t$.

\medskip

 In this paper, we focus on the quartic (gKdV) equation, which is not integrable and not close to any integrable model.
 Recall that for this equation, we have described in \cite{MMcol1} the collision of two solitons with very different speeds. Indeed, considering
 two solitons $Q_{c_1}$, $Q_{c_2}$ such that   $c_2\gg c_1$, it follows from \cite{MMcol1} that the two soliton structure is stable and that the collision is almost elastic but not exactly elastic. Let us sketch the main steps of the proofs in \cite{MMcol1}:
 
 (1) First, we   construct an approximate solution to the problem in  the collision region. The approximate solution has the form of a series in terms of $c=c_1/c_2$ and involves a delicate algebra.
 
 (2) Second, using asymptotic arguments, we   justify that the solution is close to the approximate solution (so that the description of the collision given by the approximate solution is relevant) and we   control  the solution in large time, i.e. for $|t|>T$.
 
 (3) Finally, we prove the inelastic character of the collision by a further analysis of the approximate solution. The defect is due to a nonzero extra term in the approximate solution after recomposition of the series. Thus, the defect is a direct consequence of the algebra underlying the construction of the approximate solution.
 
 \medskip
 
Turning back to our problem, keeping in mind the intuition of the integrable case, we prove that in the case of two solitons with almost same speeds, the description given in \eqref{2solii} persists at the main order for the  quartic (gKdV) equation and we describe precisely the interaction. Let us present the strategy of the proofs of the main results.

Theorem \ref{TH:2} is a direct consequence of  the proof of Theorem \ref{TH:1}. 
The proof of Theorem \ref{TH:1}, as in \cite{MMcol1},   follows from a combination of three different types of arguments.

\medskip

(1) We construct an approximate solution to the problem in terms of a series in $e^{-y(t)}$ where $y(t)=y_1(t)-y_2(t)$ is the distance between the two solitons. Whereas the first order of the interaction of the two solitons is $e^{-y(t)}$, we are able to compute the solution up to order $e^{-\frac 32 y(t)}$ -- see  Proposition \ref{PR:21} (this improves the  ansatz of \cite{Mi} limited to \eqref{2solii}).
This construction implies that the soliton parameters $c_1(t),$ $c_2(t),$ $y_1(t),$ $y_2(t)$ have to  satisfy an approximate differential system.

Note from Proposition \ref{PR:21} that the approximate solution contains a tail of order $e^{-y(t)}$ between the two solitons (as in the (KdV) case, see \cite{MMprepa}), which is relevant in the description of the exact solution, see Remark~\ref{rk:mm}.
We will see that the inelasticity is not related to this tail of order $e^{-y(t)}$.

It is only at  order $e^{-\frac 32 y(t)}$ in the construction of the approximate solution  that our analysis points out a deep difference between (KdV)  and nonintegrable (gKdV). For the quartic (gKdV) case, one cannot build an approximate solution at order $e^{-\frac32 y(t)}$ in the energy space. Indeed, at this order,  a  tail necessarily appears in the approximate solution at $\infty$ in space.
We then have to cut off this tail to obtain a rougher approximate solution in the energy space.

Note that the construction of the approximate solution in the present paper is   completely new.
Since for all time the distance $y(t)$ between the two solitons is very large, and since the interactions are exponential in $y(t)$, the approximate solution is found by separation of  the three variables  : $e^{-y(t)}$ and the coordinates of the two solitons.
The  approximate solution is thus of different nature  compared to  the one in \cite{MMcol1}.
 We believe that the techniques of the present paper should have wide applicability to other models
 (see the case of the  (BBM) equation in \cite{MMprepa}).

\medskip

(2) After the approximate solution is constructed, we introduce the following decomposition of the solution $U(t)$ defined in Theorem \ref{TH:1}:
$$
U(t,x) =  Q_{c_1(t)}(x-y_1(t))+ Q_{c_2(t)}(x-y_2(t)) + W(t,x) + \varepsilon(t,x),
$$
where $Q_{c_1(t)}(x-y_1(t))+ Q_{c_2(t)}(x-y_2(t)) + W(t,x)$ is the modulated approximate solution and 
$\varepsilon(t)$ is a rest term.
To prove stability of the two soliton structure, we have  to control both  the parameters $c_j(t)$ and $y_j(t)$  and the rest term $\varepsilon(t)$.

The  control of the rest term $\varepsilon(t)$ uses variants of techniques developed earlier for  large time stability and asymptotic stability of solitons and multi-solitons for the (gKdV) equations in the energy space (\cite{We2}, \cite{MM1}, \cite{MMT}, \cite{Ma2}). These techniques involve:
Liapunov functionals related to the stability of solitons, Virial identity and 
 the introduction of almost monotone variants of the conservation laws \eqref{eq:i5}--\eqref{eq:i6} around each soliton (consequences of the Kato identity \cite{Kato2}). 
Recall that these techniques apply only in  situations where the solitons are decoupled and were key arguments in several   recent developments on global behavior of solutions of the (gKdV) equation: blow up in the $L^2$ critical case (\cite{MMannals}), asymptotic stability, stability of multi-solitons and rigidity properties of the flow of the (gKdV) equations around solitons.

As a conclusion of this analysis on the rest term $\varepsilon(t)$ and of the control of the dynamical system satisfied by the parameters, we obtain the stability results of Theorem \ref{TH:1}. In particular, 
the dynamics of the parameters can be approximated  by the simple ODE: 
$$ \ddot Y =  2 \alpha e^{-Y} , \quad Y(0)= Y_0, \quad \dot Y(0)=0,$$
and for all time $t$, $\|\varepsilon(t)\|_{H^1} \leq e^{-\frac 54^- Y_0}$.
Observe that this  ODE  is the same as in the (KdV) case and seems to be  universal in this type of problems.

\medskip

(3) Finally, we prove by contradiction  that in the quartic case the interaction of two solitons  always produces a nonzero residual.  The proof of the lower bounds on the defect involves some more refined arguments but is based on the same idea.  
Assume  that $U(t)$ is  a pure $2$-soliton solution, 
the contradiction then follows from the following two facts:

On the one hand, by uniqueness properties (\cite{Ma2}), $U(t)$ satisfies $U(t,x)=U(-t,-x)$  up to translation in space and time. As a consequence, the parameters $c_j(t)$, $y_j(t)$ also have a symmetry property.

On the other hand, the dynamical system satisfied by $c_j(t)$, $y_j(t)$ is   non symmetric
by the transformation $x\to -x$, $t\to -t$ at   order $e^{-\frac 32 y(t)}$. Indeed,
the approximate solution is not symmetric since it has a tail of  this order  on the left  of the two solitons, to match the behavior of $U(t)$ at $t\to -\infty$. 

Since $U(t)$ is assumed to be pure at $\pm \infty$, it has special space decay properties at the left of the soliton, so that one can use a special functional  related to $L^1$  to refine the dynamical system. Now, solving the refined non symmetric dynamical system, we find a contradiction with the symmetry properties of  $y_1(t)$ and  $y_2(t)$.

\smallskip

The above strategy to prove  inelasticity  is  similar to the proof of a Liouville property, in the spirit of e.g.  \cite{MM1} and  \cite{MMannals}. Moreover, the $L^1$ functional mentioned above was introduced to prove instability results for  (gKdV) equations, see \cite{BSS} and \cite{MMannals}.
Note that the arguments in this step are different from the ones in \cite{MMcol1}, where the defect  is a direct consequence of a defect in the approximate solution.  

\medskip

We summarize the organization of the paper.
In Section 2, we construct an approximate solution to the problem. Section 3 is devoted to preliminary decomposition and stability results. We then prove the stability part of Theorems \ref{TH:1} and \ref{TH:2} in Section 4. Finally, in Section~5, we prove the inelasticity of the interaction.

\medskip

\noindent\textbf{Acknowledgments.}\quad 
The authors would like to thank Tetsu Mizumachi for some useful discussions about this problem.

\section{Construction of an approximate solution}

We denote by
       $\YY$  the set of functions $f\in C^\infty(\RR,\RR)$ such that
\begin{equation*}
    \forall j\in \mathbb{N},\ \exists C_j,\, r_j>0,\ \forall x\in \RR,\quad |f^{(j)}(x)|\leq C_j (1+|x|)^{r_j} e^{-|x|}. 
\end{equation*}

\begin{proposition}\label{PR:21}
There exist unique $ A_j(x) $, $ B_j(x) $, $ D_j(x)$, $\alpha$, $\beta$, $\delta$, $a$, $b_{j}$, $d_{j}$  $(j=1,2)$, $\sigma\geq 3$ and $0<\mu_*<1/10$ such that for any $0<\mu_0<\mu_*$,  the following hold.
\begin{itemize}
\item[\rm (i)] Properties of $A_j,$ $B_j,$ $D_j$ and $b_j$.
\begin{equation}\label{eq:p7}\begin{split}
&  A_j,B_j,D_j \in L^\infty(\mathbb{R}), \quad A_j',B_j',D_j'\in {\cal Y},\\
& -\lim_{\pm\infty} A_1= \lim_{\pm \infty} A_2=\pm \theta_A, \quad \theta_A = (10)^{2/3} \frac {\int Q}{\int Q^2} ,\quad 
\lim_{+\infty} D_1=\lim_{+\infty} D_2=0,\\
& \lim_{+\infty} B_1=\lim_{+\infty} B_2=0,\quad
\lim_{-\infty} B_1=- \lim_{-\infty} B_2<0 ,
\end{split}\end{equation}
and $A_j$, $B_j$, $D_j$ satisfy the orthogonality conditions of Lemmas \ref{LE:23}, \ref{LE:24}, and \ref{LE:25}.

Moreover,
\begin{equation}\label{eq:p8}
	b_1 \neq b_2. 
\end{equation}
 \item[{\rm (ii)}] Definition of the approximate solution $V_0(x;\Gamma)$.
For $\Gamma= (\mu_1,\mu_2,y_1,y_2)$,  define
\begin{equation}\label{eq:p9}
\begin{split}
  V_0(x;\Gamma)	 &	=	Q_{1+\mu_1}(x-y_1)+ Q_{1+\mu_2}(x-y_2)\\
   &	+ e^{-(y_1-y_2) } \left(A_1(x-y_1)+A_2(x-y_2)\right)    \\
   		& - 2( 10)^{-2/3} \theta_A (\mu_1-\mu_2) x Q(x-y_1) Q(x-y_2)\\
		&	+  (y_1-y_2) \, e^{-(y_1-y_2)} \left(\mu_1B_1(x-y_1)+\mu_2 B_2(x-y_2)\right) \\ &
			+  e^{-(y_1-y_2)} \left(\mu_1D_1(x-y_1)+\mu_2 D_2(x-y_2)\right).
\end{split}
\end{equation}
\item[{\rm (iii)}] Equation of $V_0(x;\Gamma(t))$. Let $I$ be some time interval and  $\Gamma(t)=(\mu_1(t),\mu_2(t),y_1(t),y_2(t))$ be a $C^1$ function defined on $I$ such that, for some constant $K>1$, 
\begin{align} 
	\forall t\in I,\quad 
	& Y_0-1\leq y_1(t)-y_2(t)\leq K \,Y_0,
	\quad |\mu_1(t)|\leq 2 \mu_0,\quad |\mu_2(t)|\leq 2 \mu_0, \label{eq:p10}\\
	& |\mu_1(t)+\mu_2(t)|\leq   Y_0^2 e^{-Y_0},\quad
	|y_1(t)+y_2(t)|\leq Y_0^4 e^{-\frac 12 Y_0}, \label{eq:p11}
\end{align} 
where 	
\begin{equation}\label{eq:p12}
	Y_0=|\ln (\mu_0^2/\alpha)|\quad \text{and}\quad \alpha=\frac {12 (10)^{2/3}}{\int Q^2}  .
\end{equation}
Let 
\begin{equation}\label{eq:p13}
\begin{split}
V_0(t,x) = V_0(x;\Gamma(t)),\quad y(t)=y_1(t)-y_2(t).
\end{split}
\end{equation}
Then, on $I$, $V_0(t,x)$ solves
\begin{equation}\label{eq:p14} 
	\partial_t V_0 + \partial_x (\partial_x^2 V_0 - V_0 + V_0^4) = \ETE(V_0) +E_0(t,x) 
\end{equation}
where 
\begin{equation}\label{eq:p15}
 \ETE(V_0)=
 \sud (\dot \mu_j-{\cal M}_j)\frac {\partial V_0} {\partial \mu_j} 
- \sud(\mu_j - \dot y_j - {\cal N}_j) \frac {\partial V_0} {\partial y_j},
\end{equation}
\begin{equation}\label{eq:p16}
\begin{split}
 {\cal M}_1(t)&= 
\alpha\,  e^{-y(t)} +  \beta\,  \mu_1(t) y(t) e^{-y(t)} + \delta  \,\mu_1(t) e^{-y(t)} ,\\
 {\cal M}_2(t)&=
- \alpha\,  e^{-y(t)} -  \beta\,  \mu_2(t) y(t) e^{-y(t)} - \delta  \,\mu_2(t) e^{-y(t)},\\
 {\cal N}_1(t)&= 
a \,e^{-y(t)} + b_1 \,\mu_1(t) y(t) e^{-y(t)} + d_1 \,\mu_1(t) e^{-y(t)},\\
 {\cal N}_j(t)&= 
a \,e^{-y(t)} + b_2 \,\mu_2(t) y(t) e^{-y(t)} + d_2 \,\mu_2(t) e^{-y(t)},
\end{split}
\end{equation}
and  for some $C=C(K)>0$,
\begin{equation}\label{eq:p17}
	\forall t\in I, \quad \sup_{x\in \RR} \left\{\left(1+e^{\frac 12(x-y_1(t))} \right)| E_0(t,x)|\right\} \leq C \,\left(1+Y_0^\sigma\right) e^{-Y_0} e^{-y(t)}.
\end{equation}
\end{itemize}
\end{proposition}

Sections 2.1--2.5 are devoted to the proof of Proposition \ref{PR:21}.

\medskip

Since the function $V_0$ above which solves an approximate equation is not in $H^1$ (it has  nonzero limits at infinity in $x$), we have to localize the result, by  introducing an $L^2$ approximation of $V_0$, using a suitable cut-off function. Note that in the integrable case $p=2$, one would have obtained $V_0$ in $L^2$ using the same scheme - it is thus related to nonintegrability (see \cite{MMprepa}).

Let $\psi:\RR\to [0,1]$ be a $C^\infty$ function such that 
\begin{equation}\label{eq:18}
\hbox{$\psi'\geq 0$, $\psi\equiv 0$ on $\RR^-$, $\psi\equiv 1$ on $[\frac 12 , +\infty)$,}
\end{equation}
As a consequence of Proposition \ref{PR:21} and direct computations and estimates, we obtain the following result.

\begin{proposition}[$L^2$ approximate solution]\label{PR:22}
 Under the assumptions of Proposition \ref{PR:21} (i)--(iii), let\begin{equation}\label{eq:p19}
	V(x;\Gamma)= V_0(x;\Gamma) \psi\left( e^{-\frac 12 Y_0}  x +1 \right),\quad
 V(t,x)=V(x;\Gamma(t)).
 \end{equation}
 Then,
\begin{itemize}
	\item[\rm (i)]  Closeness to the sum of two solitons.
	\begin{equation}\label{eq:p19b}
		\|V	- \left\{Q_{1+\mu_1}(.-y_1 ) +Q_{1+\mu_2}(.-y_2) \right\}\|_{L^\infty}
		\leq  C  e^{-y},
	\end{equation}
	\begin{equation}\label{eq:p20}
		\|V 	- \left\{ Q_{1+\mu_1}(.-y_1)+ Q_{1+\mu_2}(.-y_2) \right\}\|_{H^1}
		\leq  C \sqrt{y} e^{-y}.
	\end{equation}
	\item[\rm (ii)]  Equation of $V(t,x)$.  
\begin{equation}\label{eq:p21}
	  \partial_t V + \partial_x (\partial_x^2 V - V + V^4) = \ETE(V) + E(t,x) \\
\end{equation}
where 
\begin{equation}\label{eq:p22}
\ETE(V) = \sud (\dot \mu_j-{\cal M}_j) \frac {\partial V}{\partial \mu_j} 
- \sud  (\mu_j - \dot y_j - {\cal N}_j)  \frac {\partial V}{\partial y_j},
\end{equation}
and for some $C=C(K)>0$,
\begin{equation}\label{eq:p23}\begin{split}
\forall t \in I,\quad &	\sup_{x\in \RR} \{\left(1+e^{\frac 12(x-y_1(t))} \right)| E(t,x)|\} \leq C Y_0^\sigma e^{-Y_0} e^{-y(t)},\\
&	\|E(t)\|_{H^1} \leq C Y_0^{\sigma} e^{-\frac 34 Y_0} e^{-y(t)}.
\end{split}\end{equation}
\end{itemize}
\end{proposition}

Proposition \ref{PR:22} is proved in Section 2.6.

\subsection{Preliminary expansion}

We set ($j=1,2$)
$$
	\qtud(t,x) = Q_{1+\mu_j(t)}(x-y_j(t)),\quad
	\qud(t,x) = Q(x-y_j(t)),
$$
$$
	 \Lambda \qtud(t,x)= \Lambda Q_{1+\mu_j(t)}(x-y_j(t)), \quad \Lambda \qud(t,x)  = \Lambda Q(x-y_j(t)),
$$
and similarly for $\Lambda^2 \qud$,
where $\Lambda Q_c$ and $\Lambda^2 Q_c$ are defined in Claim \ref{CL:A2}.

We introduce the notation 
\begin{equation}\label{eq:19}\begin{split}
& \hbox{$r(t)=O_{k}$, for $k\geq 1$, if  $\exists \sigma\geq 0$ s.t. } 
\sup_{t \in I}\{ e^{y(t)} |r(t)|\}\leq C (1+Y_0^\sigma)  e^{-(k-1) Y_0} ,\\
& \text{$f(t,x)=\OO_{k}$, for $k\geq 1$, if $  \sup_{x\in \RR} \left\{\left(1+e^{\frac 12 (x-y_1(t))}\right) |f(t,x)| \right\}=  O_k$.}
\end{split}\end{equation}

Define
$$
	\SCS(v)= \partial_t v + \px\left(\px^2 v -v +v^4\right),
$$
and $\mathcal{M}_j$, $\mathcal{N}_j$ as in \eqref{eq:p16} for $\alpha,$ $\beta$, $\delta$ and
$a$, $b_j$, $d_j$ to be determined.
We look for an approximate solution of $\SCS(v)=0$ under the form $v(t,x)=v(x;\Gamma(t))$,
$$
	v=\qtun+\qtde+w,
$$
where $w(t,x)=w(x;\Gamma(t))$, so that using the equation of $Q_c$ (see \eqref{eq:A1}),
and $\frac \partial{\partial \mu_1} \qtun = \Lambda \qtun$,
$\frac \partial{\partial y_1} \qtun = - \partial_x \qtun$,
\begin{equation}\label{eq:20}
	\SCS(v) = \ETE(v) +\FF + \FTF  + \GGr(w) + \HH(w)
\end{equation}
where
\begin{align*}
\ETE(v) &= \sud (\dot \mu_j-{\cal M}_j) \frac {\partial v}{\partial \mu_j} 
- \sud  (\mu_j - \dot y_j - {\cal N}_j)  \frac {\partial v}{\partial y_j}\\
\FF &= \px \left( \left( \qtun + \qtde  \right)^4 - \qtun^4 - \qtde^4 \right) \\
	\FTF& =  {\cal M}_1 \Lambda \qtun + {\cal M}_2 \Lambda \qtde + {\cal N}_1 \px \qtun + {\cal N}_2 \px \qtde,
 	\end{align*}
and
\begin{align*}
		\GGr(w)& =   \px\left[\px^2 w -w + 4\left(\qtun^3+\qtde^3\right)w   \right]
		+ \sud   \mu_j \frac {\partial w}{\partial y_j}\\
	\HH(w) &= \px \left[ \left(\qtun+\qtde+w\right)^4 - \left(\left(\qtun+\qtde\right)^4+4 \left(\qtun^3+\qtde^3\right)w \right)\right]\\ &+\sud {\cal M}_j \frac {\partial w}{\partial \mu_j}
	- \sud   {\cal N}_j \frac {\partial w}{\partial y_j}.
\end{align*}
In the rest of Section 2.1, we perform preliminary expansions of $\FF$ and $\FTF$.

\begin{lemma}[Expansion of $\FF$]\label{LE:21}
Under the assumptions of Proposition \ref{PR:21},
$$\FF = \px \left( \left( \qtun + \qtde  \right)^4 - \qtun^4 - \qtde^4 \right)
=\FF_A + \FF_B + \FF_D  + \OO_2$$
where
\begin{align*}
  \FF_A &= 4 (10)^{1/3} e^{-y} \px\left[ e^{-(x-y_1)} \qun^3+  e^{(x-y_2)} \qde^3\right], \\
  \FF_B &= 2 (10)^{1/3} ye^{-y} \px\left[ \mu_1 e^{-(x-y_1)} \qun^3+ \mu_2 e^{(x-y_2)} \qde^3\right]\\
  \FF_D & = 4 (10)^{1/3}\mu_1 e^{-y} \px\left[ e^{-(x-y_1)} \qun^2 \left(\tfrac 23   \qun+  \tfrac 1 2 (x-y_1)   \qun + \tfrac 32 (x-y_1) \px \qun \right)\right]\\
 & + 4    (10)^{1/3}\mu_2 e^{-y} \px\left[ e^{(x-y_2)} \qde^2 \left(\tfrac 23   \qde-  \tfrac 1 2 (x-y_2)   \qde
+ \tfrac 32 (x-y_2) \px \qde \right)\right]
.
\end{align*}
\end{lemma}
\begin{proof}
Before starting the proof of Lemma \ref{LE:21}, we claim the following estimates.

\begin{claim}\label{CL:21}
Assuming \eqref{eq:p10}--\eqref{eq:p12}, the following hold
\begin{equation}\label{eq:21}
\qtun(t,x) + |\partial_x  \qtun(t,x) | \leq C e^{-(1-2\mu_0) | x - y_1(t) | } ,
\end{equation}
\begin{equation}\label{eq:22}
\left| \qtun(t,x) - \left\{ \qun(t,x) + \mu_1(t) \Lambda \qun(t,x)\right\}\right| \leq C \mu_0^2(1+|x-y_1(t)|^2) e^{-(1-2\mu_0) |x-y_1(t)|},
\end{equation}
together with similar estimates for $\qtde$.

Moveover, for $\omega\geq 0$,
\begin{equation}\label{eq:23}
	 \Big(1+\sud |x-y_j|^\omega\Big)
	 e^{-(1-2\mu_0) |x-y_1|} e^{-(1-2\mu_0) |x-y_2|} = \OO_1,
\end{equation}
\begin{equation}\label{eq:24}
	\int \Big(1+\sud |x-y_j|^\omega\Big)
	 e^{-(1-2\mu_0) |x-y_1|}e^{-(1-2\mu_0) |x-y_2|} dx
	 \leq C \left(1+|y|^{1+\omega}\right) e^{-y}.
\end{equation}
\end{claim}
\begin{proof}[Proof of Claim \ref{CL:21}]
Since $Q(x)+|Q'(x)|\leq Ce^{-|x|}$ and $|\mu_1(t)|\leq 2 \mu_0$, we have
$$\qtun(t,x)+ |\partial_x  \qtun(t,x) |\leq C e^{-\sqrt{1-2\mu_0} | x - y_1(t) | }
\leq C e^{- (1-2\mu_0)  | x - y_1(t) | }.$$

To prove \eqref{eq:22}, we using the Taylor formula in the $\mu_1$ variable
(recall the notation $\Lambda Q$ from Claim \ref{CL:A2})
\begin{align*}
	\qtun(t,x)= Q_{1+\mu_1}(x-y_1) & = Q(x-y_1) + \mu_1 \Lambda Q(x-y_1)
	\\ & +   \mu_1^2 \int_0^1 (1-s) \Lambda^2 Q_{1+s\mu_1} (x-y_1) ds.
\end{align*}
From \eqref{eq:A7}, $|\Lambda^2 Q_{1+s\mu_1}(x)|\leq C
(1+|x|^2) e^{- (1-2\mu_0)  | x |}$ and \eqref{eq:22} follows.

To prove \eqref{eq:23} and \eqref{eq:24}, we argue as follows.
For $y_2<x<y_1$, we have
$$e^{-(1- 2 \mu_0) |x-y_1|}e^{-(1- 2 \mu_0) |x-y_2|}
= e^{-(1- 2 {\mu_0})y} \leq C e^{-y},$$
since \eqref{eq:p10} implies $ 2 {\mu_0  y } \leq C {\mu_0 Y_0} \leq 1$ for $Y_0$ large enough by \eqref{eq:p12}.

For $x>y_1>y_2$,
\begin{align*}
e^{-(1- 2 \mu_0) |x-y_1|}e^{-(1- 2 \mu_0) |x-y_2|}
&= e^{-(1- 2 {\mu_0})(2 x - y_1 -y_2)}
= e^{- 2 (1- 2 \mu_0) (x-y_1)} e^{- (1- 2 \mu_0) y }
\\ & \leq C e^{-\frac 32 |x-y_1|} e^{-y }.
\end{align*}
Arguing similarly for the case $x<y_2<y_1$, we prove \eqref{eq:23} and \eqref{eq:24}.
\end{proof}

We expand $\FF$,
$$
\FF = \partial_x \left( 4 \qtun^3 \qtde + 4 \qtun \qtde^3 + 6 \qtun^2 \qtde^2\right).
$$
We have immediately  $\partial_x\left(\qtun^2 \qtde^2 \right)= \OO_2$ (see \eqref{eq:21} and \eqref{eq:23}).

Now, we focus on the term $\px\left( \qtun^3 \qtde\right)$.
Using \eqref{eq:21}, \eqref{eq:22}, \eqref{eq:23}, $|\mu_j(t)|\leq 2 \mu_0\leq C e^{-\frac 12 Y_0}$, and the expression of $\Lambda Q$ in \eqref{eq:A8}, we obtain
\begin{align*}
\px\left(\qtun^3 \qtde\right)
& = \partial_x\left( \left(\qun^3 + 3 \mu_1 \qun^2 \Lambda \qun\right) \left(\qde + \mu_2 \Lambda \qde\right)\right)+\OO_2 \\
& = \partial_x\left(\qun^3 \qde + 3 \mu_1 \qun^2 \Lambda \qun \qde + \mu_2 \qun^3 \Lambda \qde\right)+\OO_2.
\end{align*}
Now, using the asymptotic behavior of $Q$, $Q'$ and $\Lambda Q$ at $+\infty$ (see \eqref{eq:A16} and \eqref{eq:A17}), we find
\begin{align*}
\px\left(\qtun^3 \qtde\right)
 & = (10)^{1/3} e^{-y} \px\left[ e^{-(x-y_1)} \qun^3\right] \\
& + 3(10)^{1/3}\mu_1 e^{-y} \px\left[ e^{-(x-y_1)} \qun^2 \Lambda \qun\right] \\
& + (10)^{1/3} \mu_2 e^{-y} \px\left[- \tfrac 12 ye^{-(x-y_1)} \qun^3  + \left(\tfrac 13 - \tfrac 12 (x-y_1)\right) e^{-(x-y_1)}  \qun^3 \right]+\OO_2.
\end{align*}
Therefore, using \eqref{eq:p11} and then the expression of $\Lambda Q$, we obtain
\begin{align*}
\px\left(\qtun^3 \qtde\right)
 & = (10)^{1/3} e^{-y} \px\left[ e^{-(x-y_1)} \qun^3\right] \\
& +\frac 12 (10)^{1/3} \mu_1 y e^{-y}  \px\left[ e^{-(x-y_1)} \qun^3\right] \\
& +  (10)^{1/3}\mu_1 e^{-y} \px\left[ 3 e^{-(x-y_1)} \qun^2 \Lambda \qun-\left(\tfrac 13 - \tfrac 12 (x-y_1)\right) e^{-(x-y_1)}  \qun^3 \right]+\OO_2\\
& = (10)^{1/3} e^{-y} \px\left[ e^{-(x-y_1)} \qun^3\right] \\
& +\frac 12 (10)^{1/3} \mu_1 y e^{-y}  \px\left[ e^{-(x-y_1)} \qun^3\right] \\
& +  (10)^{1/3}\mu_1 e^{-y} \px\left[ e^{-(x-y_1)} \qun^2 \left(\tfrac 23   \qun+  \tfrac 1 2 (x-y_1)   \qun
+ \tfrac 32 (x-y_1) \px \qun \right)\right]+\OO_2.
\end{align*}

Similar computations give
\begin{align*}
\px\left(\qtun \qtde^3\right) & = (10)^{1/3} e^{-y} \px\left[ e^{(x-y_2)} \qde^3\right] \\
& +\frac 12 (10)^{1/3} \mu_2 y e^{-y}  \px\left[ e^{(x-y_2)} \qde^3\right] \\
& +  (10)^{1/3}\mu_2 e^{-y} \px\left[ e^{(x-y_2)} \qde^2 \left(\tfrac 23   \qde-  \tfrac 1 2 (x-y_2)   \qde
+ \tfrac 32 (x-y_2) \px \qde \right)\right]+\OO_2.
\end{align*}

Lemma \ref{LE:21} is proved by combining these computations.
\end{proof}

\begin{lemma}[Expansion of $\FTF$]\label{LE:22}
\begin{equation*}
	\FTF		=   {\cal M}_1 \Lambda \qtun + {\cal M}_2 \Lambda \qtde + {\cal N}_1 \px \qtun + {\cal N}_2 \px \qtde  = \FTF_A + \FTF_B + \FTF_D + \OO_2,
\end{equation*}
where
\begin{align*}
\FTF_A & = \alpha e^{-y} \Lambda \qun 
+ ae^{-y} \px \qun - \alpha e^{-y} \Lambda \qde + a e^{-y} \px \qde,\\
\FTF_B &= \beta \mu_1 y e^{-y} \Lambda \qun 
+ b_1 \mu_1 y e^{-y} \px \qun  - \beta \mu_2 y e^{-y} \Lambda \qde + b_2 \mu_2 y e^{-y} \qde\\
\FTF_D & = \delta \mu_1   e^{-y} \Lambda \qun + \alpha \mu_1 e^{-y} \Lambda^2 \qun  + d_1 \mu_1   e^{-y} \px \qun+ a \mu_1 e^{-y}  \px \Lambda \qun 
 \\
& - \delta \mu_2  e^{-y} \Lambda \qde  - \alpha \mu_2 e^{-y} \Lambda^2 \qde + d_2 \mu_2   e^{-y} \px \qde + a \mu_2 e^{-y}  \px \Lambda \qde
			 .
\end{align*}
\end{lemma}
\begin{proof}
Recall   the expressions of ${\cal M}_j$ and ${\cal N}_j$ in \eqref{eq:p16}.
We expand $$e^{-y} \Lambda \qtun = e^{-y}\Lambda \qun + \mu_1 e^{-y}  \Lambda^2 \qun + \OO_2$$
$$e^{-y} \px \qtun = e^{-y}\px \qun + \mu_1 e^{-y}  \px \Lambda \qun + \OO_2$$
and similarly for $\qtde$, using $|\mu_1 |\leq 2 \mu_0 \leq C e^{-\frac 12 Y_0}$. Thus,
\begin{align*}
	\FTF					& =  {\cal M}_1 \Lambda \qun + {\cal M}_2 \Lambda \qde + {\cal N}_1 \px \qun + {\cal N}_2 \px \qde\\
			&	+ {\cal M}_1 \mu_1  \Lambda^2 \qun +{\cal M}_2  \mu_2   \Lambda^2 \qun
			+ {\cal N}_1  \mu_1 \px \Lambda \qun +  {\cal N}_2  \mu_2  \px \Lambda \qde +\OO_2.
\end{align*}
Some terms in the last line are $\OO_2$, using again $|\mu_1 |\leq 2 \mu_0 \leq C e^{-\frac 12 Y_0}$, so that
\begin{align*}
	\FTF		
			& =  {\cal M}_1 \Lambda \qun + {\cal M}_2 \Lambda \qde + {\cal N}_1 \px \qun + {\cal N}_2 \px \qde\\
			&	+ \alpha \mu_1 e^{-y} \Lambda^2 \qun - \alpha \mu_2 e^{-y} \Lambda^2 \qde
			+ a \mu_1 e^{-y}  \px \Lambda \qun + a \mu_2 e^{-y}  \px \Lambda \qde +\OO_2,
\end{align*}
and the expressions of $\FTF_A$, $\FTF_B$ and $\FTF_D$ follow from expanding $\mathcal{M}_j$ and $\mathcal{N}_j$ in the first line.
\end{proof}
\subsection{Determination of $A_1$, $A_2$}

\begin{lemma}[Definition and equation of $w_A$]\label{LE:23}
Let
$$
	\alpha= {12\, (10)^{2/3}}\frac 1{\int Q^2}
	,\quad \theta_A= (10)^{2/3}\frac { \int Q}{\int Q^2}.
$$
\begin{itemize}
\item[{\rm (i)}] There exist $a$ and $\hat A_1 \in \YY$ such that
$A_1= \hat A_1 + \theta_A \frac {Q'} Q$ solves
$$
	(-L A_1)' + 4 \theta_A (Q^3)' + \alpha \Lambda Q + a Q' = -4 (10)^{1/3} \left( e^{-x} Q^3\right)',
$$
$$
	\int A_1 Q'=\int (A_1+\theta_A) Q = 0.
$$
\item[{\rm (ii)}]
Set $A_2(x)=A_1(-x)$ and 
$$
	w_A(t,x)= e^{-y(t)} \left(A_1(x-y_1(t)) + A_2(x-y_2(t)\right).
$$
Then, 
\begin{equation*}
	\FF_A + \FTF_A + \GGr(w_A) = - 2 (10)^{-2/3} \theta_A (\mu_1-\mu_2) \qun \qde + e^{-y}\big[ \mu_1 \SSr_1(x-y_1) +\mu_2 \SSr_2(x-y_2)\big]+\OO_2,\end{equation*}
	where $\SSr_1 \in \YY$, $\SSr_2(x)=-\SSr_1(-x)$.

Moreover,
$$
\left|\int w_A \qun \right|+\left|\int w_A \px \qun \right|+ \left|\int w_A \qde\right|+ \left|\int w_A \px \qde \right|=\OO_2.
$$
\end{itemize}
\end{lemma}

\begin{proof} Proof of (i). First, we determine the unique possible value of $\alpha$.
Indeed, assume that $A_1$ satisfies the above equation, multiplying by $Q$, integrating and using $LQ'=0$ and parity properties, we find (using \eqref{eq:A13} and \eqref{eq:A13b})
$$
\alpha \int Q \Lambda Q = \alpha \,\frac 16 \int Q^2 = -4 (10)^{1/3} \int (e^{-x} Q^3)'Q
= (10)^{1/3} \int e^{-x} Q^4 = 2 (10)^{2/3} .
$$

Second, we determine the unique possible value of $\theta_A$. Let $A_1=\hat A_1 + \theta_A \frac {Q'}Q$, using \eqref{eq:A11},
we find for $\hat A_1$:
$$
	(-L\hat A_1)' - \theta_A \left(-\frac {36}5 Q^3 + \frac {99}{25} Q^6 \right)
	+4 \theta_A (Q^3)' + \alpha \Lambda Q + aQ' = -4 (10)^{1/3} \left(e^{-x} Q^3\right)'.
$$
To find $\hat A_1$ in $\YY$, which implies $L\hat A_1 \in \YY$, we need, using \eqref{eq:A12}, \eqref{eq:A13}
and \eqref{eq:A13b},
$$
	\theta_A \int \left( \frac {36}5 Q^3 - \frac {99}{25} Q^6 \right) +\alpha \int \Lambda Q =
	2 \theta_A - \alpha \,\frac  16 \int Q=0.
$$

Finally, we prove the existence of $\hat A_1$ and $a$.
Let $Z\in \YY$, $\int ZQ'=0$ be such that 
$$Z'= \theta_A \left(-\frac {36}5 Q^3 + \frac {99}{25} Q^6 \right)
	-4 \theta_A (Q^3)' - \alpha \Lambda Q   -4 (10)^{1/3} \left(e^{-x} Q^3\right)'.$$
Then, it suffices to solve $-L \hat A_1 + a Q = Z.$ 
By Claim \ref{CL:A1}, there exists a unique $A\in \YY$, $\int AQ'=0$ such that $-L A= Z$.
Thus, we set $\hat A_1 = A - a \Lambda Q$, which from \eqref{eq:A9} solves the equation for all $a$.
Finally, since $\int Q \Lambda Q\neq 0$ (see \eqref{eq:A13}), we uniquely fix $a$ so that $\int (\hat A_1+\theta_A) Q=0$. It is now straightforward to check that $A_1= \hat A_1 + \theta_A \frac {Q'}Q$ 
satisfies (i).

\medskip

Proof of (ii). First, by the parity properties of $Q$, $A_2(x)=A_1(-x)$ satisfies
$$
		(-L A_2)' + 4 \theta_A (Q^3)' - \alpha \Lambda Q + a Q' = -4 (10)^{1/3} \left( e^{x} Q^3\right)'.
$$
Now, we compute $\FF_A+\FTF_A+G(w_A)$. Using Claim \ref{CL:21}, we have
\begin{align*}
	G(w_A) & = \px\left( \px^2 w_A - w_A + 4\left(\qun^3  + \qde^3\right) w_A\right) \\
	& + 12 \px \left(\left( \mu_1 \qun^2 \Lambda \qun + \mu_2 \qde^2 \Lambda \qde\right) w_A \right) + \mu_1 \frac {\partial w_A}{\partial y_1} + \mu_2 \frac {\partial w_A}{\partial y_2} + \OO_2.
\end{align*}

First, 
\begin{align*}
	& \px\left( \px^2 w_A - w_A + 4 \left(\qun^3  + \qde^3\right) w_A\right)\\
	& = e^{-y} \left(- L A_1 + 4 \theta_A Q^3 \right)'(x-y_1) 
	+ e^{-y} \px\left(4 \qun^3 (A_2(x-y_2)-\theta_A) \right) \\
	& + e^{-y} \left(- L A_2 + 4 \theta_A Q^3 \right)'(x-y_2)
	+ e^{-y} \px\left(4 \qde^3 (A_1(x-y_1)-\theta_A) \right).
\end{align*}
Using the estimate 
\begin{equation}\label{eq:29}
	|A_2(x-y_2)-\theta_A|\leq C (1+|x-y_2|^\omega) e^{-(x-y_2)}
\quad \text{for $x>y_2$}
\end{equation}
and Claim \ref{CL:21}, we have
$$
	e^{-y} \qun^3(A_2(x-y_2)-\theta_A) = \OO_2 \quad \text{and similarly} \quad
	e^{-y} \qde^3(A_1(x-y_1)-\theta_A) = \OO_2.
$$
Thus, using the expressions of $\FF_A$ and $\FTF_A$ in Lemmas \ref{LE:21} and \ref{LE:22}
and the equations of $A_1$ and $A_2$, we find
$$
	\FF_A+\FTF_A+\px\left( \px^2 w_A - w_A + 4 \left(\qun^3  + \qde^3\right) w_A\right)=\OO_2.
$$

Second, by similar arguments,
\begin{align*}
	  12 \px \left(\left( \mu_1 \qun^2 \Lambda \qun + \mu_2 \qde^2 \Lambda \qde\right) w_A \right) 
	 &= 12 \mu_1 e^{-y} \px \left(   \qun^2 \Lambda \qun (A_1(x-y_1)+\theta_A)\right)\\ &
	+12  \mu_2 e^{-y}  \px \left(  \qde^2 \Lambda \qde (A_2(x-y_2) + \theta_A) \right)  +\OO_2.
\end{align*}

Finally, we compute $\mu_1 \frac {\partial w_A}{\partial y_1} + \mu_2 \frac {\partial w_A}{\partial y_2}$.
We have
$$
	\frac {\partial w_A}{\partial y_1} = -w_A - e^{-y} A_1'(x-y_1),\quad 
	\frac {\partial w_A}{\partial y_2} = w_A - e^{-y} A_2'(x-y_2).
$$
Thus, using \eqref{eq:p11},
\begin{align*}
	 \mu_1 \frac {\partial w_A}{\partial y_1} + \mu_2 \frac {\partial w_A}{\partial y_2}
	 & = -(\mu_1-\mu_2)  w_A - \mu_1 e^{-y} A_1'(x-y_1)  - \mu_2 e^{-y} A_2'(x-y_2)\\
	& = - \theta_A (\mu_1 - \mu_2) e^{-y} \left(\frac {\px \qun}\qun -\frac {\px \qde} \qde\right)
	\\ & - \mu_1 e^{-y} (2\hat A_1+A_1')(x-y_1)  - \mu_2 e^{-y} (-2 \hat A_2+ A_2')(x-y_2) +\OO_2.
\end{align*}
For this term,  we use the following claim.

\begin{claim}\label{CL:22} Let $P_1(x)=P(-x)$, $P_2(x)=P(x)$. Then,
	$$
		e^{-y} \left(\frac {\px \qun}\qun -\frac {\px \qde} \qde\right) =
		2 (10)^{-2/3} \qun \qde - e^{-y} P_1(x-y_1) - e^{-y} P_2(x-y_2) + \OO_2.
	$$
\end{claim}
Indeed, in the region  $x-y_2>\frac y2$, from \eqref{eq:A10} and the definition of $P$ (see \eqref{eq:A18}), we have
$$e^{-y} \left(\frac {\px \qun}{\qun} + P_1(x-y_1) - \frac {\px \qde}{\qde} + P_2(x-y_2) \right) 
= 2 (10)^{-1/3} e^{-y} e^{-(x-y_1)} \qun + \OO_2,
$$
and from \eqref{eq:A16}, in the same region,
$$
2 (10)^{-2/3} \qun\qde = 2 (10)^{-1/3} e^{-y} e^{-(x-y_1)} \qun + \OO_2.
$$
In the complementary region $x-y_1<-\frac y2$, we argue similarly.

\medskip

Using Claim \ref{CL:22}, we obtain
\begin{align*}
	& \mu_1 \frac {\partial w_A}{\partial y_1} + \mu_2 \frac {\partial w_A}{\partial y_2}
	  = - \theta_A (\mu_1 - \mu_2) \qun\qde
	\\ & - \mu_1 e^{-y} (-2 \theta_A P_1+ 2\hat A_1+A_1')(x-y_1)  - \mu_2 e^{-y} (2 \theta_A P_2 - 2 \hat A_2+ A_2')(x-y_2)+\OO_2.
\end{align*}
Combining these computations, we obtain 
\begin{align*}
	&\FF_A + \FTF_A + \GGr(w_A)  = -\theta_A (\mu_1-\mu_2) \qun \qde\\
	& + \mu_1 e^{-y} \left( 12 \px \left( Q^2 \Lambda Q (A_1+\theta_A)\right) + 2 \theta_A P_1 -( 2 \hat A_1 + \px A_1)\right)(x-y_1)\\
	& + \mu_2 e^{-y} \left( 12 \px \left( Q ^2 \Lambda Q (A_2+\theta_A)\right) - 2 \theta_A P_2 + (2\hat A_2- \px A_2)\right)(x-y_2),
\end{align*}
so that 
$$
S_1=12 (Q^2\Lambda Q(A_1+\theta_A)) + 2\theta_A P (-x) - 2 \hat A_1 - A_1'.
$$
Using \eqref{eq:29}, we have
$$
\int w_A \qun = e^{-y} \int \left[ A_1 Q + \theta_A Q\right] + \OO_2 = \OO_2,
$$
and similarly for the other scalar products.
\end{proof}

\subsection{Nonlocalized $\OO_{3/2}$ term}
\begin{lemma}\label{LE:23b}
	Let
	$$w_Q = - 2(10)^{-2/3} \theta_A (\mu_1-\mu_2) x\qun \qde.$$
	Then
	\begin{align*}
	  \GGr(w_Q) & = 
	2(10)^{-2/3} \theta_A (\mu_1-\mu_2) \qun\qde \\
	&+		 		  
		  2 (10)^{- 1/3}  \theta_A  \mu_1  y e^{-y}    e^{-(x-y_1)} \left( 3  (\qun-\px \qun -\px(\qun^4)) + \qun^4\right) \\
		  & 
		  -2 (10)^{-1/3}  \theta_A  \mu_2  y e^{-y}   e^{(x-y_2)} \left( 3 (\qde+\px \qde+\px(\qde^4))
		  +   \qde^4 \right) \\&
		  + e^{-y}\left( \mu_1 \SSt_1(x-y_1)+\mu_2 \SSt_2(x-y_2)\right) + \OO_2,
\end{align*}
where $\SSt_1\in \YY$ and $\SSt_1(x)=-\SSt_2(-x)$.
\end{lemma}
\begin{proof}
The proof is based on Claim \ref{CL:A5} in Appendix A.

First, arguing as in the proof of Lemma \ref{LE:23}, we have
$$
\GGr(w_Q)= \px\left(\px^2 w_Q - w_Q + 4 \left(\qun^3 + \qde^3\right) w_Q\right) +\OO_2.
$$
Moreover,
since 
$x = \frac 12 (x-y_1 + x - y_2) + \frac 12 (y_1+y_2)$, using \eqref{eq:p11},
we have 
$$w_Q = - \frac 12 \theta_A  (\mu_1-\mu_2) (x-y_1 + x - y_2)\qun\qde + \OO_2.$$

Therefore, using Claim \ref{CL:A5} and the asymptotics of $Q$ from \eqref{eq:A16}, we get
\begin{align*}
	 \GGr(w_Q) & 	   =   (10)^{-2/3} \theta_A (\mu_1-\mu_2)\px \big\{ - \px^2 ((x-y_1 + x - y_2)\qun\qde)  + (x-y_1 + x - y_2)\qun\qde 
	 \\ & \quad - 4 (\qun^3 + \qde^3) (x-y_1 + x - y_2)\qun\qde \big\}+\OO_2 \\
		 &=  2(10)^{-2/3} \theta_A  (\mu_1-\mu_2) \qun \qde\\
&+		 	 		  
		2  (10)^{-1/3}  \theta_A  \mu_1  y e^{-y}    e^{-(x-y_1)} \left( 3  (\qun-\px \qun -\px(\qun^4)) + \qun^4\right) \\
		  & 
		  - 2 (10)^{-1/3}  \theta_A  \mu_2  y e^{-y}   e^{(x-y_2)} \left( 3 (\qde+\px \qde+\px(\qde^4))
		  +   \qde^4 \right) \\&
		  + e^{-y}\left( \mu_1 \SSt_1(x-y_1)+\mu_2 \SSt_2(x-y_2)\right) + \OO_2, 		  
		,
\end{align*}
where
$\SSt_1$ and $\SSt_2$ satisfy the desired conditions.
\end{proof}

\subsection{Determination of $B_1$, $B_2$ and $D_1$, $D_2$}
\begin{lemma}[Definition and equation of $w_B$]\label{LE:24}
Let 
\begin{align*}
& Z(x)  =  - 2 (10)^{1/3}  \left(e^{-x} Q^3\right)' -   2 (10)^{- 1/3} \theta_A e^{-x}\left( 3 (Q-Q'- (Q^4)')+Q^4 \right),\\
	& \beta  = \frac 6 {\int Q^2}  (10)^{2/3},
	 \quad  \theta_B = \frac 32 (10)^{2/3} \frac {\int Q}{\int Q^2}>0.
	 \end{align*}
\begin{itemize}
\item[{\rm (i)}] There exist unique $b_1$ and $\hat B_1 \in \YY$ such that
$B_1= \hat B_1 + \theta_B \left(1+\frac {Q'} Q\right)$ satisfies
$$
	(-L B_1)'  + \beta \Lambda Q + b_1 Q' = Z,
\quad \int B_1 Q'=\int B_1 Q = 0.
$$
\item[{\rm (ii)}]
There exist unique $b_2$ and $\hat B_2\in \YY$ such that 
$B_2 =\hat B_2  - \theta_B \left(1+\frac {Q'} Q\right)$
satisfies
$$
	(-L B_2)'  - 8 \theta_B (Q^3)'  - \beta \Lambda Q + b_2 Q' = -Z(-x),
\quad 
	\int B_2 Q'=\int (B_2 - 2 \theta_B) Q = 0.
$$
Moreover,
\begin{equation}\label{eq:b}
	b_1 \neq b_2.
\end{equation}
\item[{\rm (iii)}]
Set
$$
	w_B(t,x)= y e^{-y(t)} \left( \mu_1 B_1(x-y_1(t)) + \mu_2 B_2(x-y_2(t))\right).
$$
Then,
\begin{align}
&	\FF_A + \FTF_A + \GGr(w_A) +\GGr(w_Q) + \FF_B+ \FTF_B + \GGr(w_B)\nonumber\\
& = e^{-y} \left[\mu_1 (\SSr_1+\SSt_1)(x-y_1) +\mu_2 (\SSr_2+\SSt_2)(x-y_2)\right] + \OO_2 ,\label{eq:298}
\end{align}
\begin{equation}\label{eq:255}
\left|\int w_B \qun \right| +\left| \int w_B \px \qun \right|+\left| \int w_B \qde\right| + \left| \int w_B \px \qde\right| = \OO_{5/2}.
\end{equation}
\end{itemize}
\end{lemma}
\begin{proof} We follow the strategy of the proof of Lemma \ref{LE:23}.
The only difference is that we now look for solutions $B_1$, $B_2$ both with  limit $0$ at $+\infty$.

Proof of (i). We find the  value of $\beta$ from the equation of $B_1$ multiplied by $Q$,
using \eqref{eq:A13b},
\begin{align*}
\beta \frac 16 \int Q^2 &= \int ZQ  = \frac 12 (10)^{1/3} \int e^{-x} Q^4
- 2 (10)^{-1/3}\theta_A  \left( \frac 32   \int e^{-x} Q^2 - \frac 75  \int e^{-x} Q^5\right)
\\
&=   (10)^{2/3}  .
\end{align*}
Next, since 
$$\int Z = - 2 (10)^{-1/3} \theta_A\left( 3 \int e^{-x} (Q-Q') - 2 \int e^{-x} Q^4\right)
=2\theta_A.$$
from \eqref{eq:A16} and \eqref{eq:A13b}, we find $\theta_B$ by integrating   the equation of $B_1$ ($2\theta_B = \int (-LB_1)'$)
\begin{equation*}
2 \theta_B   = \beta \frac 16 \int Q + \int Z = 3 (10)^{2/3}
\frac {\int Q}{\int Q^2}   .\end{equation*}
We now obtain the existence of  $\hat B_1\in \YY$ as in the proof of Lemma \ref{LE:23}, with $b_1$ uniquely chosen so that
$\int B_1 Q=0$ and $\int B_1 Q'$=0.

\medskip

Proof of (ii). We solve the equation of $B_2$ exactly in the same way. We check that the values of $\beta$ and 
$\theta_B$ are suitable to solve the problem, and we obtain  unique $\hat B_2\in \YY$ and $b_2$ so that $\int B_2 Q' =\int (B_2 - 2 \theta_B) Q=0$. 

We now check that $b_1\neq b_2$.
Let $B(x)=\hat B_1(x) - \hat B_2(-x)= B_1(x)-B_2(-x) - 2 \theta_B$. Then $B\in \YY$ and
$$
(-L B)' + 16 \theta_B (Q^3)'+ (b_1-b_2) Q'=0,\quad \int BQ = - 4 \theta_B \int Q.
$$
By integration
$$
-L B + 16 \theta_B Q^3+ (b_1-b_2) Q=0.
$$
Multiplying the equation of $B$ by $\Lambda Q$ and using $L(\Lambda Q)= - Q$ (see \eqref{eq:A9}),
we find
$$
- 4 \theta_B \int Q + 16 \theta_B \int Q^3 \Lambda Q + (b_1-b_2) \int Q \Lambda Q =0.
$$
Since $\int Q^3 \Lambda Q= \frac 5{24} \int Q$ and $\int Q \Lambda Q= \frac 16 \int Q^2$ (see \eqref{eq:A13}) we obtain finally
$$
b_1-b_2 = 4  \theta _B \frac {\int Q}{\int Q^2} \neq 0.
$$

Proof of (iii). We finish the proof of Lemma \ref{LE:24} as the one of Lemma \ref{LE:23}.
In particular, using the limits of $B_1$ and $B_2$ at $\pm \infty$, and \eqref{eq:p11},
\begin{align*}
	G(w_B) & = \mu_1 ye^{-y} (-LB_1)'(x-y_1) + \mu_2 y e^{-y} (-LB_2 - 8 \theta_B Q^3)' (x-y_2) +\OO_2.
\end{align*}
This, combined with the equations of $B_1$ and $B_2$ and Lemmas \ref{LE:21}, \ref{LE:22}, \ref{LE:23} and \ref{LE:23b} proves \eqref{eq:298}. Note that $w_B$ is not in $L^2$ since it has a nonzero limit at $-\infty$.
However, it has exponential decay as $x\to +\infty$. This allows us to prove that all rest terms are indeed of the form $\OO_2$ (see notation $\OO_2$ in \eqref{eq:19}).

The control of the various scalar products is easily obtained as in Lemma \ref{LE:23} from the properties of $B_1$, $B_2$.
\end{proof}

We claim without proof the following existence result.

\begin{lemma}[Definition and equation of $w_D$]\label{LE:25}
Let
$$
	\SSr= - 4 (10)^{1/3} \left(e^{-x} Q^2 \left(\tfrac 23 Q +\tfrac 12 x Q +\tfrac 32 xQ'\right)\right)'
	- \alpha \Lambda^2 Q + a (\Lambda Q)'-\SSr_1 - \SSt_1 .
$$
\begin{itemize}
\item[{\rm (i)}] There exist unique $\delta$, $\theta_D$, $d_1$ and $\hat D_1 \in \YY$ such that
$D_1= \hat D_1 + \theta_D \left(1+\frac {Q'} Q\right)$ satisfies
$$
	(-L D_1)'  + \delta \Lambda Q + d_1 Q' = S(x),
\quad \int D_1 Q'=\int D_1 Q = 0.
$$
\item[{\rm (ii)}]
There exist unique $d_2$ and $\hat D_2\in \YY$ such that 
$D_2 =\hat D_2  - \theta_D \left(1+\frac {Q'} Q\right)$
satisfies
$$
	(-L D_2)'  - 8 \theta_D (Q^3)'  - \delta \Lambda Q + d_2 Q' = -S(-x),
\quad 
	\int D_2 Q'=\int (D_2 - 2 \theta_D) Q = 0.
$$
\item[{\rm (iii)}]
Set
$$
	w_D(t,x)= e^{-y(t)} \left( \mu_1 D_1(x-y_1(t)) + \mu_2 D_2(x-y_2(t))\right).
$$
Then,
\begin{align*}
&	\FF_A + \FTF_A + \GGr(w_A) +\GGr(w_Q) + \FF_B+ \FTF_B + \GGr(w_B) +\FF_D+ \FTF_D +\GGr(w_D)=
  \OO_2,
\end{align*}
$$
\left|\int w_D \qun \right| +\left| \int w_D \px \qun \right|+\left| \int w_D \qde\right| + \left| \int w_D \px \qde\right| = \OO_{5/2}.
$$
\end{itemize}
\end{lemma}
The proof is exactly the same as the one of  Lemma \ref{LE:24} except that we do not need the values of
$\delta$, $\theta_D$ and $d_1-d_2$. The exact expressions of $\SSr_1$ and $\SSt_1$ are thus not needed.

\subsection{End of the proof of Proposition \ref{PR:21}}
Set 
$$
V_0 = \qtun + \qtde + W_0, \quad W_0=w_A+w_Q+w_B+w_D.
$$
From the preliminary expansion \eqref{eq:20}, we have
$$
\SCS(V_0) = \ETE(V_0) + E_0, \quad E_0=\FF + \FTF + \GGr(W_0) + \HH(W_0).
$$
In view of  notation \eqref{eq:19}, estimate \eqref{eq:p17} holds true for some $\sigma>0$ provided that 
$E_0=\OO_2.$
From Lemmas \ref{LE:21}, \ref{LE:22} and \ref{LE:25}, we have $\FF + \FTF + G(W_0) =\OO_2$.
Thus, we only have to check that $H(W_0)=\OO_2$.

First,
$$
 \px \left[ \left(\qtun+\qtde+W_0\right)^4 - \left(\left(\qtun+\qtde\right)^4
 +4 \left(\qtun^3+\qtde^3\right)W_0\right)\right]=\OO_2
 $$
since this term is quadratic in $W_0$.

Second, since $|\mathcal{M}_j |+ |\mathcal{N}_j|\leq C e^{-y}$,
we also obtain 
$$ \sud {\cal M}_j \frac {\partial W_0}{\partial \mu_j}	- \sud   {\cal N}_j \frac {\partial W_0}{\partial y_j}
 =\OO_2.
$$
Thus, Proposition \ref{PR:21} is proved.

\subsection{Proof of Proposition \ref{PR:22}}
Denote $\tilde x = e^{-\frac 12 Y_0} x+1$.
Multiplying \eqref{eq:p14} by $\psi(\tilde x)$ and using $\ETE(V_0) \psi(\tilde x)= \ETE(V)$, we observe that $V(t,x)$ solves
\begin{equation}\label{eq:VVA}
	 \partial_t V + \partial_x (\partial_x^2 V - V + V^4) = \ETE(V) +E(t,x) ,
\end{equation}
where 
\begin{align*}
E(t,x) & = E_0(t,x) \psi(\tilde x)  
 + e^{-\frac 32 Y_0} \psi'''(\tilde x) V_0 + 3e^{-Y_0} \psi''(\tilde x) \px V_0 
+ 3 e^{-\frac 12 Y_0} \psi'(\tilde x) \px^2 V_0 \\
& - e^{-\frac 12 Y_0} \psi'(\tilde x) V_0 +\psi(\tilde x)(\psi^3(\tilde x)-1) \px (V_0^4) 
+ 4 \psi^3(\tilde x)\psi'(\tilde x) V_0^4 .
\end{align*}

Note that 
\begin{equation}\label{eq:pssii}
	\|\psi(\tilde x)/(1+e^{\frac 12 ( x-y_1(t))})\|_{L^2} + \|\psi'(\tilde x)\|_{L^2}+
	\|\psi''(\tilde x)\|_{L^2}+\|\psi'''(\tilde x)\|_{L^2}
	\leq  C   e^{ \frac 14 Y_0}.
\end{equation}
Moreover, by the properties of $\psi$, and $|y_j(t)|\leq K Y_0$ (combine \eqref{eq:p10}--\eqref{eq:p11}), we have
\begin{equation}\label{eq:pssi} 
 \left(|\psi'(\tilde x)|+|\psi''(\tilde x)|+|\psi'''(\tilde x)|+|1-\psi(\tilde x)|\right) \left(e^{-\frac 12 |x-y_1|}+ e^{-\frac 12 |x-y_2|}\right) 
	\leq C \exp(-\tfrac 14 e^{\frac 12 Y_0}).
	\end{equation}

First, using \eqref{eq:p17} and \eqref{eq:pssii},
\begin{equation*}
	\|E_0(t,x) \psi(\tilde x)   \|_{L^2} \leq C Y_0^{\sigma} e^{-Y_0} e^{-y(t)} \|\psi(\tilde x)/(1+e^{\frac 12 (x-y_1(t))})\|_{L^2}
	\leq  C Y_0^{\sigma} e^{-\frac 34 Y_0} e^{-y(t)}.
\end{equation*}

Second, by the structure of $V_0$, we easily check that
\begin{equation}\label{eq:surV0} 
 |V_0|+|\px V_0|+|\px^2 V_0|\leq C \left(e^{-\frac 12 |x-y_1|}+ e^{-\frac 12 |x-y_2|}\right) + C Y_0 e^{-\frac 12 Y_0} e^{-y(t)},
\end{equation}
and all the other terms in $E$ are controled in $L^2$ as desired using \eqref{eq:surV0} combined with \eqref{eq:pssi}. The estimate in $H^1$ is obtained similarly.

\section{Preliminary stability arguments}

\subsection{Stability of the $2$-soliton structure in the interaction region}

We start by decomposing any solution of \eqref{eq:KDV} close the approximate solution $V$ (introduced in Proposition \ref{PR:22}). 
See Appendix \ref{AP:B} for the proof.

\begin{lemma}[Decomposition around the approximate solution]\label{PR:de}
There exists $\omega_0>0$, $C>0$, $\bar y_0>0$ such that if
$u(t)$ is a solution of {\eqref{eq:KDV}} on some time interval $I$ satisfying
for $0<\omega<\omega_0$, $y_0>\bar y_0$
\begin{equation}\label{close}
	\forall t\in I,\quad 
	\inf_{y_1-y_2>y_0} \|u(t) - V(.;(0,0,y_1,y_2))\|_{H^1}\leq\omega,
\end{equation}
then there exists a unique decomposition $(\Gamma(t),\varepsilon(t))$ of  $u(t)$ on $I$,
\begin{equation}
	u(t,x)=V(x;\Gamma(t)) + \varepsilon(t,x), \quad
	\Gamma(t)=(\mu_1(t),\mu_2(t),y_1(t),y_2(t)) \text{ of class $C^1$},
\end{equation}
such that $\forall t\in I$,
\begin{equation}\label{eq:or}
\begin{split}
	&	\int \varepsilon(t) \qtun(t)  = \int \varepsilon(t)\px \qtun(t)= \int \varepsilon(t) \qtde(t)  = \int \varepsilon(t) \px\qtde(t)=0,\\
	& y(t)=y_1(t)-y_2(t)> y_0- C\omega, \quad \|\varepsilon(t)\|_{H^1}+|\mu_1(t)|+|\mu_2(t)|\leq C\omega ,
\end{split}
\end{equation}
\begin{equation}\label{eq:ep}\begin{split}
	&  \partial_t \varepsilon 
	+ \partial_x\left(\partial_x^2 \varepsilon - \varepsilon
	+ \left(V + \varepsilon \right)^4 - V^4 \right) + \ETE(V) + E(t,x)   = 0,
\end{split}\end{equation}
where
$$
	\qtun(t,x) = Q_{1+\mu_1(t)}(x-y_1(t)),\quad
	\qtde(t,x) = Q_{1+\mu_2(t)}(x-y_2(t)),
$$
and   $V$, $\ETE(V)$ and  $E(t,x)$  are defined in Proposition \ref{PR:22}.

Moreover, assuming
\begin{equation}\label{eq:gga}
\forall t\in I,\quad 	(|\mu_1(t)|+|\mu_2(t)|) y(t) \leq 1,
\end{equation}
 $\dot \Gamma(t)$ satisfies the following estimates
\begin{equation}\label{eq:ga}
\begin{split} 
& |  \dot \mu_j - {\cal M}_j   | \leq C  \left(\|\varepsilon\|_{L^2}^2 + y e^{-y} \|\varepsilon\|_{L^2}\right) + C \int |E| \left(\qtun+ \qtde\right)  , \\
&|\mu_j - \dot y_j - {\cal N}_j | \leq 
C  \|\varepsilon\|_{L^2}+ C \int |E| \left(\qtun+\qtde\right) .
\end{split}\end{equation}
\end{lemma}

In the next proposition, we present almost monotonicity laws which are essential in proving long time stability  results in the interaction region. They will allow us to compare the approximate solution $V(t,x)$ with exact solutions. The functional is different depending on whether $\mu_1(t)>\mu_2(t)$ or $\mu_1(t)<\mu_2(t)$. 
The introduction  of such variants of the energy  and mass is related to Weinstein's approach for stability of one soliton \cite{We2} and to   Kato identity for the \eqref{eq:KDV} equation (see \cite{Kato2}). These techniques have been developed in \cite{MM1}, \cite{MMT}, \cite{Ma2}  and   extended in \cite{MMas2} and have had decisive applications to long time issues for \eqref{eq:gKdV} : blow up in the $L^2$ critical case, stability and asymptotic stability of multi-solitons in the energy space in both integrable and nonintegrable cases.
 
\medskip
 
Fix a constant $0<\rho<1/32$ and set
\begin{equation}\label{eq:ph}
\begin{split}
	&\varphi(x)=\frac 2 \pi \arctan(\exp(8 \rho x)), \quad \text{so that }
        \lim_{-\infty} \varphi=0,\ \lim_{\infty} \varphi=1,\\
        & \forall x\in \mathbb{R},\quad \varphi(-x)=1-\varphi(x),\quad
 \varphi'(x)=\frac {8\rho}{  \pi\cosh( 8 \rho x)},\\
& |\varphi''(x)|\leq 8 \rho |\varphi'(x)|, \quad  |\varphi'''(x)|\leq (8 \rho)^2 |\varphi'(x)|.
\end{split}\end{equation}

\begin{proposition}[Almost monotonicity laws]\label{PR:cFG} Under the assumptions of Lemma \ref{PR:de},
let
\begin{equation}\label{eq:en}
\begin{split}
	& {\cal F}_+(t) =
		\int \left[\left((\partial_x \varepsilon)^2 + \varepsilon^2 
		- \tfrac 25 \left((\varepsilon+V)^5 - V^5 - 5 V^4\varepsilon \right)\right)
		+    \varepsilon^2  \Phi(t,x)\right]dx,
		\end{split}\end{equation} 
where $\Phi(t,x)=\mu_1(t) \varphi(x)+ \mu_2(t) (1-\varphi(x))$;
\begin{equation}\label{eq:eG}
\begin{split}
	& {\cal F}_-(t) =
		\int \left[\left((\partial_x \varepsilon)^2 + \varepsilon^2 -\tfrac 25 \left((\varepsilon+V)^5 - V^5 - 5 V^4\varepsilon \right)\right)\Phi_1(t,x)+  \varepsilon^2\Phi_2(t,x)\right]dx,
\end{split}\end{equation}
where
$$\Phi_1(t,x)=\frac {\varphi(x)}{(1+\mu_1(t))^2} + \frac {1-\varphi(x)}{(1+\mu_2(t))^2},\quad\Phi_2(t,x)=\frac {\mu_1(t)\varphi(x)}{(1+\mu_1(t))^2} + \frac {\mu_2(t)(1-\varphi(x))}{(1+\mu_2(t))^2}.$$
There exists $C>0$ such that
\begin{equation}\label{eq:FGcoer}
\|\varepsilon(t)\|_{H^1}^2 \leq C {\cal F}_+(t),\qquad \|\varepsilon(t)\|_{H^1}^2 \leq C {\cal F}_-(t).
\end{equation}
Moreover,

\noindent{\rm (i)}  If $t\in I$ is such that 
$\mu_1(t)\geq \mu_2(t)$ and $y_2(t)\leq - \frac 14 y(t)$,
	$y_1(t)\geq \frac 14 y(t)$,
then
\begin{equation}\label{eq:cF}
\begin{split}
\frac d{dt}{\cal F}_+(t) & \leq  C \|\varepsilon\|_{L^2}^2 
\left[e^{- \frac 34  y}   + (|\mu_1|+|\mu_2|+\|\varepsilon\|_{L^2} )( e^{-2 \rho y} +  \|\varepsilon\|_{L^2} )  \right]
+ C \|\varepsilon\|_{H^1} \|E\|_{H^1}.
\end{split}\end{equation}
 {\rm (ii)}
	If $t\in I$ is such that 
$
		\mu_2(t)\geq \mu_1(t) $ and $y_2(t)\leq -\frac 14 y(t),$
		$y_1(t)\geq \frac 14 y(t),$
	then
	\begin{equation}\label{eq:cG}
	\begin{split}
	\frac d{dt}{\cal F}_-(t) & \leq  C \|\varepsilon\|_{L^2}^2 
	\left[  e^{- \frac 34  y} +(|\mu_1|+|\mu_2|+\|\varepsilon\|_{L^2} ) (e^{-2 \rho y}       
	+  \|\varepsilon\|_{L^2}) \right]
	+ C \|\varepsilon\|_{H^1} \|E\|_{H^1}.
	\end{split}\end{equation}
\end{proposition}

See proof of Proposition \ref{PR:cFG} in Appendix \ref{AP:B}.

\subsection{Stability of the two soliton structure for large time}

In this section, we claim a stability result for the two soliton structure for large time, i.e. far away from the interaction time. The argument is similar to the one of Proposition \ref{PR:cFG}.
See a sketch of proof in Appendix B.

\begin{proposition}[Stability for large time]\label{PR:STAB}
There exists $C>0$ such that for
$ \mu_0>0$ and $\omega>0$ small enough, if $u(t)$ is an $H^1$ solution of {\eqref{eq:KDV}} satisfying
\begin{equation}\label{eq:clo}
\	\left\|u(t_0) - Q_{1-\mu_0}(. + \mu_0 t_0) - Q_{1+\mu_0}(.-\mu_0 t_0)\right\|_{H^1(\mathbb{R})} \leq \omega \mu_0,
 \end{equation}
 for some $t_0<- (\rho\mu_0)^{-1} |\log \mu_0|$, then 
there exist $y_1(t)$, $y_2(t)$ and $\mu_1^+$, $\mu_2^+$ such that
\begin{enumerate}
 \item[\rm (i)]  For all $t_0 \leq  t \leq - (\rho\mu_0)^{-1} |\log \mu_0|$,
\begin{equation}\label{eq:stab+}\begin{split}
	& \left\|u(t)- Q_{1-\mu_0}(. - y_1( t ) ) - Q_{1+\mu_0}(.- y_2( t ) )\right\|_{H^1(\mathbb{R})}
	\leq   C \omega\mu_0 +  C \exp\left({-4 \rho { \mu_0  }   |t|}\right),\\
	& y_1(t)-y_2(t)\geq \frac 32 \mu_0 |t|, \\
	& |-\mu_0 - \dot y_1(t)| + |\mu_0 - \dot y_2(t)|\leq C  \omega \mu_0+  C \exp\left({-4 \rho { \mu_0  }   |t|}\right).
\end{split} \end{equation}
 \item[\rm (ii)]  For all $t\leq t_0$,
\begin{equation}\label{eq:stab-}\begin{split}
	& \left\|u(t)- Q_{1-\mu_0}(. - y_1( t ) ) - Q_{1+\mu_0}(. - y_2( t ) )\right\|_{H^1(\mathbb{R})}
	\leq   C \omega\mu_0 +  C \exp\left({-4 \rho { \mu_0  }   |t_0|}\right),
	\\
	& y_1(t)-y_2(t)\geq \frac 32 \mu_0 |t|, \\
	& |-\mu_0 - \dot y_1(t)| + |\mu_0 - \dot y_2(t)|\leq C  \omega \mu_0+  C \exp\left({-4 \rho { \mu_0  }   |t_0|}\right).
\end{split} \end{equation}
 \item[\rm (iii)]  Asymptotic stability.
 \begin{equation}\label{eq:as2}\begin{split}
&	\lim_{t\to -\infty} \left\|u(t)- Q_{1+\mu_1^+}(. - y_1(t)) - Q_{1+\mu_2^+}(.- y_2(t) )\right\|_{H^1(x<\frac{99}{100}|t|)}=0,\\
&	\lim_{t\to -\infty} \dot y_1(t) = \mu_1^{+}, \quad \lim_{t\to -\infty} \dot y_2(t)=\mu_2^{+},\\
&	|\mu_1^+ + \mu_0|+|\mu_2^+ -\mu_0|\leq C \omega\mu_0 +  C \exp\left({- 4 \rho { \mu_0  }|t_0|}\right).
	\end{split}
 \end{equation}
 \end{enumerate}
\end{proposition}

\begin{remark}\label{RE:STAB}
Using the invariance of the gKdV equation by the transformation
\begin{equation}\label{eq:inv}
x\to -x,\quad t\to -t,
\end{equation}
a statement similar to Proposition \ref{PR:STAB} holds for $t_0>(\rho\mu_0)^{-1} |\log \mu_0|$.
\end{remark}

\begin{corollary}\label{COR:pur}
Let $u(t)$ be the  unique solution of {\eqref{eq:KDV}} satisfying 
$$
\lim_{t\to -\infty}
\left\|u(t)- Q_{1-\mu_0}( . + \mu_0t  ) - Q_{1+\mu_0}( . - \mu_0t  )\right\|_{H^1} = 0.
$$
Then, for all $	 t\leq  - (\rho\mu_0)^{-1} |\log \mu_0|, $
\begin{equation}\label{eq:pp}
 \left\|u(t)- Q_{1-\mu_0}( . + \mu_0t  ) - Q_{1+\mu_0}( . - \mu_0t  )\right\|_{H^1}
	\leq    \exp\left({- 4 \rho \mu_0 |t| }\right). 
 \end{equation}
\end{corollary}
We refer to Theorem 1 in \cite{Ma2} for the existence and uniqueness of the solution $U(t)$.

\begin{proof}[Proof  of Corollary \ref{COR:pur} assuming Proposition \ref{PR:STAB}]
For fixed $t$, we  can pass to the limit $\omega\to 0$, $t_0\to -\infty$ in \eqref{eq:stab+}.
Then, we integrate the estimates on $\dot y_1(t)$ and $\dot y_2(t)$ (see \eqref{eq:stab+}) from $-\infty$ to $t$.
\end{proof}

\section{Stability of  the $2$-soliton structure}

In this section, using the approximate solution constructed in Propositions \ref{PR:21}
and \ref{PR:22} and 
the asymptotic arguments of Section 3, we prove the stability part of Theorem \ref{TH:1} and Theorem~\ref{TH:2}.
These properties are much more refined than the asymptotic results obtained in \cite{MMT} ; they rely on the approximate solution constructed in Section 2 and on a reduction to a finite dimensional dynamical system. In particular, they cannot be derived from \cite{MMT}.

\subsection{Description of the global behavior of the asymptotic $2$-soliton solution}

Recall that $0<\rho<1/16$, $\alpha>0$  and $\sigma\geq 3$ are defined respectively in  Propositions \ref{PR:cFG}, \eqref{eq:p12} and  Proposition \ref{PR:21}. We also recall the following notation from the Introduction\begin{equation}\label{eq:defH}
\YYzz=|\ln(\mu_0^2/ \alpha)|\quad \text{or equivalently}\quad 
\mu_0 = \sqrt{\alpha} e^{-\frac 12{\YYzz}},\end{equation}
\begin{equation}\label{eq:HH}
\text{$\YYrr(t)= \YYzz + 2 \ln(\cosh ( \mu_0 t))$    solution of 
	$\ddot \YYrr= 2 \alpha e^{-\YYrr}, \ \lim_{t\to -\infty} \dot \YYrr(t) =- 2 \mu_0$,  $\dot Y(0)=0$.}
\end{equation}
Note that $\dot Y(t) = 2 \mu_0 \tanh(\mu_0 t)$ and,  for all $t\in \RR$,
\begin{equation}\label{eq:eH}
	0 \leq  \YYrr( t ) - \left(\YYzz + 2 \mu_0 |t|  - 2 \ln 2 \right)\leq 2 \exp\left(-2 \mu_0  |t|\right).
\end{equation}

\begin{proposition}[Description of the $2$-soliton solution in the interaction region]\label{pr:st} \ \\
Let  $U(t)$ be the  unique solution of {\eqref{eq:KDV}} such that
\begin{equation}\label{eq:pur2}
	\lim_{t\to -\infty} 
	\left\|U(t) - Q_{1-\mu_0}(. + \tfrac 12 \YYrr(t) ) - Q_{1+\mu_0}(. - \tfrac 12 \YYrr(t) )\right\|_{H^1(\mathbb{R})} =0.
\end{equation}
Let $\TSR>0$ be such that $\YYrr(\TSR)= 400 \rho^{-2} \YYzz$.
Then, for $\mu_0>0$ small enough, there exists   $(\Gamma(t),\varepsilon(t)) \in C^1$ such that for all $t\in [-\TSR,\TSR],$
$$
	U(t,x)= V(t;\Gamma(t)) + \varepsilon(t,x),\quad \Gamma(t)=(\mu_1(t),\mu_2(t),y_1(t),y_2(t)),
$$
\begin{align}\label{eq:J11}
&	|\bar \mu(t)|\leq   \YYzz^2 e^{- \YYzz}, \quad	|\bar y(t)| \leq   \YYzz^4 e^{-\frac 12 \YYzz},
\\
\label{eq:U+-}  
&  	|\mu(t)-\dot \YYrr(t)|\leq C \YYzz^{\sigma+1} e^{- \frac 54  \YYzz},\quad
	|\pyy(t)-\YYrr(t)|\leq C \YYzz^{\sigma+2} e^{- \frac 34  \YYzz}, 
\\	\label{eq:U++}
& \|\varepsilon(t)\|_{H^1} \leq C \YYzz^{\sigma} e^{-\frac 54 \YYzz},
	\end{align}
where 	
	\begin{equation}\label{eq:mth}\begin{split}&
	\mu(t)=\mu_1(t)-\mu_2(t), \quad \pyy(t)=y_1(t)-y_2(t),\\ & \summu(t)=\mu_1(t)+\mu_2(t),\quad
	\pyyb(t)=y_1(t)+y_2(t).
	\end{split}\end{equation}
Moreover, there exists $t_0$ such that
\begin{equation}\label{eq:t0}
	|t_0|\leq C \YYzz^{\sigma}e^{- \frac 14 \YYzz}, \quad \text{$\mu(t_0)=0;$ \ $\forall t\in [-\TSR,t_0)$,\  $\mu(t)<0;$ \ $\forall t\in (t_0,\TSR]$,\  $\mu(t)>0$.} 
\end{equation}
\end{proposition}

\begin{proof}
	It follows from Corollary \ref{COR:pur} and \eqref{eq:eH} that
	for all $
	 t\leq  -  (\rho\mu_0)^{-1} |\ln \mu_0|, $
	\begin{equation}\label{eq:ppb} 	\left\|U(t)- Q_{1-\mu_0}(x + \tfrac 12 \YYrr( t ) ) - Q_{1+\mu_0}(x- \tfrac 12 \YYrr( t ) )\right\|_{H^1}
	\leq    \exp\left({- 2 \rho \YYrr(t)}\right).
 \end{equation}
	This estimate is the starting point of our analysis.

	 Let  
	$0<\epsilon <\frac 1{100}$ to be chosen later, and let 
	 $\TSR>\TPS>\TPP>0$ be defined as follows
	\begin{equation}\label{eq:lesT}
		\YYrr( \TSR ) = 400 \rho^{-2}   \YYzz,\quad 
		\YYrr( \TPS ) = 40 \rho^{-1}  \YYzz , \quad
		\YYrr( \TPP ) = \YYzz + \epsilon^2.
	\end{equation}
	Note that by the explicit expression of $\YYrr(t)$ in \eqref{eq:HH},  for $0<C<C'$ independent of $\epsilon$, we have
	\begin{equation}\begin{split}\label{eq:cH}
	 & 	C\YYzz e^{\frac 12 \YYzz} \leq \TPS < \TSR < C' \YYzz e^{\frac 12 {\YYzz} }, \\
	& 
	 C \epsilon e^{-\frac 12 \YYzz}\leq \TPP\leq C' \epsilon e^{-\frac 12 \YYzz},\quad \dot \YYrr( \TPP ) \geq C \epsilon e^{-\frac 12 \YYzz}.  
	\end{split}\end{equation}

	Applying estimate \eqref{eq:ppb} at $t=-\TSR$, and using \eqref{eq:lesT}, we find
	\begin{equation}\label{eq:st}
		\left\|U(-\TSR)- Q_{1-\mu_0}(x + \tfrac 12 \YYrr( \TSR ) ) - Q_{1+\mu_0}(x-\tfrac 12 \YYrr( \TSR ) )\right\|_{H^1}
		\leq    		C  e^{-\rho  \YYrr(\TSR)} e^{-10 \YYrr(\TPS)}. 
	\end{equation}
	For $t \geq -\TSR$, as long as $u(t)$ stays close to the sum of two solitons, we introduce the decomposition $(\Gamma(t), \varepsilon(t))$ of $U(t)$ as constructed in Lemma \ref{PR:de}. 
	At $t=-T$, from \eqref{eq:st}, we obtain
	\begin{align}
	&	\|\varepsilon(-\TSR)\|_{H^1} + |\mu_1(-\TSR)+\mu_0|+|\mu_2(-\TSR)-\mu_0|  \leq 
	C  e^{-\rho  \YYrr(\TSR)} e^{-10 \YYrr(\TPS)}   , \label{eq:mod2}\\
	&	| y_1(-\TSR)-\tfrac 12 \YYrr(\TSR) |+| y_2(-\TSR)+\tfrac 12 \YYrr(\TSR) |\leq  C  e^{-\rho  \YYrr(\TSR)} e^{-10 \YYrr(\TPS)}   . \label{eq:mod3}
	\end{align}

	Let $C_3>C_2>C_1>1$ to be chosen later, and consider the following estimates:

	\noindent   For   $t\in [-\TSR,-\TPS]$,
	\begin{equation}\label{eq:imp}\begin{split}
	&    e^{\frac 14 \YYzz} \|\varepsilon(t) \|_{H^1} + e^{-\frac 14 \YYzz} |\pyy(t)-\YYrr(t)|+
	e^{\frac 12 \YYzz} |\mu(t)-\dot \YYrr(t)|  \leq  C_1  \YYzz^{\sigma} e^{-(1-\rho)\YYrr(\TPS)}
	e^{-\rho \YYrr(t)}.
	\end{split}\end{equation}
	\noindent   For   $t\in [-\TPS,\TPP]$,
	\begin{equation}\label{eq:imp2} 
	    \|\varepsilon(t) \|_{H^1} + |\mu(t)-\dot \YYrr(t)|\leq   C_2 \YYzz^{\sigma} e^{-\frac 14 \YYzz} e^{-  \YYrr(t)},
	\quad |\pyy(t)-\YYrr(t)|\leq  C_2 \YYzz^{\sigma} e^{-\frac 34 \YYzz} .
	 \end{equation}
	\noindent  For   $t\in [\TPP,\TSR]$,
	\begin{equation}\label{eq:imp3} 
	   \YYzz^2  e^{\frac 12 \YYzz} \|\varepsilon(t) \|_{H^1}  + \YYzz e^{ \frac 12 \YYzz} |\mu(t)-\dot \YYrr(t)|  
	+|\pyy(t)-\YYrr(t)|\leq  C_3 \YYzz^{\sigma+2} e^{-\frac 34 \YYzz}.
	 \end{equation}
	From \eqref{eq:mod2}-\eqref{eq:mod3} and 
	the continuity of $t\mapsto u(t)$ in $H^1$, it follows that we can define
	$$
	T^*=\sup\left\{  \text{$t\in [-\TSR,\TSR]$ such that \eqref{eq:J11}-\eqref{eq:imp}-\eqref{eq:imp2}-\eqref{eq:imp3} hold on $[-\TSR,t]$}\right\}.
	$$
	Now, we aim at proving that $T^*=\TSR$ by strictly improving estimates \eqref{eq:J11}-\eqref{eq:imp}-\eqref{eq:imp2}-\eqref{eq:imp3} on $[-\TSR,T^*]$ for $C_1,$ $C_2$, $C_3$ large enough (independent of $Y_0$) and  $Y_0$ large enough (possibly depending on $C_1$, $C_2$, $C_3$).

\medskip	
	
	\emph{Step 1.} Preliminary simplication of the dynamical system.

 \begin{claim}[Simplified dynamical system]\label{LE:SD}
For all $t \in [-T,T^*]$,
\begin{equation}\label{eq:sd2}\begin{split}
&  |\dot \mu -  2 \alpha e^{-\pyy}|  \leq 	
	C \YYzz^{\sigma} e^{-\YYzz} e^{-\YYrr(t)}+ C \|\varepsilon(t)\|_{L^2}^2 ,\\
&  |\dot \pyy - \mu  |\leq  C |\mu| \YYzz e^{-\YYrr(t)}  +  C \YYzz^\sigma e^{- \YYzz} e^{-\YYrr(t)}+ C \|\varepsilon(t)\|_{L^2},\\
&  |\dot \summu | \leq C |\mu| \YYzz e^{-\YYrr(t)}  +C \YYzz^{\sigma} e^{-\YYzz} e^{-\YYrr(t)} + C \|\varepsilon(t)\|_{L^2}^2   ,\\
& |\dot \pyyb  | \leq | \summu | + C   e^{-\YYrr(t)}+ C \|\varepsilon(t)\|_{L^2}.
\end{split}\end{equation}
\end{claim}

\begin{proof}[Proof of Claim \ref{LE:SD}]
First, on $[-T,T^*]$, \eqref{eq:p10}-\eqref{eq:p11} hold and so by \eqref{eq:p23},
\begin{equation}\label{eq:eEb}
	 \left|\int E(t) \tilde Q_j(t)\right|\leq C \YYzz^{\sigma} e^{-\YYzz}e^{-	\YYrr(t)},\quad
	\|E(t)\|_{H^1}\leq C \YYzz^{\sigma} e^{-\frac 34 \YYzz} e^{-\YYrr(t)}.
\end{equation}
Thus, by \eqref{eq:ga}, for some constant $C=C(K)>0$,
	\begin{align}
	& |\dot \mu_j - {\cal M}_j| \leq 
	C \YYzz^{\sigma} e^{-\YYzz} e^{-\YYrr(t)}+ C \YYzz \|\varepsilon\|_{L^2} e^{- \YYrr(t)}+ C \|\varepsilon\|_{L^2}^2 ,\label{eq:gab}\\
	& |\mu_j - \dot y_j - {\cal N}_j| \leq 
	C \YYzz^{\sigma} e^{-\YYzz} e^{-\YYrr(t)}+ C \|\varepsilon\|_{L^2} .\label{eq:gabb}
	\end{align}

Second, the following estimates are obtained by combining the  definitions of ${\cal M}_j$ and ${\cal N}_j$ in \eqref{eq:p16}:
	\begin{equation}\label{eq:sd222}\begin{split}
	&  |\dot \mu -\{ 2 \alpha e^{-\pyy} + \beta \summu \pyy e^{-\pyy} + \delta \summu e^{-\pyy}\}|  \leq \sum_{j=1,2}|\dot \mu_j  - {\cal M}_j| ,\\
	&  |\dot \summu -\{ \beta \mu \pyy e^{-\pyy} + \delta \mu e^{-\pyy}  \}| \leq  \sum_{j=1,2}|\dot \mu_j  - {\cal M}_j|   ,\\
	&  |\dot \pyy -\{\mu + b_+ \mu \pyy e^{-\pyy} + b_- \summu \pyy e^{-\pyy} +  d_+ \mu e^{-\pyy} + d_- \summu e^{-\pyy}\}|\leq   \sum_{j=1,2}  |\mu_j - \dot y_j - {\cal N}_j |,\\
	& |\dot \pyyb - \{ \summu - 2a e^{-\pyy} + b_- \mu \pyy e^{-\pyy} + b_+ \summu \pyy e^{-\pyy}+ d_- \mu e^{-\pyy} + d_+ \summu  e^{-\pyy}
	\}|\leq \sum_{j=1,2}  |\mu_j - \dot y_j - {\cal N}_j |,
	\end{split}\end{equation}
	where
	\begin{equation}\label{eq:fg}
		b_+=-\frac 12 (b_1+b_2),\quad b_-=-\frac 12 (b_1-b_2),\quad
		d_+ = -\frac 12 (d_1+d_2),\quad d_-=-\frac 12 (d_1-d_2).
	\end{equation}
	Using the previous estimates, \eqref{eq:J11} and $\sigma\geq 3$, this implies \eqref{eq:sd2}.
\end{proof}

\emph{Step 2.} Bootstrap    estimates.

\emph{Bootstrap of \eqref{eq:J11}.}
From the definition of $T^*$  and \eqref{eq:sd2}, we improve \eqref{eq:J11} on $[-T,T^*]$.
Note that using \eqref{eq:imp}-\eqref{eq:imp2}-\eqref{eq:imp3}, we have for all $t\in [-T,T^*]$,
\begin{equation}\label{eq:impfac}
	\|\varepsilon(t)\|_{H^1}\leq  C_3 Y_0^\sigma e^{-\frac 54 Y_0},\quad
	 |\mu(t)-\dot Y(t) |\leq C_3 Y_0^{\sigma+1} e^{-\frac 54 Y_0},\quad
	|y(t)-Y(t)|\leq C_3 Y_0^{\sigma+2} e^{-\frac 34 Y_0}.
\end{equation}
Concerning $\summu$, by \eqref{eq:sd2} and \eqref{eq:impfac}, for $Y_0$ large enough,
$$|\dot \summu|\leq C  Y_0 e^{-\frac 12 Y_0}e^{-Y(t)} + C C_3^2 Y_0^{2\sigma} e^{-\frac 52 Y_0} .$$
Thus, by direct integration, using the expression of $Y(t)$, $T<C Y_0 e^{\frac 12 Y_0}$ (see \eqref{eq:cH}) and \eqref{eq:mod2}, we find
for all $t\in [-T,T^*]$, $|\summu(t)|\leq C Y_0 e^{- Y_0}$,
for $Y_0$ large enough, and thus
by possibly  taking a larger $Y_0$, 
\begin{equation}\label{eq:mub}
	\forall t\in [-T,T^*],\quad |\summu(t)|\leq  \frac 12 Y_0^2 e^{- Y_0}.
\end{equation}

Concerning $\pyyb$, we proceed similarly. By \eqref{eq:sd2} and \eqref{eq:impfac}, for $Y_0$ large enough,
$|\dot \pyyb|\leq  Y_0^2 e^{- Y_0} + C e^{-Y_0}$, and by direct integration,  $T<C Y_0 e^{\frac 12 Y_0}$ and \eqref{eq:mod2}, we find
for all $t\in [-T,T^*]$, $|\pyyb(t)|\leq C Y_0^3 e^{- \frac 12 Y_0}$.
Finally,  taking $Y_0$ large enough, we find
\begin{equation}\label{eq:yb}
	\forall t\in [-T,T^*],\quad |\pyyb(t)|\leq  \frac 12 Y_0^4 e^{- \frac 12 Y_0}.
\end{equation}

It follows that
$$
T^*=\sup \left\{ t\in [-T,T] \text{ such that  \eqref{eq:imp}-\eqref{eq:imp2}-\eqref{eq:imp3} hold on $[-T,t]$}\right\}.
$$

\medskip

\emph{Bootstrap of \eqref{eq:imp} on $[-\TSR,-\TPS]$.}
Now, we prove that   $T^*>-T'$ for $C_1>1$  large enough. 

First, we estimate $\varepsilon(t)$ on $[-\TSR,T^*]$.
We use the functional ${\cal F}_-(t)$ defined  in \eqref{eq:eG} since $\mu_2\geq \mu_1$ on $[-T,T']$.
Using \eqref{eq:cF}, \eqref{eq:eEb} and \eqref{eq:imp},  we have
\begin{equation*}\begin{split}
\frac d{dt}{\cal F}_-(t) & \leq C C_1 \YYzz^{2 \sigma  } e^{- \YYzz} e^{- (1-\rho) \YYrr( \TPS )} e^{-(1+\rho) \YYrr(t)}+
C C_1^2 \YYzz^{2 \sigma} e^{- \YYzz} e^{- 2 (1-\rho) \YYrr(\TPS)}  e^{- 3 \rho \YYrr(t)}\\
& + C C_1^4 \YYzz^{4 \sigma} e^{- \YYzz} e^{- 4(1-\rho) \YYrr(\TPS)} e^{- 4 \rho \YYrr(t)}    \leq 2 C C_1 \YYzz^{2 \sigma} e^{- \YYzz} e^{- 2 (1-\rho)\YYrr( \TPS )} e^{-2 \rho \YYrr(t)} ,
\end{split}\end{equation*}
for $\YYzz$ large enough (depending on $C_1$).
By integrating over $[-\TSR,t]$, for all $t\in [-\TSR,T^*]$, using $\dot \YYrr(t)\geq C e^{-\frac 12 \YYzz}$, we obtain, for $C_1$ large,
$$
{\cal F}_-(t)-{\cal F}_-(-\TSR) \leq C C_1 \YYzz^{2 \sigma} e^{-\frac 12 \YYzz} e^{- 2 (1-\rho)\YYrr( \TPS )} e^{-2 \rho \YYrr(t)}.
$$
Using \eqref{eq:FGcoer} and \eqref{eq:mod2}, we obtain
\begin{equation}\label{eq:Te1}
\forall t\in [-\TSR,T^*],\quad
\|\varepsilon(t)\|_{H^1}\leq C \sqrt{C_1}  \YYzz^{\sigma } e^{-\frac 14 \YYzz} e^{- (1-\rho)\YYrr( \TPS )} e^{-\rho \YYrr(t)}.
\end{equation}

Second, we control $\mu(t)$ and $\pyy(t)$ on $[-\TSR,T^*]$. .
By \eqref{eq:mod2}-\eqref{eq:mod3}, we have
\begin{equation}\label{eq:inP}
	|\mu(-\TSR)-\dot \YYrr(\TSR)|+|\summu(-\TSR)| + |\pyy(-\TSR)- \YYrr(\TSR)| + |\pyyb(t)|
\leq C  e^{-\rho  \YYrr(\TSR)} e^{-5 \YYrr(\TPS)} ,
\end{equation}
and  using \eqref{eq:sd2} and \eqref{eq:imp} on $[-\TSR,T^*]$, we have 
\begin{align}
  &  | \dot \mu - 2 \alpha e^{-\pyy} | \leq C \YYzz^\sigma e^{-\YYzz} e^{- (1-\rho) \YYrr( \TPS )} e^{-\rho \YYrr(t)}  ,\label{eq:sysP1}  \\
&   | \dot \pyy - \mu | \leq C \sqrt{C_1} \YYzz^{\sigma} e^{-\frac 14 \YYzz} e^{-(1-\rho) \YYrr( \TPS ) }e^{-\rho \YYrr(t)}, \label{eq:sysP2}
\end{align}
by choosing $\YYzz$ large enough (depending on $C_1$).
Moreover, $\ddot \YYrr = 2 \alpha e^{-\YYrr}$, so that by \eqref{eq:sysP1}  
and \eqref{eq:imp},
$$
|\dot \mu - \ddot \YYrr| \leq C \YYzz^\sigma e^{-\YYzz} e^{- (1-\rho) \YYrr( \TPS )} e^{-\rho \YYrr(t)}.
$$
Integrating over $[-\TSR,t]$ using $\dot \YYrr(t)\geq C e^{-\frac 12 \YYzz}$ and \eqref{eq:inP},  we find $\forall t\in [-\TSR,T^*],$
\begin{equation}\label{eq:Tm1}
	  |\mu(t) - \dot \YYrr(t)|\leq C \YYzz^\sigma e^{-\frac 12 \YYzz} e^{- (1-\rho) \YYrr( \TPS )} e^{-\rho \YYrr(t)}.
\end{equation}
Inserting this into \eqref{eq:sysP2}, and integrating using \eqref{eq:inP}, we find $\forall t\in [-\TSR,T^*],$
\begin{equation}\label{eq:Th1}
	 | \pyy(t) - \YYrr(t) | \leq   C \sqrt{C_1} \YYzz^{\sigma} e^{\frac 14 \YYzz} e^{-(1-\rho) \YYrr( \TPS ) }e^{-\rho \YYrr(t)}.
\end{equation}
Now, we choose $C_1$ large enough so that by \eqref{eq:Te1}, \eqref{eq:Tm1}, \eqref{eq:Th1},  we have 
\begin{equation}\label{eq:impf}\begin{split}
&    e^{\frac 14 \YYzz} \|\varepsilon(t) \|_{H^1} + e^{-\frac 14 \YYzz} |\pyy(t)-\YYrr(t)|+
e^{\frac 12 \YYzz} |\mu(t)-\dot \YYrr(t)|  \leq  (C_1/2)  \YYzz^{\sigma} e^{-(1-\rho)\YYrr(\TPS)}
e^{-\rho \YYrr(t)}.
\end{split}\end{equation}
By \eqref{eq:mub}-\eqref{eq:yb} and continuity arguments, it follows that for such $C_1$, $T^*>-\TPS$. The constant $C_1$ is now fixed.

\medskip

\emph{Bootstrap of \eqref{eq:imp2}.}
Let $C_2>2 C_1$ to be chosen later. From the previous step, we have
\begin{equation}\label{eq:atT0}\begin{split}
&    \|\varepsilon(-\TPS) \|_{H^1} \leq  C_1 \YYzz^{\sigma} e^{-\frac 14 \YYzz} e^{-\YYrr(\TPS)},
\quad |\pyy(-\TPS)-\YYrr(-\TPS)|\leq  C_1 \YYzz^{\sigma} e^{\frac 14 \YYzz}e^{-\YYrr(\TPS)}, 
\\&  |\mu(-\TPS)-\dot \YYrr(-\TPS)|\leq  C_1  \YYzz^{\sigma} e^{- \frac 12  \YYzz}e^{-\YYrr(\TPS)},
\end{split}\end{equation}
As before, the objective is to prove $T^*>\TPP$ for $C_2>C_1$ large enough by strictly improving \eqref{eq:imp2} on $[-\TPS,T^*]$.

First, we prove that $T^*>-\TPP$. We work on the time interval $[-\TPS,t^*]$, where $t^*=\min(-\TPP,T^*)$.
To estimate $\varepsilon(t)$, we argue as before; using \eqref{eq:cF} and \eqref{eq:imp2}, we have
\begin{align*}
\frac d{dt}{\cal F}_-(t) & \leq 
	C C_2 \YYzz^{2\sigma} e^{-\YYzz} e^{-2 \YYrr(t)}  + C C_2^2 \YYzz^{2\sigma} e^{-(1+\rho)\YYzz} e^{-2 \YYrr(t)} + C C_2^4 \YYzz^{4 \sigma} e^{- \YYzz} e^{-4 \YYrr(t)} \\
& \leq 2 C C_2 \YYzz^{2\sigma} e^{-\YYzz} e^{-2 \YYrr(t)},
\end{align*}
for $\YYzz$ large (depending on $C_2$).
Using $|\dot \YYrr(t)|\geq \dot \YYrr(\TPP)\geq C\epsilon e^{-\frac 12 \YYzz}$ (see \eqref{eq:cH}), \eqref{eq:FGcoer} and \eqref{eq:atT0}, we obtain
\begin{equation}\label{eq:Te2}
	\forall t\in [-\TPS,t^*],\quad 
\|\varepsilon(t)\|_{H^1} \leq C (\sqrt{\frac {C_2} \epsilon}+C_1) \YYzz^\sigma e^{-\frac 14 \YYzz} e^{-\YYrr(t)}
\leq  2 C  \sqrt{\frac {C_2} \epsilon}  \YYzz^\sigma e^{-\frac 14 \YYzz} e^{-\YYrr(t)},
\end{equation}
by choosing $C_2$ large enough (depending on $\epsilon$).

Now, we turn to $\mu(t)$ and $y(t)$.  
We obtain from \eqref{eq:sd2}  
\begin{equation}\label{eq:sysQ}\begin{split}
& |\dot \mu - 2 \alpha e^{-\pyy} | \leq C \YYzz^\sigma e^{-\YYzz} e^{-\YYrr(t)},\\
& |\dot \pyy - \mu| \leq C \sqrt{\frac {C_2} \epsilon} \YYzz^\sigma e^{-\frac 14 \YYzz} e^{-\YYrr(t)}.
\end{split}\end{equation}
We need an approximate conserved quantity for the dynamical system; let
\begin{equation}\label{eq:deF} 
	H (t)=\mu^2(t) + 4 \alpha e^{-\pyy(t)}.
\end{equation}
Then, $\dot H (t)= 2 \mu \dot \mu - 4 \alpha \dot \pyy e^{-\pyy}$, so that by \eqref{eq:sysQ} and $|\mu(t)|\leq C e^{-\frac 12 \YYzz}$,
$$
|\dot H |\leq C \YYzz^\sigma e^{-\frac 32 \YYzz} e^{-\YYrr(t)} + C  \sqrt{\frac {C_2} \epsilon} \YYzz^\sigma  e^{-\frac 14 \YYzz} e^{-2 \YYrr(t)} \leq
C \sqrt{\frac {C_2} \epsilon} \YYzz^\sigma e^{-\frac 54 \YYzz} e^{-\YYrr(t)}.
$$
By \eqref{eq:atT0} and $(\dot \YYrr)^2 + 4 \alpha e^{-\YYrr} = 4 \alpha e^{-\YYzz}$, we have 
$$| H (-\TPS) - 4 \alpha e^{-\YYzz} | \leq C \YYzz^\sigma e^{-\frac 34 \YYzz} e^{-\YYrr(\TPS)}.
$$
Thus, integrating on $[-\TPS,t]$, for all $t\in [-\TPS,t^*]$, using $|\dot \YYrr(t)|\geq C\epsilon e^{-\frac 12 \YYzz}$, we get
$$
|H (t) - 4 \alpha e^{-\YYzz}|\leq C \epsilon^{-3/2}  \sqrt{C_2} \YYzz^\sigma e^{-\frac 34 \YYzz} e^{-\YYrr(t)}
$$
Using \eqref{eq:sysQ}, we have $|\dot \pyy - \mu| \leq C \sqrt{\frac {C_2} \epsilon} \YYzz^\sigma e^{-\frac 14 \YYzz} e^{-\YYrr(t)}$, and thus
$$
|(\dot \pyy(t))^2 + 4 \alpha e^{-\pyy(t)} - 4 \alpha e^{-\YYzz}|
\leq  C  \epsilon^{-3/2} \sqrt{C_2} \YYzz^\sigma e^{-\frac 34 \YYzz} e^{-\YYrr(t)}
\leq {2 C}  \epsilon^{-3/2}  \sqrt{C_2} \YYzz^\sigma e^{-\frac 34 \YYzz} e^{-\pyy(t)}.
$$
Set $a_0 =  (C/(2\alpha\epsilon^{3/2})) \sqrt{C_2}     \YYzz^\sigma e^{-\frac 34 \YYzz} $, so that
\begin{equation}\label{eq:hh}
4 \alpha (1-a_0) e^{-\pyy(t)}\leq 4 \alpha e^{-\YYzz} - (\dot \pyy(t))^2 \leq 4 \alpha (1+a_0)e^{-\pyy(t)}.
\end{equation}
For $s\in [-S_0,s^*]$, where $S_0 = \sqrt{\alpha} e^{-\frac 12 \YYzz} \TPS$, $s^* = \sqrt{\alpha} e^{-\frac 12 \YYzz} t^*\leq -s_1 = -\sqrt{\alpha} e^{-\frac 12 \YYzz} \TPP \leq -C \epsilon$, we define
\begin{equation}\label{def:f}
f(s)=\exp \left(\frac 12 \left(y\left(\frac 1 {\sqrt{\alpha}} e^{\frac {\YYzz}2} s\right) - \YYzz\right) \right).
\end{equation}
For $s<s^*$, we have by \eqref{eq:HH} and \eqref{eq:imp2}
\begin{equation}\label{eq:minf}
\YYrr\left(\frac 1 {\sqrt{\alpha}} e^{\frac 12 \YYzz} s \right) - \YYzz \geq 
 2 \ln (\cosh(s^*)) \geq C \epsilon, \quad \text{and so} \quad f(s)\geq 1+ C\epsilon.
 \end{equation}

Next, we have $\dot f(s)= \frac 1{2 \sqrt{\alpha}} e^{\frac {\YYzz}2} \dot \pyy\left(\frac 1 {\sqrt{\alpha}} e^{\frac {\YYzz}2} s\right) f(s)$, and by \eqref{eq:hh}, we find
$$
1- a_0 \leq f^2 - (\dot f)^2 \leq 1+ a_0.
$$
We define
$$
f_+(s)= \frac {f(s)}{\sqrt{1-a_0}},\quad f_-(s)= \frac {f(s)}{\sqrt{1+a_0}},
$$
so that by \eqref{eq:minf}, for $\YYzz$ large enough (depending on $\epsilon$ and $C_2$), we have
$f_+(s)> f_-(s) > 1$ on $ [-S_0,s^*]$, and $f_{\pm}$ satisfy
$$
f_+^2 - (\dot f_+)^2 \geq 1, \quad f_-^2 - (\dot f_-)^2 \leq 1.
$$
Let now $g_{\pm}>0$ be such that $\cosh(g_{\pm}) = f_{\pm}$. Then
$(\dot g_+)^2 \leq 1$ and $(\dot g_-)^2 \geq 1$. Since $\dot g_{\pm} < 0 $ (it has the sign of $\dot \pyy$ and $\mu$), we find:
$$
\forall s\in [-S_0,s^*],\quad
\dot g_+(s) \geq  - 1 \quad \text{and} \quad \dot g_(s)- \leq -1.
$$
At $s=-S_0$, by \eqref{eq:atT0}, we have
$$
|f(-S_0) - \cosh S_0|\leq C e^{\frac 12 (\YYrr(\TPS)-\YYzz)} (\pyy(-\TPS)-\YYrr(\TPS))
\leq C \YYzz^\sigma e^{-\frac 14 \YYzz}e^{-\frac 12 \YYrr(\TPS)} \leq C a_0,
$$
and so
$$
|\cosh(g_{\pm}(-S_0)) - \cosh(S_0\mp \frac 12 a_0)|\leq C a_0
\quad \Rightarrow \quad |g_{\pm}(-S_0)-S_0|\leq  C \YYzz^\sigma e^{-\frac 34 \YYzz} 
\leq C a_0.
$$
By integration on $[-S_0,s]$, for all $s\in [-S_0,s^*]$, we find
$$
g_+(s) \geq -s - C a_0, \quad g_-(s) \leq -s + Ca_0.
$$
Thus, $\forall s\in [-S_0,s^*]$,
$$
\cosh(s+C'a_0)\leq \sqrt{1-a_0} \cosh(s+Ca_0) \leq f(s) \leq \sqrt{1+a_0} \cosh(s - Ca_0)\leq \cosh(s-C'a_0),
$$
and since $\cosh(s)=\exp(\frac 12 (\YYrr(t)-\YYzz))$, $t=\frac 1{\sqrt{\alpha}}
e^{\frac 12 Y_0} s$ (see \eqref{eq:HH}),
$$
	\YYrr(t+\frac C{\sqrt{\alpha}} e^{\frac 12 \YYzz} a_0) \leq \pyy(t) \leq \YYrr(t-\frac C{\sqrt{\alpha}} e^{\frac 12 \YYzz} a_0) 
$$
so that
\begin{equation}\label{eq:hf}
\forall t\in [-\TPS,t^*],\quad
	|\pyy(t)-\YYrr(t)|\leq Ca_0=  C  \epsilon^{-3/2}  \sqrt{C_2} \YYzz^\sigma e^{-\frac 34 \YYzz}.
\end{equation}
By \eqref{eq:sysQ} and \eqref{eq:hf}, we deduce
$$
|\dot \mu - 2 \alpha e^{-\YYrr}|\leq C \YYzz^\sigma e^{-\YYzz} e^{-\YYrr(t)} + C e^{-\YYrr(t)} a_0.
$$
Integrating on $[-\TPS,t^*]$, using $|\dot \YYrr(t)|\geq C \epsilon e^{-\frac 12 \YYzz}$ and \eqref{eq:atT0}, we find
\begin{equation}\label{eq:Tm2}
	|\mu(t) - \dot \YYrr(t)| \leq  C {\epsilon^{-5/2}}  \sqrt{C_2} \YYzz^\sigma e^{-\frac 14 \YYzz} e^{-\YYrr(t)}.
\end{equation}
By \eqref{eq:Te2}, \eqref{eq:hf}, \eqref{eq:Tm2},  
assuming $C_2 > C_2^\epsilon= C /\epsilon^5$, for $C>0$ large enough, we strictly improve \eqref{eq:imp2} and thus we obtain
$t^*=-\TPP$ and by continuity $-\TPP<T^*$ for such $C_2$.

\medskip

Now, assuming $T^*\leq T''$, we argue on $[-\TPP,T^*]$ in order to prove by contradiction that $T^*>\TPP$.
First, by \eqref{eq:imp2}, we have
$$
|e^{-\pyy(t)}-e^{-\YYrr(t)}|\leq 2 e^{-\YYrr(t)} |\pyy(t)-\YYrr(t)|\leq  C  C_2 \YYzz^\sigma e^{-\frac 74 \YYzz},
$$
and thus, by \eqref{eq:sd2}, \eqref{eq:imp2} and $\ddot \YYrr = 2\alpha e^{-\YYrr}$, we obtain
$$
|\dot \mu - \ddot \YYrr|\leq   C  C_2  \YYzz^\sigma e^{-\frac 74 \YYzz}.
$$
By integration on $[-\TPP,t]$, for all $t\in [-\TPP,T^*]$, and using \eqref{eq:Tm2} at $t=-\TPP$, we find
\begin{equation}\label{eq:Tm3}
|\mu(t) - \dot \YYrr(t)| \leq   C {\epsilon^{-5/2}}  \sqrt{C_2} \YYzz^\sigma e^{-\frac 54 \YYzz}
+  C  C_2 \epsilon \YYzz^\sigma e^{-\frac 54 \YYzz}.
\end{equation}
From \eqref{eq:Tm3}, \eqref{eq:sd2} and \eqref{eq:imp2}, we have ($Y_0$ large enough)
$$
|\dot \pyy - \dot \YYrr | \leq  C {\epsilon^{-5/2}}  \sqrt{C_2} \YYzz^\sigma e^{-\frac 54 \YYzz}
+  C  C_2 \YYzz^\sigma e^{-\frac 54 \YYzz},
$$
and thus, integrating on $[-\TPP,t]$, for all $t\in [-\TPP,T^*]$, and using \eqref{eq:hf} at $t=-\TPP$, we find
\begin{equation}\label{eq:Th3}
	|\pyy(t) - \YYrr(t)|\leq C \epsilon^{3/2} \sqrt{C_2} \YYzz^\sigma e^{-\frac 34 \YYzz}
+  C  C_2 \epsilon  \YYzz^\sigma e^{-\frac 34 \YYzz}.
\end{equation}
From \eqref{eq:Tm3} and \eqref{eq:Th3}, we fix $\epsilon>0$ (independent of $C_2$) small enough so that
\begin{equation*}\begin{split}
& |\mu(t) - \dot \YYrr(t)| \leq C ' \sqrt{C_2} \YYzz^\sigma e^{-\frac 54 \YYzz}
+ \frac 1{16}   C_2  \YYzz^\sigma e^{-\frac 54 \YYzz},\\
& |\dot \pyy(t) - \dot \YYrr(t) | \leq  C' \sqrt{C_2} \YYzz^\sigma e^{-\frac 34 \YYzz}
+ \frac 1{16}   C_2 \YYzz^\sigma e^{-\frac 34 \YYzz}.
\end{split}
\end{equation*}
Thus, $\epsilon>0$ being fixed, for $C_2$ large enough ($C_2 \geq 16 (C')^2$), we obtain
\begin{equation}\label{eq:Tmh3} 
  |\mu(t) - \dot \YYrr(t)| \leq  \frac 18   C_2  \YYzz^\sigma e^{-\frac 54 \YYzz},\quad 
 |\dot \pyy(t)- \dot \YYrr(t)| \leq \frac 18   C_2 \YYzz^\sigma e^{-\frac 34 \YYzz}.
\end{equation}

On $[-\TPP,T^*]$, since $|\dot \mu(t) - \ddot \YYrr(t)|\leq C \YYzz^\sigma e^{-\frac 74 \YYzz}$
and $\ddot \YYrr\geq  \frac 32  \alpha e^{-\YYzz}$, we have $\dot \mu \geq \alpha e^{-\YYzz}$.
The fact that $\dot \YYrr(0)=0$ and this lower bound on $\dot \mu$ implies, if $T^*$ is large enough, 
the existence of a unique $t_0\in [-\TPP,T^*]$ such that 
$\mu(t)<0$ if $t\in [-\TPP,t_0]$, $\mu(t_0)=0$ and $\mu(t)>0$ for $t\in [-\TPP,T^*]$,
moreover, $|t_0|\leq C \YYrr^{\sigma} e^{-\frac 14 \YYzz}$.

\smallskip

Finally, to control $\|\varepsilon(t)\|_{H^1}$ on $[-\TPP,T^*]$, we use the function ${\cal F}_-(t)$ introduced in Proposition \ref{PR:cFG} on $[-\TPP,t_0]$ and the function ${\cal F}_+(t)$ introduced in Proposition \ref{PR:cFG} on $[t_0,T^*]$.
As before, we get the following  bound on $\varepsilon(t)$   using  \eqref{eq:cF}, \eqref{eq:cG}, \eqref{eq:imp2}
and \eqref{eq:Te2} (recall that $\epsilon>0$ has been fixed)
\begin{equation}\label{eq:Te3}
\forall t\in [-\TPP,T^*),\quad 
\|\varepsilon(t)\|_{H^1} \leq C \sqrt{C_2} \YYzz^\sigma e^{-\frac 54 \YYzz} \leq 
\frac {C_2} 8 \YYzz^\sigma e^{-\frac 54 \YYzz} ,
\end{equation}
for $C_2$ large enough.
Thus, combining \eqref{eq:Tmh3}  and \eqref{eq:Te3}, we have proved for   $t\in [-\TPS,T^*]$,
(note that for $t\in [-\TPP,\TPP]$, $|\YYrr(t)-\YYzz|\leq C \epsilon^2<\frac 12$ so that $e^{-\YYzz}\leq 2 e^{-\YYrr(t)}$)
\begin{equation}\label{eq:imp2f} 
    \|\varepsilon(t) \|_{H^1} + |\mu(t)-\dot \YYrr(t)|\leq   (C_2/2) \YYzz^{\sigma} e^{-\frac 14 \YYzz} e^{-  \YYrr(t)},
\quad |\pyy(t)-\YYrr(t)|\leq  (C_2/2) \YYzz^{\sigma} e^{-\frac 34 \YYzz},
 \end{equation}
which improves strictly \eqref{eq:imp2} on $[-\TPP,T^*]$
and thus,   $T^*>\TPP$.
 
\medskip

\emph{Bootstrap of \eqref{eq:imp3}}.
Now, we prove $T^*=T$.
Let us first estimate $\|\varepsilon(t)\|$ on $[\TPP,T^*]$, using the function ${\cal F}_+(t)$ defined in Proposition \ref{PR:cFG}, \eqref{eq:cF}. We get
\begin{equation*}
	\frac d{dt}{\cal F}_+(t) \leq C C_{3} \YYrr_{0}^{2\sigma} e^{-2\YYrr_{0}} e^{-\YYrr(t)} 
		+ CC_{3}^2 \YYrr_{0}^{2\sigma}e^{-3 \YYrr_{0}}  e^{-\rho \YYrr_{0}} 
		+ CC_{3}^4 \YYrr_{0}^{4\sigma}e^{5 \YYrr_{0}}.
\end{equation*}
Integrating on $[\TPP,t]$, for all $t\in [\TPP,T^*]$, using $\dot \YYrr(t)\geq C e^{-\frac 12 \YYrr_{0}}$
($\epsilon>0$ is fixed ), $|\TPS|\leq C \YYrr_{0} e^{\frac 12 \YYrr_{0}}$, and $|{\cal F}_+(\TPP)|\leq C \YYzz^{2 \sigma} e^{-\frac 52 \YYzz}$ (from \eqref{eq:imp2}), we obtain,
for $\YYrr_{0}$ large enough  (depending on $C_{3}$),
\begin{equation}\label{eq:Te4}
\forall t\in [\TPP,T^*],\quad
	\|\varepsilon(t)\|_{H^1}\leq C \sqrt{C_{3}} \YYrr_{0}^\sigma e^{-\frac 54 \YYrr_{0}}.
\end{equation}

Next, we control the dynamics of $\mu$ and $y$. We need to argue differently on two regions in time: $[\TPP,\TTPP]$ and $[\TTPP,\TSR]$, where
$$ \text{$\TTPP\in (\TPP,\TSR)$ is such that $M \YYzz e^{-\YYrr(\TTPP)}= e^{-\YYzz}$,}$$
 for $M>10$ large enough to be chosen later.

First, we prove that $T^*>\TTPP$, setting $t^*=\min(T^*,\TTPP)$ and strictly improving estimates \eqref{eq:imp3} on $[\TPP,t^*]$. By \eqref{eq:sd2},
\eqref{eq:imp3} and \eqref{eq:Te4}, we have
\begin{equation}\label{eq:mod11}\begin{split}
&	|\dot \mu - 2 \alpha e^{-\pyy}| \leq C \YYzz^\sigma e^{-\YYzz} e^{-\YYrr(t)} 
+ C C_3 \YYzz^{2\sigma} e^{-\frac 52 \YYzz},\\
& |\dot \pyy- \mu |\leq C \sqrt{C_3} \YYzz^\sigma e^{-\frac 54 \YYzz}.
\end{split}\end{equation}
As in the proof of \eqref{eq:imp2}, we set $H (t)= (\dot \pyy)^2 + 4 \alpha e^{-\pyy}$, so that by \eqref{eq:mod11},
$$
|\dot H |\leq   C C_3^2 \YYzz^{2\sigma} e^{- 3 \YYzz} + C \sqrt{C_3} \YYzz^\sigma e^{-\frac 54 \YYzz}e^{-\YYrr(t)}.
$$
By \eqref{eq:imp2}, taken at $t=\TPP$, we have
$$
|\mu(\TPP)-\dot \YYrr(\TPP)|\leq C \YYzz^\sigma e^{-\frac 54 \YYzz},\quad
|\pyy(\TPP) - \YYrr(\TPP)|\leq C \YYzz^\sigma e^{-\frac 34 \YYzz}.
$$
Hence, by $(\dot \YYrr)^2 + 4 \alpha e^{-\YYrr} = 4 \alpha e^{-\YYzz}$, we obtain
$|H (\TPP) - 4 \alpha e^{-\YYzz}|\leq C \YYzz^\sigma e^{-\frac 74 \YYzz}.$
Thus, integrating in $[\TPP, t]$, for all $t\in [\TPP,t^*]$, we get
$$
|H (t) - 4 \alpha e^{-\YYzz}| \leq C \sqrt{C_3} \YYzz^\sigma e^{-\frac 74 \YYzz}.
$$
Since $\TPP\leq t\leq t^*\leq \TTPP$, we have 
$e^{-\YYzz}= M\YYzz e^{-\YYrr(\TTPP)} \leq M\YYzz e^{-\YYrr(t)}$, and so
$$
|H (t) - 4 \alpha e^{‚Äî\YYzz}|\leq CM \sqrt{C_3} \YYzz^{\sigma+1} e^{-\frac 34 \YYzz} e^{-\YYrr(t)}.
$$
Using $|\mu-\dot \pyy|\leq C  M \sqrt{C_3} \YYzz^\sigma e^{-\frac 14 \YYzz} e^{-\YYrr(t)}$,
we obtain
\begin{equation}\label{mod12}
|\dot \pyy^2 + 4 \alpha e^{-\pyy}  - 4 \alpha e^{-\YYzz}|\leq CM \sqrt{C_3}\YYzz^{\sigma+1} e^{-\frac 34 \YYzz} e^{-\YYrr(t)}
\leq 2CM \sqrt{C_3} \YYzz^{\sigma+1} e^{-\frac 34 \YYzz} e^{-\pyy(t)}.
\end{equation}
Let $b_0=2 C M \sqrt{C_3} \YYzz^{\sigma+1} e^{-\frac 34 \YYzz}$. 
Applying the same strategy as before, we set $f$
as in \eqref{def:f}, 
$$
f_+(s)=\frac {f(s)}{\sqrt{1-b_0}},\quad
f_-(s)=\frac {f(s)}{\sqrt{1+b_0}},
$$
and $\cosh(g_\pm)=f_\pm$. Arguing as in the proof of \eqref{eq:imp2}, using
\eqref{mod12} and \eqref{eq:imp2} at $t=\TPP$, we obtain
$$\dot g_+\leq 1, \quad |g_+(s_1)-s_1|\leq C b_0, \text{ and so }
g_+(s)\leq s + Cb_0,$$
$$\dot g_-\geq 1, \quad |g_-(s_1)-s_1|\leq C b_0, \text{ and so }
g_-(s)\geq s - Cb_0.$$
We deduce the following
\begin{equation}\label{eq:Th4}
\forall t\in [\TPP,t^*],\quad
|\pyy(t)-\YYrr(t)|\leq C b_0=  C M \sqrt{C_3} \YYzz^{\sigma+1} e^{-\frac 34 \YYzz}.
\end{equation}
Thus,
$$
|\dot \mu- 2 \alpha e^{-\YYrr}|\leq C  b_0 e^{-\YYrr(t)},
$$
and by integation, using \eqref{eq:imp2},
\begin{equation}\label{eq:Tm4}
\forall t\in [\TPP,t^*],\quad
|\mu(t) - \dot \YYrr(t)|\leq C  M \sqrt{C_3} \YYzz^{\sigma+1} e^{-\frac 54 \YYzz}.
\end{equation} 
Therefore, from \eqref{eq:Te4}, \eqref{eq:Th4} and \eqref{eq:Tm4}, we see that for $C_3> C_3^M= CM^2$, for $C>0$ large enough, we strictly improve estimate \eqref{eq:imp3} and thus prove $t^*=\TTPP$, and by continuity $T^*\in (\TTPP,\TSR]$.

\medskip

Second, we prove that $T^*=\TSR$, arguing on $[\TTPP,T^*]$. This is where we will need to fix $M$ large enough.
Using \eqref{eq:imp3}, we have
$$
|e^{-\pyy}-e^{-\YYrr}|\leq C C_3 \YYzz^{\sigma+2} e^{-\frac 34 \YYzz} e^{-\YYrr(t)}.
$$
Thus, by $\ddot \YYrr = 2 \alpha e^{-\YYrr}$ and \eqref{eq:mod11}, we obtain
$$
|\dot \mu - \ddot \YYrr|\leq 
C C_3 \YYzz^{\sigma+2} e^{-\frac 54 \YYzz} e^{-\YYrr(t)}
+ C C_3^2 \YYzz^{2 \sigma} e^{-\frac 52 \YYzz}. 
$$
Integrating on $[\TTPP,t]$, for all $t\in [\TTPP,T^*]$, 
using $\dot \YYrr\geq C e^{-\frac 12 \YYzz}$, and $\TPP\leq C \YYzz e^{‚Äî\frac 12 \YYzz}$, and \eqref{eq:Tm4} at $t^*=\TTPP$, we obtain
\begin{align}
|\mu(t) - \dot \YYrr(t)| & 
\leq C C_3 \YYzz^{\sigma+2} e^{-\frac 34 \YYzz} e^{-\YYrr(\TTPP)} + 
C M \sqrt{C_3} \YYzz^{\sigma+1} e^{-\frac 54 \YYzz}\nonumber\\
& \leq C \left(\frac {C_3}{M} + M \sqrt{C_3} \right) \YYzz^{\sigma+1} e^{-\frac 54 \YYzz}.
\label{eq:Tm5}\end{align}
Since $|\dot \pyy- \mu|\leq C \sqrt{C_3} \YYzz^\sigma e^{-\frac 54 \YYzz}$, we get
$$
| \dot \pyy - \dot \YYrr |
\leq C \left(\frac {C_3}{M} + M \sqrt{C_3} \right) \YYzz^{\sigma+1} e^{-\frac 54 \YYzz}.
$$
By integration, using   $\TSR\leq C \YYzz e^{\frac 12 \YYzz}$,  and \eqref{eq:Th4}
we obtain
\begin{equation}\label{eq:Th5}
|\pyy(t)-\YYrr(t)|\leq C \left(\frac {C_3}{\sqrt{M}} + M \sqrt{C_3} \right) \YYzz^{\sigma+2} e^{-\frac 34 \YYzz}.
\end{equation}
Fix $M>1$ large enough so that \eqref{eq:Tm5} and \eqref{eq:Th5} imply
\begin{align*}
&|\mu(t) - \dot \YYrr(t)| 
\leq \left(\frac 14 {C_3} + MC  \sqrt{C_3} \right) \YYzz^{\sigma+1} e^{-\frac 54 \YYzz},
\\
&|\pyy(t)-\YYrr(t)|\leq \left(\frac 14 {C_3} + MC \sqrt{C_3} \right) \YYzz^{\sigma+2} e^{-\frac 34 \YYzz}.
\end{align*}
Then, $M$ being fixed to such value, we can choose $C_3$ large enough so that for all $t\in [\TPP,T^*]$
\begin{equation}\label{eq:imp3f} 
   \YYzz^2  e^{ \YYzz} \|\varepsilon(t) \|_{H^1}  + \YYzz e^{ \YYzz} |\mu(t)-\dot \YYrr(t)|  
+|\pyy(t)-\YYrr(t)|\leq  (C_3/2) \YYzz^{\sigma+2} e^{-\frac 34 \YYzz}.
 \end{equation}
This proves $T^*=\TSR$ and Proposition \ref{pr:st} is proved.
\end{proof} 

For future reference, we observe that from \eqref{eq:U+-}, \eqref{eq:U++}, \eqref{eq:p23},  \eqref{eq:gab}, \eqref{eq:gabb}, for $t\in [-T,T]$,
\begin{equation}\label{eq:J3} 
	|\dot \mu_j - {\cal M}_j| \leq C \YYzz^{\sigma} e^{-2 \YYzz},\quad 
	|\mu_j - \dot y_j - {\cal N}_j| \leq 	C \YYzz^{\sigma} e^{-\frac 54 \YYzz}.
\end{equation}

\subsection{Conclusion of the proof of the stability of the $2$-soliton structure}

In this section, we finish the proof of the stability part of Theorem \ref{TH:1}.

\medskip

\noindent\emph{Proof of \eqref{eq:th5}--\eqref{eq:th6} and partial proof of \eqref{eq:th8} and \eqref{eq:th9}.}
Let $T>0$ be defined as in Proposition \ref{pr:st}. We prove the existence of $\mu_j(t)$ and $y_j(t)$ and estimates  \eqref{eq:th5}--\eqref{eq:th6} separately on $(-\infty,-T]$, $[-T,T]$ and $[T,+\infty)$. It is straightforward that the functions $\mu_j(t)$ and $y_j(t)$ can be ajusted to have $C^1$ regularity on $\RR$.

\smallskip

For $t<-T$, estimate \eqref{eq:ppb} clearly implies \eqref{eq:th5}--\eqref{eq:th6}.

\smallskip

On $[-T,T]$, \eqref{eq:th5}--\eqref{eq:th6} are direct consequences of \eqref{eq:J11}--\eqref{eq:U++} and \eqref{eq:p20} (comparing in $H^1$ the approximate solution with the sum of two solitons). 

\begin{remark}\label{rk:mm}
By \eqref{eq:U++} and the definition of $V$
(see \eqref{eq:p9} and \eqref{eq:p19}), for $t\in [-T,T]$,
\begin{equation}\label{eq:rel}
\|U(t) - \qtun(t) - \qtde(t) - e^{-y(t)} (A_1(.-y_1(t))+A_2(.-y_2(t)))\|_{H^1}
\leq C Y_0^\sigma e^{-\frac 54 Y_0},
\end{equation}
where for $t$ close to $0$, the term
$e^{-y} (A_1(x-y_1)+A_2(x-y_2))$ is indeed relevant as a correction term in the computation of $U(t)$. From the behavior at $\pm \infty$ of the functions $A_1$ and $A_2$  (see Lemma \ref{LE:23}), this term decays exponentially for $x>y_1(t)$ and $x<y_2(t)$ but it contains a tail for $y_2(t)<x<y_1(t)$. Note that a similar tail appears in the integrable case $p=2$, see \cite{MMprepa}.
\end{remark}

Now, we consider the region $t\geq T$.
By \eqref{eq:eH} and \eqref{eq:lesT}, we have
$\TSR > \frac {10  } { \sqrt{\alpha}}\rho^{-1} \YYzz e^{\frac 12 \YYzz} > 10 (\rho\mu_0)^{-1} |\ln \mu_0|$. 
From \eqref{eq:J11}, \eqref{eq:U+-}, \eqref{eq:U++} and \eqref{eq:p20} written at $t=T$, we have
$$
\left\| U(T) - Q_{1+\mu_1(T)} (.-y_1(T)) -Q_{1+\mu_2(T)}(.-y_2(T))\right\| \leq C Y_0^\sigma e^{-\frac 54 Y_0}
\leq C' Y_0^\sigma e^{-\frac 34 Y_0} \mu_0 ,
$$
where $|\mu_1(T)-\mu_0|+|\mu_2(T)+\mu_0|\leq C Y_0^2 e^{-Y_0}$.

Therefore, we can apply Proposition \ref{PR:STAB} backwards (i.e. for  $t\geq T$ -- see Remark \ref{RE:STAB}), with $\omega=C' Y_0^\sigma e^{-\frac 34 Y_0}$. There exist
$y_1(t)$, $y_2(t)$ and $\displaystyle\mu_1^+=\lim_{+\infty} \mu_1$, $\displaystyle\mu_2^+=\lim_{+\infty} \mu_2$, such that
$$
w(t)= U(t) - Q_{1+\mu_1(T)} (.-y_1(t)) -Q_{1+\mu_2(T)}(.-y_2(t))
$$
satisfies
\begin{equation}\label{eq:final}\begin{split}
	\sup_{t\in [T,+\infty)} \|w(t)\|_{H^1}\leq C \YYzz^{\sigma} e^{-\frac 54 \YYzz},\quad  \lim_{t\to +\infty} \|w(t)\|_{H^1(x>-(99/100)t)}=0, 
	\\ |\mu_1^+-\mu_0|+ |\mu_2^+ +\mu_0|\leq  C Y_0^2 e^{-Y_0},
	\quad 
\lim_{+\infty} \dot y_j=\mu_j^+\quad (j=1,2).
\end{split}\end{equation}

Finally, using the conservation laws and the above asymptotics for $w(t)$, we claim the following refined estimates on the limiting scaling parameters:
\begin{equation}\label{eq:am}  	0 \leq   {\mu_1^+}  - \mu_0 \leq  C \YYzz^{2\sigma} e^{-2 \YYzz},\quad
0 \leq   -{\mu_2^+}  - \mu_0 \leq  C \YYzz^{2\sigma} e^{-2 \YYzz},
\end{equation}
which is a consequence of \eqref{eq:final} and the following lemma.

\begin{lemma}[Monotonicity of the speeds by conservation laws]\label{le:am}
	There exists $C>0$ such that
	\begin{equation*}
		\begin{split}
		&	\frac 1 C e^{\YYzz}\limsup_{t\to +\infty} \|w(t)\|_{H^1}^2 \leq \frac {\mu_1^+}{\mu_0} - 1 \leq  C e^{\YYzz}\liminf_{t\to +\infty} \|w(t)\|_{H^1}^2\leq C \YYzz^{2\sigma} e^{-\frac 32 \YYzz}, \\
		&	\frac 1 C e^{\YYzz}\limsup_{t\to +\infty} \|w(t)\|_{H^1}^2 \leq  - \frac {\mu_2^+}{\mu_0} - 1 \leq  Ce^{\YYzz} \liminf_{t\to +\infty} \|w(t)\|_{H^1}^2\leq C \YYzz^{2\sigma} e^{-\frac 32 \YYzz}.
		\end{split}
	\end{equation*}
\end{lemma}

\begin{proof}
We write the  mass and energy conservation for $U(t)$ (see \eqref{eq:i5} and \eqref{eq:i6}) and then pass to the limit
$t\to -\infty$, $t\to +\infty$, using \eqref{eq:final}.
We deduce the existence of the limits
$\lim_{+\infty} M(w)$ and $\lim_{+\infty} \mathcal{E}(w)$ and the following identities
\begin{align}
&	M(Q_{1+\mu_0})+M(Q_{1-\mu_0}) = M(Q_{1+\mu_1^+})+M(Q_{1+\mu_2^+}) + \lim_{+\infty} M(w),	 \label{eq:am1}\\
&	\mathcal{E}(Q_{1+\mu_0})+\mathcal{E}(Q_{1-\mu_0}) = \mathcal{E}(Q_{1+\mu_1^+})+\mathcal{E}(Q_{1+\mu_2^+}) + \lim_{+\infty} \mathcal{E}(w).	 \label{eq:am2}
\end{align}
Let
\begin{equation*}
	\nu_1 = \frac {\mathcal{E}(Q_{1+\mu_0})-\mathcal{E}(Q_{1+\mu_1^+})}{M(Q_{1+\mu_1^+})
-M(Q_{1+\mu_0})},\quad
	\nu_2 = \frac {\mathcal{E}(Q_{1-\mu_0})-\mathcal{E}(Q_{1+\mu_2^+})}{M(Q_{1+\mu_2^+})
-M(Q_{1-\mu_0})}.
\end{equation*}
so that by \eqref{eq:A15} and \eqref{eq:final},
\begin{equation*}
	\left| \frac {\nu_1 -1 }{\mu_0}- 1 \right|\leq \frac 14, \quad 
		\left| \frac {\nu_2-1}{-\mu_0}- 1 \right|\leq \frac 14.
\end{equation*}
We combine \eqref{eq:am1} and \eqref{eq:am2} to get
\begin{align*}
	\lim_{+\infty} \mathcal{E}(w) 	& = \nu_1 (M(Q_{1+\mu_1^+})-M(Q_{1+\mu_0})) + \nu_2 (M(Q_{1+\mu_2^+})-M(Q_{1-\mu_0})), 	\\
    		& = (\nu_1 - \nu_2)  (M(Q_{1+\mu_1^+})-M(Q_{1+\mu_0})) - \nu_2 \lim_{+\infty} M(w),	\\
		& = (\nu_1 - \nu_2)  (M(Q_{1-\mu_0})-M(Q_{1+\mu_2^+})) - \nu_1 \lim_{+\infty} M(w).
\end{align*}
Since $\|w\|_{L^\infty}\leq C \|w\|_{H^1}\leq C \YYzz^{\sigma} e^{-\frac 54 \YYzz}$,  we have
$$ \frac 12 \limsup_{+\infty} \|w\|_{H^1}^2<
\lim_{+\infty} \mathcal{E}(w) + \nu_2 \lim_{+\infty} M(w) <2 \liminf_{+\infty} \|w\|_{H^1}^2,$$ 
$$ \frac 12 \limsup_{+\infty} \|w\|_{H^1}^2<\lim_{+\infty} \mathcal{E}(w) + \nu_1 \lim_{+\infty} M(w) <2 \liminf_{+\infty} \|w\|_{H^1}^2.$$
Using ${\frac d{dc} Q_{c}}_{| c=1} >0$,
we finish the proof of Lemma \ref{le:am}.
\end{proof}

\subsection{Proof of Theorem \ref{TH:2}}
First, we claim the following sharp stability result to be proved in Appendix B.

\begin{proposition}\label{PR:SHARP}
Let $U$ be defined as in Theorem \ref{TH:1}.
For $\mu_0>0$  small enough, if $u(t)$ is a solution of {\eqref{eq:KDV}} such that
\begin{equation}\label{eq:sh1}
	\|u({{\mathcal T}_1})-U({{\mathcal T}_1})\|_{H^1} = \omega \mu_0
\end{equation}
for some ${\mathcal T}_1$, where $0<\omega<|\ln \mu_0|^{-2}$,
then there exist $t\in \RR \mapsto (T(t), X(t))\in \RR^2$ such that
\begin{equation}\label{eq:sh2}
	\forall t\in \RR ,\quad 
 \|u(t+T(t),.+X(t))-U(t)\|_{H^1} + |\dot X(t)|+ e^{-\frac 12 \YYzz}|\dot T(t)|\leq C \omega \mu_0 .
\end{equation}
\end{proposition}

We continue the proof of Theorem \ref{TH:2}.
Let $\tilde \mu_0\in \RR$ and $\tilde Y_{0}>0$ be such that $$\mu_0 = \left(\tilde \mu_0^2 + 4 \alpha e^{-\tilde Y_0}\right)^{1/2}$$ is small enough, let
$u_0\in H^1$ be as in \eqref{eq:th2b} and let $u(t)$ be the corresponding solution of \eqref{eq:KDV}. We assume that $\tilde \mu_0\leq 0$, the proof being the same in the case $\tilde \mu_0>0$
by using the transformation $x\to -x$, $t\to -t$ and translation in space invariance.

For this value of $\mu_0>0$, let $U(t,x)$ and $Y(t)$ be defined as in Theorem \ref{TH:1} and Sections 4.1 and 4.2.
Recall  that  for all $t$, $ \dot Y^2(t) +  4\alpha e^{-Y(t)} =  4\mu_0^2$.
Since $\tilde Y_{0}\geq Y_0$, there exists
$\tilde T_0<0$   such that $Y(\tilde T_0)=\tilde Y_0$, so that $\dot {Y}(\tilde  T_0)= 2 \tilde \mu_0$.  We claim that for some  $\mathcal{X}_1\in \RR$,
\begin{equation}\label{eq:tbp}
	\|U(\tilde T_0,.+\mathcal{X}_1) - Q_{1-\tilde \mu_0}(.- \tfrac 12 \tilde Y_0) - Q_{1+\tilde \mu_0}(.+\tfrac 12  \tilde Y_0)\|_{H^1} \leq C |\ln \mu_0|^{\sigma+2} \mu_0^{3/2}.
\end{equation}
Indeed, if $\tilde T_{0}<-T$ then we simply use
\eqref{eq:ppb}. Otherwise, by Proposition \ref{pr:st} applied at $t=\tilde T_{0}$: 
\begin{equation*}
	\begin{split}
		& | y(\tilde T_{0}) - Y(\tilde T_{0}) |\leq C |\ln \mu_0|^{\sigma+2} \mu_0^{3/2},
		\quad \|\varepsilon(\tilde T_{0})\|_{H^1} \leq C |\ln \mu_0|^{\sigma} \mu_0^{5/2},	\\
		& |\mu_1(\tilde T_{0}) + \tfrac 12 \dot Y(\tilde T_{0}) | \leq C |\ln \mu_0|^2 \mu_0^{2}, \quad
		|\mu_2(\tilde T_{0}) - \tfrac 12 \dot Y(\tilde T_{0}) | \leq C |\ln \mu_0|^2 \mu_0^{2},
	\end{split}
\end{equation*}
and \eqref{eq:tbp} then follows from \eqref{eq:p20}.

Combining  \eqref{eq:tbp} and \eqref{eq:th2b}, we get
$$
\| u_{0}  - U(\tilde T_{0},.+\mathcal{X}_1)\|_{H^1} \leq C \omega \mu_0 + C |\ln \mu_0|^{\sigma +2} \mu_0^{3/2},
$$
and by Proposition \ref{PR:SHARP}, this implies \eqref{eq:th2c} for some $\sigma$.

\section{Nonexistence of a pure $2$-soliton and interaction defect}\label{sec:5}

In this section, we complete the proof of Theorem \ref{TH:1} by proving the lower bounds in \eqref{eq:th8} and \eqref{eq:th9}.

\subsection{Refined control of the translation parameters}

From  the analysis of Section 4, the error term in the dynamical system for $\mu_j(t)$, $y_j(t)$ is not sharp enough to justify rigorously  the defect in the interaction.
Now we introduce specific functionals $\mathcal{J}_j(t)$ related to 
the translation parameters $y_j(t)$ to obtain a sharper version of  the dynamical system. 
Recall that $\qtud$ and $\Lambda \qtud$ are defined at the beginning of Section 2.1.

\begin{lemma}\label{PR:J}
Under the assumptions of Proposition \ref{pr:st}, for $j=1,2$, let
\begin{equation}\label{eq:J}
 {\cal J}_j(t)= \frac 1{\int Q \Lambda Q} \int \varepsilon(t,x) \iscal_j(t,x) dx \quad  \text{where} \quad \iscal_j(t,x)=  \int_{-\infty}^x \Lambda \qtud(t,y) dy.
\end{equation}
Then $\JJ_j(t)$ is well-defined and the following hold
\begin{enumerate}
	\item[\rm (i)]  Estimates on $\JJ_j$.
	\begin{equation}\label{eq:JJ} \forall t\in [-\TSR,\TSR],\quad 
		|\JJ_1(t)|+|\JJ_2(t)|\leq C \YYzz^{\sigma+1} e^{-\frac 54 \YYzz}.
	\end{equation}
	\item[\rm (ii)]  Equation of $\JJ_j$. For $j=1,2$,
	\begin{equation}\label{eq:JJJ} \forall t\in [-\TSR,\TSR],\quad 
		\left| \frac d{dt}{\cal J}_j(t) - (\mu_j - \dot \yy_j - {\cal N}_j)  \right| \leq  C \YYzz^{\sigma+2} e^{-\frac 74 \YYzz}.
	\end{equation}
\end{enumerate}
\end{lemma}

\begin{remark}\label{rk:J}
The constant $\int Q \Lambda Q$ is not zero (see \eqref{eq:A13}). 

Note also from \eqref{eq:A13} that the $\int \Lambda Q\neq 0$, and so the function $\iscal_j(x)$ is bounded but has a nonzero limit as $x\to +\infty$.
Therefore, $\JJ_j(t)$ is not well-defined for a general $\varepsilon\in H^1$.
Part of the proof of Lemma \ref{PR:J} consists on obtaining decay in space for $\varepsilon(t)$, in order to give sense to $\JJ_1$ and $\JJ_2$.
\end{remark}

\begin{remark}
	Estimate \eqref{eq:JJJ} says formally that $\mu_j - \dot y_j - {\cal N}_j$ is of order $O_{7/4}$, which is 
	a decisive improvement with respect to \eqref{eq:J3} (gain of a factor $e^{-\frac 12 \YYzz}$).
\end{remark}

\begin{proof}
\emph{Preliminary estimates.}
We work under the assumptions of Proposition \ref{pr:st}, and on the interval $[-\TSR,\TSR]$.
First,   we claim without proof exponential uniform decay properties of $U(t)$ on the right ($x>y_1(t)$).

\begin{claim}[Decay estimate on $u(t)$]\label{cl:J1}
	There exist $C>1$ and $\rho_0 <1$ such that for all $t\in [-\TSR,\TSR]$, for all $X_0>1$,  
	\begin{equation}\label{eq:J4}
		\int_{x>X_0+y_1(t)} (U^2+(\partial_x U)^2) (t,x)  dx\leq C e^{- \rho_0 {X_0}}.
	\end{equation}
\end{claim}

Recall that the proof of Claim \ref{cl:J1} is a consequence of  
$$
	\lim_{t\to -\infty} \|U(t)\|_{H^1(x>\frac 12 |t|)}=0
$$
combined with known monotonicity arguments, see e.g. \cite{LM}.

\medskip

\emph{Estimate of $\JJ_j$.}
Note that $\iscal_j$ does not belong to $L^2$ (see Remark \ref{rk:J}) but   satisfies
\begin{equation}\label{eq:J5}
	\sup_{x\in \RR} \left\{\left(1+e^{- \frac 12 (x-\yy_j(t))} \right) | \iscal_j(t,x)|\right\} \leq C.
\end{equation}
It follows from \eqref{eq:J4}, \eqref{eq:J5}, \eqref{eq:U++} and the decomposition of $U(t)$ in Lemma \ref{PR:de} that
\begin{equation*}\begin{split}
	|\JJ_1(t)| & \leq   C \int_{x<y_1(t)} |\varepsilon(t,x)| |\iscal_j(t,x)| dx+  C \int_{y_1(t)\leq x<y_1(t)+10\rho_0^{-1} \YYzz} |\varepsilon(t,x)|  dx\\ &
	+  C \int_{x>y_1(t)+10\rho_0^{-1} \YYzz} |\varepsilon(t,x)| |\iscal_j(t,x)| dx \\
	& \leq C (1+ \YYzz) \|\varepsilon(t)\|_{L^\infty}+ C \int_{x>y_1(t)+10\rho_0 \YYzz} |U(t,x)| + C e^{- 5 \YYzz }
	\\ & \leq C \YYzz^{\sigma+1} e^{-\frac 54 \YYzz}.
\end{split}\end{equation*}
Moreover, using  $\yy_1(t)-\yy_2(t)=\pyy(t)\leq   \YYrr(T) \leq C \YYzz$, one gets  by similar arguments $|\JJ_2(t)|\leq C \YYzz^{\sigma +1} e^{-\frac 54 \YYzz}$.

\medskip

\emph{Equation of $\JJ_1$.}
To prove \eqref{eq:JJJ}, we make use of the equation of $\varepsilon$ (see \eqref{eq:ep}),
and of the special algebraic structure of the approximate solution $V(t,x)$ introduced in Propositions \ref{PR:21} and \ref{PR:22}.
We have 
$$\left(\int Q \Lambda  Q\right) \frac d{dt} \JJ_1(t)=\int (\partial_t \varepsilon )\iscal_1 + \int \varepsilon \partial_t \iscal_1.$$

First observe that
$$
\partial_t  \iscal_1  (x)=  \int_{-\infty}^x \partial_t( \Lambda \qtun)(y) dy
= \int_{-\infty}^x \left\{ \dot \mu_1 \frac {\partial \Lambda \qtun} {\partial \mu_1} 
+ \dot \yy_1 \frac {\partial \Lambda \qtde} {\partial \yy_1}\right\}(y) dy.
$$
Thus, by  $|{\cal M}_j|+|{\cal N}_j|\leq C e^{-\YYzz}$, $|\mu_j(t)|\leq C e^{-\frac 12 \YYzz}$, \eqref{eq:J3} and  \eqref{eq:J4}, arguing as in the proof of \eqref{eq:JJ},
\begin{equation}\label{eq:J6}
	\left|\int \varepsilon \partial_t  \iscal_j   \right|\leq C \YYzz^{\sigma+1} e^{-\frac 74 \YYzz}.
\end{equation}

Next, using \eqref{eq:ep} and $\partial_x \iscal_1= \Lambda \qtun$, we have
\begin{equation*}
\int (\partial_t \varepsilon) \iscal_1 = \int (\partial_x^2 \varepsilon - \varepsilon  + (V +\varepsilon) ^4 - V^4 ) \Lambda \qtun
-\int E \iscal_1  +\int \ETE(V) \iscal_1		.
\end{equation*}

For the first term, i.e. $\int (\partial_x^2 \varepsilon - \varepsilon + (V+\varepsilon)^4 -V^4 ) \Lambda \qtun$, we argue as the proof of Lemma \ref{PR:de}.
Using $L \Lambda Q = - Q$ and  $\int \varepsilon \qtun=0$, 
\eqref{eq:U++} and the definition of $V$ (see Proposition \ref{PR:22}), we obtain
\begin{equation*}
	\left| \int (\partial_x^2 \varepsilon - \varepsilon +(V+\varepsilon)^4 - V^4) \Lambda \qtun\right|
	\leq C \YYzz^2 e^{-\YYzz} \|\varepsilon\|_{L^2}+ C \|\varepsilon\|_{L^2}^2
	\leq C \YYzz^{\sigma+2} e^{-\frac  94\YYzz} .
\end{equation*}

By \eqref{eq:eEb} and \eqref{eq:J5}, we have
\begin{equation*}
	\left| \int E \iscal_1 \right| \leq C \YYzz^\sigma e^{- 2 \YYzz} .
\end{equation*}

Next, we consider the term $\int \ETE(V) \iscal_1$. From the definition of $\ETE(V)$ in \eqref{eq:p22}, the structure of
$V_0$ and $V$, see \eqref{eq:p9} and \eqref{eq:p19} (see also \eqref{eq:de1}), and \eqref{eq:J3}, we have
\begin{equation}\label{eq:J7}
	\sup_{x\in \RR}\bigg\{(1+e^{\frac 12(x-y_1(t))})\Big| \ETE(V) -\sum_{j=1,2} (\mu_j-\dot y_j - {\cal N}_j)  \partial_x \qtud \Big|\bigg\}
	\leq C \YYzz^{\sigma} e^{- 2 \YYzz}.
\end{equation}
Thus, by \eqref{eq:J5}, we obtain
\begin{equation}\label{eq:J8}
	\left| \int \ETE(V) \iscal_1 -  (\mu_1-\dot y_1 - {\cal N}_1) \int (\partial_x \qtun) \iscal_1 \right|
	\leq C \YYzz^\sigma e^{- 2 \YYzz} .
\end{equation}
Finally, using $\left|\int (\partial_x \qtun) \iscal_1 + \int  Q S \right|
\leq C |\mu_1(t)|\leq C e^{-\frac 12 \YYzz} $ and \eqref{eq:J3}, we obtain \eqref{eq:JJJ}.

The proof for $\frac d{dt} \JJ_2$ is exactly the same.
\end{proof}

\subsection{Preliminary symmetry arguments}

First, we claim some additional information on the parameters of the solution $U(t)$, under the assumptions of Proposition \ref{pr:st}.

\begin{claim} For all $t\in [-T,T],$
\begin{equation}\label{eq:J9}
	|\mu_1(t)-\mu_2(-t)|\leq C \YYzz^{\sigma+3} e^{-\frac  54 \YYzz}.
\end{equation}
\end{claim}

\begin{proof}[Proof of \eqref{eq:J9}]
From \eqref{eq:U+-} and $\dot \YYrr(0)=0$, we have
$|\mu_1(0)-\mu_2(0)|\leq C \YYzz^{\sigma+1} e^{-\frac 54 \YYzz}$.
From \eqref{eq:J11}, $|\pyy(t)-\pyy(-t)|\leq |\pyy(t)-\YYrr(t)|+|\YYrr(t)-\pyy(-t)|\leq  C \YYzz^{\sigma+2} e^{-\frac 34 \YYzz}$.
Thus, by \eqref{eq:J3} and the expression of ${\cal M}_j$ in \eqref{eq:p16}, we obtain
\begin{equation*}\begin{split}
&	 \left| \dot \mu_1(t) - \left\{\alpha\,  e^{-\pyy(t)} +  \beta\,  \mu_1(t) \pyy(t) e^{-\pyy(t)} + \delta  \,\mu_1(t) e^{-\pyy(t)}\right\}\right|\leq
C \YYzz^{\sigma} e^{- 2 \YYzz},	\\
&	\left| -\dot \mu_2(-t) - \left\{\alpha\,  e^{-\pyy(-t)} +  \beta\,  \mu_2(-t) \pyy(-t) e^{-\pyy(-t)} + \delta  \,\mu_2(-t) e^{-\pyy(-t)}\right\}\right|\leq
C \YYzz^{\sigma} e^{- 2 \YYzz}, \\
\text{and so}\quad
&	\left| -\dot \mu_2(-t) - \left\{\alpha\,  e^{-\pyy(t)} + \beta\,  \mu_2(-t) \pyy(t) e^{-\pyy(t)} + \delta  \,\mu_2(-t) e^{-\pyy(t)}\right\}\right|\leq
C \YYzz^{\sigma+2 } e^{- \frac 7 4  \YYzz}.
\end{split}\end{equation*}
By  $T\leq C Y_0 e^{\frac 12 T_0}$, it follows that for all $t\in [-\TSR,\TSR]$, $|\mu_1(t)-\mu_2(-t)|\leq C \YYzz^{\sigma+3} e^{-\frac  54 \YYzz}$.
\end{proof}

The next lemma claims that if the asymptotic $2$-soliton solution $U(t)$ considered in Proposition~\ref{pr:st} has an approximate symmetry property (i.e. $U(t,x) -  U(-t + t_0, -x +  x_0) $ is small for some $t_0$, $x_0$) then the corresponding parameters ($\Gamma(t)$ in Proposition \ref{pr:st}) also have some symmetry properties, despite the fact that the decomposition itself is not symmetric (see the definition of $V(t,x)$ in Propositions \ref{PR:21} and \ref{PR:22}). This result relies in particular on parity properties of $A_j$, $B_j$ and $D_j$ and on the choice of orthogonality conditions for $B_{j}$, $D_{j}$ in Section 2.

\begin{lemma}\label{le:id}
	Let $t_0,$ $x_0$ be such that $|t_0|\leq 1$, $|x_0|\leq 1$.
	Under the assumptions of Proposition \ref{pr:st}, for all $t\in [-\TSR,\TSR]$,
	\begin{equation}\label{eq:idd}\begin{split}
		 & |\mu_1(t)-\mu_2(-t+t_0)|+ |y_1(t)+y_2(-t+t_0)-x_0| \\ 
		 & \leq C \| U(t,x) -  U(-t + t_0, -x +  x_0) \|_{H^1} + C\YYzz^5e^{-2\YYzz}.
	\end{split}\end{equation}
	In particular, assume that $U(t,x) =  U(-t + t_0, -x +  x_0)$ for some $t_0,$ $x_0$, then
	\begin{equation}\label{eq:iddrrr} 
	  |\mu_1(t)-\mu_2(-t+t_0)|+ |y_1(t)+y_2(-t+t_0)-x_0|  \leq C \YYzz^5e^{-2\YYzz}.
	 \end{equation}
\end{lemma}

\begin{remark}
	Assuming $\| U(t,x) -  U(-t + t_0, -x +  x_0) \|_{H^1}\leq C \YYzz^{\sigma+3} e^{-\frac 54 \YYzz},$ it follows from \eqref{eq:idd} that the following hold
	\begin{equation}\label{eq:iddd}
		|x_0|\leq C \YYzz^{2} e^{-\frac 12 \YYzz }\quad |t_0|\leq C \YYzz^{\sigma+3} e^{-\frac 14 \YYzz}.
	\end{equation}
	Indeed,
	on the one hand, estimate $|x_0|\leq C \YYzz^2 e^{-\frac 12 \YYzz}$ follows from \eqref{eq:J11} taken at time $t=t_0/2$ and \eqref{eq:idd} taken at $t=t_0/2$.

	On the other hand, from \eqref{eq:J9} and \eqref{eq:idd}, we have $|\mu_1(t)-\mu_1(t-t_0)|\leq C \YYzz^{\sigma+3} e^{-\frac 54 \YYzz}$. Since
	$\dot \mu_1(t)\geq C e^{-\YYzz}$ for $|t|$ close to $0$, we obtain $|t_0|\leq C \YYzz^{\sigma+3} e^{-\frac 14 \YYzz}$.
\end{remark}
\begin{proof}
Set
	\begin{align*}
	& \llb U \rrb(t,x) = U(t,x) {-}  U({-}t {+} t_0, {-}x {+}  x_0),\\
	&	 \llb \qtud \rrb (t,x) = \qtud  (t,x) \psi\left(e^{- \frac 12 \YYzz}x+1\right){-} \qtud  ({-}t{+}t_0,{-}x{+}x_0) \psi\left(e^{- \frac 12 \YYzz}(-x+x_0)+1\right),	 \\
	&	 \llb w_A \rrb(t,x) = w_A(t,x)\psi\left(e^{- \frac 12 \YYzz}x+1\right){-}w_A ({-}t{+}t_0,{-}x{+}x_0)\psi\left(e^{- \frac 12 \YYzz}(-x+x_0)+1\right), \\
	& 	\llb \varepsilon \rrb (t,x) = \varepsilon (t,x) - \varepsilon({-t}{+}t_0,{-}x {+} x_0),	\\
	& 	\llb \mu_j \rrb  (t) = \mu_j(t) {-} \mu_k({-}t{+}t_0),\quad [y_j](t) = y_j(t) {+} y_k({-}t{+}t_0) {-} x_0 \ (j\neq k),
	\end{align*}
and similar definitions for $\llb R_j \rrb$, $\llb w_Q \rrb$, $\llb w_B \rrb$ and $\llb w_D \rrb$,
where $w_A$, $w_Q$, $w_B$ and $w_D$ are defined in the proof of Proposition \ref{PR:21}, so that 
by the expression of $V$ in Proposition \ref{PR:22} and the decomposition $U(t,x)= V(x;\Gamma(t)) + \varepsilon(t,x)$, the following holds
$$
	\llb U  \rrb=   \llb \qtun\rrb   + \llb \qtde\rrb   + \llb w_A \rrb+  \llb w_Q \rrb+ \llb w_B \rrb + \llb w_D \rrb +\llb \varepsilon \rrb.
$$

 We now compute and estimate the various  terms above.
First, by the following elementary claim, for $\mu$ and $y$ small
(see e.g. proof of Claim \ref{CL:21})
$$
	| \{ Q_{1+\mu} ( x -  y ) - Q  ( x ) \} - \{ \mu \Lambda Q (x) - y Q'(x) \} | \leq C ( | \mu |^2 + | y |^2 ) e^{-\frac 9{10} |x|}, 
$$
 the parity of $Q$ and the properties of the cut-off function $\psi$ (see \eqref{eq:18}) we observe that
(for $\YYzz$ large)
\begin{equation}\label{eq:id1}
	|\llb \qtud \rrb - \{\llb \mu_j \rrb \Lambda \qtud - \llb y_j  \rrb\partial_x \qtud \}|
	\leq C e^{-\frac 34 |x-y_j|}(|\llb \mu_j\rrb|^2 + |\llb y_j \rrb|^2  + e^{-100 \YYzz}).
\end{equation}
Note also
\begin{equation}\label{t:Q}
	|\llb   \qud  \rrb-  \llb y_j \rrb\partial_x  \qud \}|\leq C |\llb y_j \rrb|^2  e^{-   |x-y_j|},\quad 
	\|\llb \qud \rrb R_k\|_{H^1}\leq C |\llb y_j \rrb| \YYzz e^{-\YYzz} \quad (j\neq k).
\end{equation}

Then, we note that 
\begin{equation}\label{eq:id2}
	\pyy(t) - \pyy({-}t{+}t_0) = \llb y_1 \rrb - \llb y_2\rrb \quad \text{and so} \quad
	|e^{-\pyy(t)} - e^{-\pyy({-}t{+}t_0)}|\leq C e^{-\pyy(t)} (|\llb y_1 \rrb| + |\llb y_2 \rrb|).
\end{equation}
From Section 2, we know that $A_1(x)=A_2(-x)$ and so
$A_j(-x+x_0-y_j(-t+t_0)) = A_k(x-x_0+y_j(-t+t_0))$ ($k\neq j$). In particular, using also \eqref{eq:id2}, we deduce
\begin{equation}\label{eq:id3}
	\| \llb w_A \rrb\|_{H^1} \leq   C  \sqrt{y} e^{-\pyy(t)} (|\llb y_1\rrb| + |\llb y_2 \rrb|)+ e^{-100 \YYzz}).
\end{equation}

Next,
\begin{align*}
\llb w_Q \rrb	& = {\theta_A} \left\{ (\mu(t){+}\mu({-}t{+}t_0)) x\qun\qde {-} \mu({-}t{+}t_0) \{x\qun\qde {+} ({-}x{+}x_0) (\qun{-} \llb \qun \rrb) (\qde {-}\llb \qde\rrb )\}\right\}	\\ & + O(e^{-100 \YYzz})	 \\
		& = {\theta_A} \{ (\mu(t){+}\mu({-}t{+}t_0)) x\qun\qde  \\ 
		 &  {-} \mu({-}t{+}t_0) ( x_0 \qun\qde {+} ({-}x{+}x_0) (-\qun \llb \qde \rrb
		-\qde\llb \qun\rrb{+} \llb \qun \rrb \llb \qde\rrb)) \} + O(e^{-100 \YYzz})		 .
\end{align*}
Using  $| \mu(t) + \mu(-t+t_0) | \leq |\llb \mu_1 \rrb | + |\llb \mu_2 \rrb|$ and 
\begin{equation*} 
	- x_0  = \frac 12 \llb y_1 \rrb + \frac 12 \llb y_2 \rrb - \frac 12 (y_1+y_2)(-t+t_0) - \frac 12 (y_1(t)+y_2(t)),
\end{equation*}
we obtain  by \eqref{eq:J11}
$$
	| x_0|\leq C (|\llb y_1 \rrb|+|\llb y_2 \rrb| + \YYzz^4 e^{-\frac 12 \YYzz}).
$$
Thus, using $\|\qun \qde\|_{H^1}\leq C \sqrt{y} e^{-y}$ (see \eqref{eq:24}),
 $|\mu(t)|\leq C e^{-\frac 12 \YYzz}$ (see \eqref{eq:U+-}) and \eqref{t:Q}, we obtain
\begin{equation}\label{eq:id4}
	\|\llb w_Q \rrb \|_{H^1}  \leq C   e^{-\frac 34 \YYzz} (|\llb \mu_1 \rrb| + |\llb \mu_2 \rrb| + |\llb y_1 \rrb|+|\llb y_2 \rrb|) + C \YYzz^5 e^{- 2 \YYzz} .
\end{equation}

Summarizing, so far, we have obtained
\begin{equation}\label{eq:ete}\begin{split}
&	\left\| \llb \qtun\rrb + \llb \qtde \rrb + \llb w_A\rrb +\llb w_Q \rrb
	-\left\{ \llb \mu_1\rrb \Lambda \qtun  + \llb \mu_2\rrb \Lambda \qtun 
	-\llb y_1 \rrb \px \qtun -\llb y_2 \rrb \px \qtde\right\} \right\|_{H^1} \\
&	\leq C e^{-\frac 34 Y_0}  (|\llb \mu_1 \rrb| + |\llb \mu_2 \rrb| + |\llb y_1 \rrb|+|\llb y_2 \rrb|) + C \YYzz^5  e^{- 2 \YYzz}.
\end{split}\end{equation}
Now, using orthogonality properties of $B_j$, $D_j$ and $\varepsilon$,
 we claim
\begin{equation}\label{eq:id5}\begin{split}
	& \left|\int \llb w_B \rrb  \qtud \right|+ \left|\int \llb w_B \rrb  \partial_x\qtud \right|
	+\left|\int \llb w_D \rrb  \qtud \right|+ \left|\int \llb w_D \rrb  \partial_x\qtud \right| +\left|\int \llb \varepsilon \rrb  \qtud \right|+ \left|\int \llb \varepsilon \rrb  \partial_x\qtud \right|
	\\& \leq C  e^{-  \YYzz} (|\llb y_1 \rrb|+ |\llb y_2 \rrb|+|\llb \mu_1 \rrb|+|\llb \mu_2 \rrb|)
	+ C Y_0 e^{-2 \YYzz}.
\end{split}\end{equation}

Let us assume \eqref{eq:id5} and finish the proof of \eqref{eq:idd}.
Since
$
\left| \int \llb U\rrb \qtud\right|
+ \left| \int \llb U\rrb \px \qtud \right|
\leq C \| \llb U \rrb|,
$
and
$\int \qtud \Lambda \qtud \geq c_0$, $\int Q Q'=\int Q'\Lambda Q=0$, 
by combining \eqref{eq:ete} and \eqref{eq:id5}, we obtain
$$
|\llb y_1 \rrb|+ |\llb y_2 \rrb|+|\llb \mu_1 \rrb|+|\llb \mu_2 \rrb| 
\leq C \|\llb U\rrb\|_{H^1}+ C   e^{- \frac 34 \YYzz} (|\llb y_1 \rrb|+ |\llb y_2 \rrb|+|\llb \mu_1 \rrb|+|\llb \mu_2 \rrb|)
	+ C\YYzz^5 e^{-2 \YYzz},
$$
and \eqref{eq:idd} follows for $Y_0$ large enough.

\medskip

\emph{Proof of \eqref{eq:id5}.}
Recall that the orthogonality conditions chosen on $B_1$, $B_2$ implies approximate orthogonality conditions on $w_B$, see \eqref{eq:255}.
We deduce
\begin{align*} 
	&\left|\int \llb w_B \rrb   \qtud\right|
	+\left|\int \llb w_B  \rrb \partial_x\qtud \right|
	 \\ &\leq 
	 \|w_B\|_{L^\infty} (\|\qtun-\qun\|_{H^1}+\|\qtde-\qde\|_{H^1})
	 + \|w_B\|_{L^\infty} (\|\llb \qun\rrb\|_{H^1} + \|\llb \qde\rrb\|_{H^1})
	 + C Y_0^\sigma e^{-\frac 52 Y_0}
	\\ & \leq C \|w_B\|_{L^\infty} (e^{-\frac 12 Y_0}+|\llb y_j \rrb|)  
	  + C e^{- 2 Y_0} \leq C Y_0 e^{-\frac 32 \YYzz} (e^{-\frac 12 Y_0}+ |\llb y_j \rrb|) 	  + C e^{- 2 Y_0},
 \end{align*}
 and similary for $w_D$.

Finally, by the orthogonality conditions $\int \varepsilon \qtud=\int \varepsilon  \partial_x \qtud=0$ and \eqref{eq:id1}, we obtain
\begin{equation*}
	\left|\int \llb \varepsilon \rrb   \qtud\right|
	+\left|\int \llb \varepsilon  \rrb \partial_x\qtud \right|
	 \leq C \|\varepsilon\|_{H^1} (|\llb y_j \rrb|+|\llb \mu_j \rrb|)  
 \leq C \YYzz^\sigma e^{-\frac 54 \YYzz} (|\llb y_j \rrb|+|\llb \mu_j \rrb|),
 \end{equation*}
 which finishes the proof of \eqref{eq:id5}.
\end{proof}

\subsection{Nonexistence of a pure $2$-soliton solution}\label{sec:424}

\begin{proposition}\label{PR:rigidity}
	Under the assumptions of Theorem \ref{TH:1}, the solution $U(t)$ of {\eqref{eq:KDV}} satisfying \eqref{eq:pur2} is not an asymptotic $2$-soliton solution at $+\infty$. Equivalently, 
	\begin{equation}\label{eq:ri1}
		\lim_{t\to +\infty} \|w(t)\|_{H^1(\mathbb{R})} \neq 0.
	\end{equation}
\end{proposition}

In Section \ref{sec:5.5}, we shall prove a statement stronger than Proposition \ref{PR:rigidity}.
However, we give a direct proof for Proposition \ref{PR:rigidity} at this point because of its own interest and simplicity.
Indeed, the nonexistence of a global $2$-soliton solution  can be seen as a rigidity property, 
and its proof can now  be obtained by a  simple symmetry argument. The proof of the lower bound stated in Section \ref{sec:5.5} is based on a similar symmetry argument but also requires   the sharper stability result proved in Proposition \ref{PR:SHARP}.

\begin{proof}[Proof of Proposition \ref{PR:rigidity}]
The proof is by contradiction.
We assume that $U(t)$ is a pure $2$-soliton solution at $+\infty$.

\medskip

\emph{Step 1.}
By Lemma \ref{le:am} and the uniqueness of asymptotic pure $2$-soliton solutions
(see \cite{Ma2}, Theorem 1), 
 $\mu_1^+=\mu_0$, $\mu_2^+=-\mu_0$ and there exist $t_0$, $x_0\in \RR$ such that, for all $t$, $x\in \RR$,
	\begin{equation}\label{eq:smm}
		U(t,x) = U(-t + t_0, -x +  x_0).
	\end{equation}
Moreover, by Proposition \ref{pr:st}, $x_0$ and $t_0$ are small.

We denote
\begin{equation}\label{eq:c1}
	 \nu(t)=\mu_1(t)-\mu_2(-t+t_0),\quad \pzz(t)=y_1(t)+y_2(-t+t_0)-x_0.
\end{equation}
By \eqref{eq:smm} and Lemma \ref{le:id}, we have $\forall t\in [-\TSR,\TSR]$,
\begin{equation}\label{eq:c2}
	| \nu(t)|+ | \pzz(t)|\leq C \YYzz^5 e^{-2\YYzz}.  
\end{equation}
In Step 2, we see that \eqref{eq:c2} combined with Lemma \ref{PR:J} provides a contradiction.

\medskip

\emph{Step 2.}
We claim
\begin{align} 
	& 	|\JJ_1(t)|+|\JJ_2(t)|\leq C \YYzz^{\sigma+1} e^{-\frac 54 \YYzz},		\label{eq:stp2bis}
\\	& \left| \frac d{dt} \left( \pzz(t) -  \{(b_- \pyy(t)+(b_-+d_-)\} e^{-\pyy(t)} \right) +   (\dot \JJ_1(t) - \dot \JJ_2(-t+t_0))\right|
	\leq C \YYzz^{\sigma+2} e^{-\frac 74 \YYzz}.	\label{eq:stp2}
\end{align}

We postpone the proof of \eqref{eq:stp2bis}--\eqref{eq:stp2} and we obtain a contradiction.
By integration of  \eqref{eq:stp2}, \eqref{eq:stp2bis} and \eqref{eq:cH}, we obtain, for any $t_1,t_2\in [-\TSR,\TSR]$,
\begin{equation}\label{eq:c6}\begin{split}
&	\left| (\pzz(t_1)- \{b_- \pyy(t_1) + (b_-+d_-)\} e^{-\pyy(t_1)}) - (\pzz(t_2)- \{b_- \pyy(t_2) + (b_-+d_-)\} e^{-\pyy(t_2)}) \right| 
\\ & \leq 	C \YYzz^{\sigma + 3 } e^{-\frac 54 \YYzz}.
\end{split}\end{equation}
Thus, by \eqref{eq:c2}, for any $t_1,t_2\in [-\TSR,\TSR]$, for $k= 1 + \frac {d_-}{b_-}$, ($b_-\neq 0$),
\begin{equation}\label{eq:c7}
	|(\pyy(t_1)+k) e^{-\pyy(t_1)} - (\pyy(t_2)+k) e^{-\pyy(t_2)}|\leq C \YYzz^{\sigma +3} e^{-\frac 54 \YYzz}.
\end{equation}
But taking $0<t_1<t_2<\TSR$ such that $\pyy(t_1)=\YYzz+1$, $\pyy(t_2)=\YYzz+2$ and then $\YYzz$ large enough, \eqref{eq:c7} implies
$$
	\left|Y_0 e^{-(Y_0+1)} - Y_0 e^{-(Y_0+2)}\right| \leq Ce^{-Y_0},
$$
which is a contradiction for $Y_0$ large enough.

\medskip

Now, we prove \eqref{eq:stp2bis}--\eqref{eq:stp2}. 
Estimate  \eqref{eq:stp2bis} is exactly \eqref{eq:JJ}.
Next, by \eqref{eq:JJJ} and the expression of ${\cal N}_j$ in \eqref{eq:p16}, we have
\begin{equation}\label{eq:c3}\begin{split}
	&	\dot y_1 = \mu_1 - a e^{-\pyy} - b_1 \mu_1 \pyy e^{-\pyy} - d_1 \mu_1 e^{-\pyy} - \dot \JJ_1 +O(\YYzz^{\sigma+2} e^{-\frac 74 \YYzz}) ,	\\
	&	\dot y_2 = \mu_2 - a e^{-\pyy} - b_2 \mu_2 \pyy e^{-\pyy} - d_2 \mu_2 e^{-\pyy} - \dot \JJ_2 +O(\YYzz^{\sigma+2} e^{-\frac 74 \YYzz}).
\end{split}\end{equation}

Moreover, 
\begin{equation*}\begin{split}
	\pyy(t)-\pyy(-t+t_0) & = (y_1(t)-y_2(t))- (y_1(-t+t_0)-y_2(-t+t_0))
	\\ &= (y_1(t)+y_2(-t+t_0)-x_0) - (y_2(t)+y_1(-t+t_0)-x_0),
\end{split}\end{equation*}
and so by \eqref{eq:c2}, $|\pyy(t) - \pyy(-t+t_0)|\leq C \YYzz^5 e^{-2\YYzz}$, so that 
\begin{equation}\label{eq:c4}
	\left |e^{-\pyy(t)}-e^{-\pyy(-t+t_0)}\right|\leq C \YYzz^5 e^{-3 \YYzz}.
\end{equation}
By \eqref{eq:c3} and \eqref{eq:c4}, we have
\begin{equation*}\begin{split}
	\dot{\pzz}(t) &= \dot y_1(t) -  \dot y_2(-t+t_0)
	\\ &=  \nu(t) - (b_1-b_2) \mu_1 \pyy e^{-\pyy} - (d_1-d_2) \mu_1 e^{-\pyy} - (\dot \JJ_1(t)- \dot \JJ_2(-t+t_0)) + O(\YYzz^{\sigma+2} e^{-\frac 74 \YYzz}),
\end{split}\end{equation*}
Let (see \eqref{eq:fg})
$$ b_-= - \frac 12 (b_1-b_2) \neq 0,\qquad  d_-= - \frac 12 (d_1-d_2).$$
Since $|\mu_1 + \mu_2 |\leq C \YYzz e^{-\YYzz}$ (see \eqref{eq:U+-}), we have
$|\mu_1 -\frac 12 \mu|\leq C \YYzz e^{-\YYzz}$ and thus, by $|\mu  - \dot \pyy|\leq C e^{-\YYzz}$ (see \eqref{eq:sd2}), we obtain
$|\mu_1 e^{-\pyy} - \frac 12 \dot \pyy e^{-\pyy}|\leq C \YYzz e^{-2 \YYzz}.$

We obtain
$$
\dot {\pzz}=  \nu + b_- \dot \pyy \pyy e^{-\pyy} + d_- \dot \pyy e^{-\pyy} -(\dot \JJ_1(t)- \dot \JJ_2(-t+t_-)) + O(\YYzz^{\sigma+2} e^{-\frac 74 \YYzz}),
$$
where  $| \nu(t)|\leq C \YYzz^5 e^{-2 \YYzz}$ from \eqref{eq:c2}.
Thus, by elementary computations, we now obtain \eqref{eq:stp2}.
\end{proof}

\subsection{Lower bound on the defect}\label{sec:5.5}

\begin{proposition}\label{PR:qualit}
	Under the assumptions of Theorem \ref{TH:1}, there exists $c>0$ such that,   
	\begin{equation}\label{eq:lower1} 
		\liminf_{t\to +\infty}  \|w(t)\|_{H^1} \geq c \YYzz e^{-\frac 32 \YYzz} ,
	\end{equation}
	\begin{equation}\label{eq:lower2}
		\mu_1^+ - \mu_0\geq c \YYzz^2 e^{-\frac 52\YYzz} ,\quad
		-\mu_2^+ - \mu_0\geq  c \YYzz^2  e^{-\frac 52 \YYzz}.
	\end{equation}
\end{proposition}

\begin{proof}
It suffices to  prove \eqref{eq:lower1}, since estimate \eqref{eq:lower2} then follows from Lemma \ref{le:am}.
Let $\epsilon>0$ arbitrary, and suppose for the sake of contradiction that\begin{equation}\label{eq:def2}
	\liminf_{t\to +\infty}	\| w(t)\|_{H^1} \leq   \epsilon \YYzz e^{ - \frac 32 \YYzz}.
\end{equation}

\emph{Step 1.}
First, we claim that there exist $\tilde T(t)$, $\tilde X(t)$ such that, for all $t\in \RR$,
\begin{equation}\label{eq:def5}
	\|U(t,x)-U(-t+\tilde T(t),-x+\tilde X(t))\|_{H^1} 
	+ |\dot {\tilde X}(t)|+ e^{-\frac 12 \YYzz} | \dot {\tilde  T}(t)| \leq C  \epsilon \YYzz  e^{-\frac 32 \YYzz}
\end{equation}\emph{Proof of \eqref{eq:def5}.}
From \eqref{eq:def2}, there exists $\mathcal{T}_1$ arbitrarily large with
$\|w(\mathcal{T}_1)\|_{H^1}\leq 2 \epsilon Y_0 e^{-\frac 32 Y_0}$.

By Lemma \ref{le:am}, it also follows from \eqref{eq:def2} that
$$
0\leq \mu_1^+ -\mu_0 \leq C \epsilon^2 \YYzz^2 e^{- \frac 52 \YYzz},\quad
0\leq -(\mu_2^+ + \mu_0) \leq C \epsilon^2  \YYzz^2 e^{- \frac 52 \YYzz}.
$$
In particular, for all $t$ 
$$
\bigg\| \sud Q_{1+\mu_j^+}(.-y_j(t))
- (Q_{1+\mu_0}(.-y_1(t)) + Q_{1-\mu_0}(.-y_2(t)))\bigg\|_{H^1}\leq C \epsilon^2 \YYzz^2 e^{- \frac 52 \YYzz}.
$$
From the behavior of $U(t)$ as $t\to -\infty$, and  the information above,
there exists $\mathcal{T}_2>T$ and $X\in \RR$ such that
\begin{equation}\label{eq:def3}
	\|U(\mathcal{T}_1,x)-U(-\mathcal{T}_2,-x+X)\|_{H^1} \leq  2 \epsilon Y_0e^{ - \frac 32 \YYzz} +C \epsilon^2 Y_0^2 e^{-\frac 52 \YYzz}
	\leq 3  \epsilon \YYzz e^{ - \frac 32 \YYzz},
\end{equation}
taking $\YYzz$ large enough.
From  Proposition \ref{PR:SHARP} (sharp global stability of the $2$-soliton structure), it follows that there exist $\tilde T(t)$ and $\tilde X(t)$ such that \eqref{eq:def5} follows.

\medskip

\emph{Step 2.}  {Conclusion of the proof of Proposition \ref{PR:qualit}.}
As in the proof of Proposition \ref{PR:rigidity}, take $0<t_1<t_2$ such that $\YYrr(t_1) = \YYzz+1$ and
$\YYrr(t_2)=\YYzz+2$. Note that $t_2-t_1< C e^{\frac 12 \YYzz}$ by $\dot \YYrr(t) > c_0 e^{-\frac 12 \YYzz}$ for $c_0>0$ on $[t_1,t_2]$.

Note that for $t\in [-T,T]$, $\tilde T(t)$ and $\tilde X(t)$ are small.
Applying Lemma \ref{le:id}, for all $t\in [t_1,t_2]$, we obtain (for $\YYzz$ large enough depending on $\epsilon$)
\begin{equation*} |\mu_1(t)-\mu_2(-t+\tilde T(t))|+ | y_1(t)+y_2(-t+\tilde T(t)) - \tilde X(t)|\leq C  \epsilon e^{-\frac 32 \YYzz}.
\end{equation*}
By \eqref{eq:def5}, for all $t\in [t_1,t_2]$, we have $|\tilde T(t)-\tilde T(t_1)|\leq C \epsilon Y_0e^{-\frac 12 \YYzz}$, $|\tilde X(t)-\tilde X(t_1)|\leq
C \epsilon Y_0 e^{- \YYzz}$ and thus,
\begin{equation*}\begin{split}
 \forall t\in [t_1,t_2],\quad
	& | \mu_2(-t+\tilde T(t_1))- \mu_2(-t+\tilde T(t))|\leq C \epsilon Y_0 e^{-\frac 32 \YYzz} ,\\
	& | y_2(-t+\tilde T(t_1)) - \tilde X(t_1) - ( y_2(-t+\tilde T(t))- \tilde X(t)) |\leq C \epsilon Y_0 e^{-  \YYzz} .
\end{split}\end{equation*}
Therefore, setting 
$$
	\nu(t) = \mu_1(t) - \mu_2(-t+\tilde T(t_1)), \quad
	z(t) = y_1(t) + y_2(-t+\tilde T(t_1)) - \tilde X(t_1),
$$
we obtain
\begin{equation}\label{eq:def6}
	|\nu(t)|\leq C \epsilon \YYzz e^{-\frac 32 \YYzz}, \quad |z(t)|\leq C \epsilon e^{-  \YYzz}.
\end{equation}

Now, we obtain a contradiction following the strategy of the proof of Proposition \ref{PR:rigidity}.
Arguing as in the proof of \eqref{eq:stp2bis} and \eqref{eq:stp2}, using \eqref{eq:def6} we get
\begin{align} 
	& 	|\JJ_1(t)|+|\JJ_2(t)|\leq C \YYzz^{\sigma+1} e^{-\frac 54 \YYzz},		\label{eq:stp3bis}
\\	& \left| \frac d{dt} \left( \pzz(\{b_- \pyy(t)+(b_-+d_-)\} e^{-\pyy(t)} \right) +   (\dot \JJ_1(t) - \dot \JJ_2(-t+\tilde T(t_1)))\right|
	\leq C \epsilon \YYzz e^{-\frac 32 \YYzz}.	\label{eq:stp3}
\end{align}
Integrating  \eqref{eq:stp3} on $[t_1,t_2]$ using $t_2-t_1< C e^{\frac 12 \YYzz}$,  and then using \eqref{eq:stp3bis} and  we obtain
\begin{equation}\label{eq:c16}\begin{split}
&	\left| (\pzz(t_1)- \{b_- \pyy(t_1) + (b_-+d_-)\} e^{-\pyy(t_1)}) - (\pzz(t_2)- \{b_- \pyy(t_2) + (b_-+d_-)\} e^{-\pyy(t_2)}) \right| 
\\ & \leq 	C \epsilon \YYzz  e^{-\YYzz}.
\end{split}\end{equation}
Thus, by \eqref{eq:def6}, for $k= 1 + \frac {d_-}{b_-}$ ($b_-\neq 0$),
\begin{equation}\label{eq:c17}
	|(\pyy(t_1)+k) e^{-\pyy(t_1)} - (\pyy(t_2)+k) e^{-\pyy(t_2)}|\leq C \epsilon \YYzz  e^{-\YYzz}.
\end{equation}
But using $\YYrr(t_1) = \YYzz+1$ and
$\YYrr(t_2)=\YYzz+2$,  \eqref{eq:c17} is a contradiction for $\epsilon$ small enough and $\YYzz$ large enough.
\end{proof}

\appendix

\section{Preliminary results on solitons}

\subsection{Linearized operator, identities and asymptotics}
Recall
\begin{equation}\label{eq:A1}
Q_c(x)= c^{1/3} Q\left(\sqrt{c} x \right),\quad
Q_c''+Q_c^4=c\, Q_c,\quad \text{for $c>0$},
\end{equation}
where
\begin{equation}\label{eq:A2}
Q(x)= \left(\frac 52\right)^{1/3} \cosh^{-2/3} \left(\frac 32 x\right)
\end{equation}
solves
\begin{equation}\label{eq:A3}
Q''+Q^4 = Q\quad \text{and} \quad (Q')^2 + \frac 25 Q^5 = Q^2
\quad \text{on $\RR$.}
\end{equation}

\begin{claim}[Properties of the linearized operator $L$]\label{CL:A1}
The operator $L$ defined in $L^2(\RR)$ by
    \begin{equation}\label{eq:A4}
        L f= -f''+f-4Q^{3} f
    \end{equation}
    is self-adjoint and satisfies the following properties:
    \begin{itemize}
        \item[{\rm (i)}] First eigenfunction : $L Q^{\frac {5} 2} = - \frac {21}4 Q^{\frac {5} 2}$;
        \item[{\rm (ii)}] Second eigenfunction : $L Q'=0$; the kernel of $L$ is 
        $\{\lambda Q', \lambda \in \RR\}$;
        \item[{\rm (iii)}] For any   function $h \in L^2(\RR)$ orthogonal to $Q'$ for the $L^2$ scalar product, 
        there exists a unique function $f \in H^2(\RR)$ orthogonal to $Q'$ such that $L f=h$; moreover,
        if $h$ is even (respectively, odd), then $f$ is even (respectively, odd).
       \item[{\rm (iv)}] Suppose that $f\in H^2(\RR)$ is such that $L f \in \mathcal{Y}$, then $f\in \mathcal{Y}$.
       \item[{\rm (v)}] There exists $\lambda>0$ such that for all $f\in H^1(\RR)$,
       \begin{equation}\label{eq:A5}
       \int fQ=\int fQ'=0 \ \Rightarrow \ (Lf,f)\leq \lambda \|f\|_{H^1}^2.
       \end{equation}
    \end{itemize}
\end{claim}
These properties of the linearized operator are standard
(see e.g. \cite{MMcol1}, Lemma 2.2).
See  \cite{We2} for property (v).

\begin{claim}\label{CL:A2}
	\begin{itemize}
        \item[{\rm (i)}] Scaling.
        Let 
        \begin{equation}\label{eq:A6}
        		\Lambda Q_c=\left(\frac d{dc'} Q_{c'}\right)_{|c'=c},\quad
			\Lambda^2 Q_c=\left(\frac {d^2}{dc'^2} Q_{c'}\right)_{|c'=c}.
        \end{equation}
        Then,
        \begin{equation}\label{eq:A7}
        		\Lambda Q_c = \frac 1 c \left(\frac 13 Q_c + \frac 12 xQ_c'\right),\quad \Lambda^2 Q_c = \frac 1{c^2} \left( -\frac 29 Q_c + \frac 1{12} x Q_c' + \frac 14 x^2 Q_c''\right),
		\end{equation}
		\begin{equation}\label{eq:A8}
			\Lambda Q = 	\Lambda Q_1 = \frac 13 Q + \frac 12 xQ', \quad 
			\Lambda^2 Q = \Lambda^2 Q_1 = -\frac 29 Q + \frac 1{12} x Q' + \frac 14 x^2 Q''.
        \end{equation}
        \item[{\rm (ii)}] Some explicit antecedents for $L$.
        \begin{equation}\label{eq:A9}
        		L Q = -3 Q^4, \quad L (\Lambda Q) = -Q, \quad L (Q')=0,
		\end{equation}
		\begin{equation}\label{eq:A10}
			\left(\frac {Q'} Q\right)'= - \frac 35 Q^3,\quad
			\lim_{\pm \infty} \frac {Q'} Q = \mp 1,
		\end{equation}
		\begin{equation}\label{eq:A11}
			L\left( \frac {Q'} Q\right)
				= \frac {Q'} Q - \frac {11} 5 Q^2 Q',\quad
			\left( L\left( \frac {Q'} Q\right) \right)' 
				= -\frac {36}5 Q^3 + \frac {99}{25} Q^6.
		\end{equation}
        \item[{\rm (iii)}] Integral identities. For $r\geq 1$, $c>0$,
        \begin{equation}\label{eq:A12}
        		\int Q^{r+3} = \frac {5r} {2r+3} \int Q^r, \quad \int (Q')^2 = \frac 37 \int Q^2,
        \end{equation}
        \begin{equation}\label{eq:A13}
        		\int \Lambda Q = -\frac 16 \int Q,\quad \int Q^3 \Lambda Q = \frac 5{24} \int Q,
			\quad \int Q (\Lambda Q) = \frac 16 \int Q^2,
        \end{equation}
        \begin{equation}\label{eq:A13b}
        \frac 75 \int e^{-x} Q^5 = \frac 32 \int e^{-x} Q^2,\quad  
        		\int e^{-x} Q^4 =  2 (10)^{1/3},\quad \int Q^3= \frac {10}3,
        \end{equation}
        \begin{equation}\label{eq:A14}
        		\int Q_c^2 = c^{1/6} \int Q^2,\quad
			\mathcal{E}(Q_c) = c^{7/6} \mathcal{E}(Q),\quad \mathcal{E}(Q)= -\frac 1{7} \int Q^2,
        \end{equation}
        \begin{equation}\label{eq:A15}
			-\frac d{dc} \mathcal{E}(Q_c) =   c  \frac d{dc} \int Q_c^2 >0.        \end{equation}
        \item[{\rm (iv)}] Asymptotics as $x\to +\infty$
        \begin{equation}\label{eq:A16}
        		Q(x) =  (10)^{1/3}  e^{-x}   + O\left(e^{-4x}\right),  \quad
		Q'(x) = - (10)^{1/3}  e^{-x}   + O\left(e^{-4x}\right),       \end{equation}
        \begin{equation}\label{eq:A17}
        		\Lambda Q (x) = (10)^{1/3} \left(\frac 13 - \frac x2\right) e^{-x} + O\left(x e^{-4x}\right) .         \end{equation}
		Let 
		\begin{equation}\label{eq:A18}
			P= \frac {Q'}Q - 1 + 2 (10)^{-1/3} e^{x} Q .
		\end{equation}
		Then,
		\begin{equation}\label{eq:A19}
			\forall x\in \RR, \quad |P(x)|\leq C e^{-2|x|}.
		\end{equation}
    \end{itemize}
\end{claim}
\begin{proof}
These results are easily obtained for the equation of $Q$. See e.g.
\cite{MMcol1}, Appendix C.1 and Claim 2.1.

We only prove \eqref{eq:A19} concerning the function $P$.
First, since $\lim_{\pm \infty}\frac {Q'}{Q}=\mp 1$ and $\lim_{+\infty} 2 (10)^{-1/3} e^{x} Q =2$,
$\lim_{-\infty} 2 (10)^{-1/3} e^{x} Q =0$, we have $\lim_{\pm \infty} P = 0$.
Moreover, we have from \eqref{eq:A10}, 
$$
	P'=-\frac 35 Q^3 + 2 (10)^{-1/3} (e^x Q)'
$$
and using the explicit expression of $Q$, we find $|(e^x Q)'|\leq Ce^{-3 x}$ for $x>0$
and $|(e^x Q)'|\leq Ce^{2 x}$ for $x<0$. This proves \eqref{eq:A19} by integration.
\end{proof}
\subsection{Approximate antecedent of $\qun\qde$}

\begin{claim}\label{CL:A5}
Let  $\xun=x-y_1$ and  $\xde=x-y_2$. Then
	\begin{align*}
		 & \px \left( - \px^2 ((\xun+\xde)\qun\qde)  + (\xun+\xde)\qun\qde - 4 (\qun^3 + \qde^3) (\xun+\xde)\qun\qde\right)  \\
		 &=2 \qun \qde  +  3 y \left (\qun-\px \qun -  \px(\qun^4) \right )\qde
		  - y \qun^4 \px \qde
		 \\ 
		 & + 3 y  (\qde+\px \qde + \px(\qde^4)  )\qun
		 + y \qde^4 \px \qun
		 \\ & 
		  +(-\qun^4 - 6 \xun \px(\qun^4)) \qde 
		 - 2 \xun \qun^4 \px \qde 	
		  -\qun^4  \qde 	
		 \\
		 & +6 \xun (\qun-\px \qun)\qde+ 6 (\qun-\px \qun) (\px \qde)- 6 (\qun - \px \qun) \qde 
		  \\&+(-\qde^4 - 6 \xde \px(\qde^4)) \qun 
		 - 2 \xde \qde^4 \px \qun 
		    -  \qde^4 \qun  	 
\\&		 - 6 \xde(\qde+\px \qde) \qun  
		 - 6 (\px \qde + \qde) (\px \qun)- 6  (\qde+\px \qde) \qun .
\end{align*}
\end{claim}
\begin{proof} 	We start with a formula for two general functions $f_1$, $f_2$,
	\begin{align*}
		 & \px \left( - \px^2 (f_1(\xun)f_2(\xde))  + f_1(\xun)f_2(\xde) - 4 (\qun^3 + \qde^3) (f_1(\xun)f_2(\xde))\right)  \\
		 & =\px \big[ f_2(\xde) (Lf_1)(\xun) + f_1(\xun) (Lf_2)(\xde)
- 2   f_1'(\xun)   f_2'(\xde) - f_1(\xun)f_2(\xde)\big]\\
		 & = f_2(\xde) (Lf_1)'(\xun) + f_1(\xun) (Lf_2)'(\xde) \\
		 &+ (Lf_1 - f_1 - 2 f_1'')(\xun) f_2'(\xde) +
		 (Lf_2 - f_2 - 2 f_2'')(\xde) f_1'(\xun).
	\end{align*}
	
	We apply this formula with $f_1=xQ$ and $f_2=Q$. Note that using \eqref{eq:A9}, we have
	\begin{align*}
		& LQ = - 3Q^4,\quad LQ-Q-2Q'' = -3 Q -Q^4, \\
		& L(xQ)= xLQ - 2 Q' = -3 x Q^4 - 2 Q',\\
		& L(xQ) -xQ - 2(xQ)'' = - x(3 Q + Q^4) - 6 Q'.
	\end{align*}
	Therefore, one gets
	\begin{align*}
		 & \px \left( - \px^2 (\xun\qun\qde)  + \xun\qun\qde - 4 (\qun^3 + \qde^3) \xun\qun\qde\right)  \\
		 &= (-\qun^4 - 3 \xun \px(\qun^4)) \qde 
		 -\xun \qun^4 \px \qde - 3\xun \qun \px(\qde^4) - (\qun +\xun \px (\qun)) \qde^4\\
		 & - 3 \xun (\px\qun) \qde
		 - 3 \xun \qun (\px\qde) - 6(\px \qun)(\px \qde) - 5 \qun \qde.
	\end{align*}
Similarly, 
\begin{align*}
		 & \px \left( - \px^2 (\xde\qde\qun)  + \xde\qde\qun - 4 (\qun^3 + \qde^3) \xde\qde\qun\right)  \\
		 &= (-\qde^4 - 3 \xde \px(\qde^4)) \qun 
		 -\xde \qde^4 \px \qun - 3\xde \qde \px(\qun^4) - (\qde +\xde \px (\qde)) \qun^4\\
		 & - 3 \xde( \px\qde) \qun
		 - 3 \xde \qde (\px\qun) - 6(\px \qde)(\px \qun) - 5 \qde \qun.
\end{align*}
Thus,
\begin{align*}
		 & \px \left( - \px^2 ((\xun+\xde)\qun\qde)  + (\xun+\xde)\qun\qde - 4 (\qun^3 + \qde^3) (\xun+\xde)\qun\qde\right)  \\
		 &= (-\qun^4 - 3 \xun \px(\qun^4)) \qde 
		 -\xun \qun^4 \px \qde - 3\xun \qun \px(\qde^4) - (\qun +\xun \px (\qun)) \qde^4	 
		 \\&+(-\qde^4 - 3 \xde \px(\qde^4)) \qun 
		 -\xde \qde^4 \px \qun - 3\xde \qde \px(\qun^4) - (\qde +\xde \px (\qde)) \qun^4\\
		  &  - 12(\px \qun)(\px \qde) - 10 \qun \qde\\
				 & - 3 \xun (\px\qun) \qde
		 - 3 \xun \qun (\px\qde) - 3 \xde (\px\qde) \qun
		 - 3 \xde \qde (\px\qun).
\end{align*}

First, the terms in the last line are handled as follows (recall that $\xde-\xun=y$)
\begin{align*}
		& - 3 \xun (\px\qun) \qde
		 - 3 \xun \qun (\px\qde) - 3 \xde (\px\qde) \qun
		 - 3 \xde \qde (\px\qun) \\
		 & = 3 \xun (\qun-\px \qun)\qde - 3 \xun \qun (\qde+\px \qde)
		 - 3 \xde (  \qde+\px \qde) \qun + 3 \xde \qde (\qun-\px\qun)\\
		 &= 6 \xun (\qun-\px \qun)\qde + 3 y  (\qun-\px \qun)\qde 
		 - 6 \xde(\qde+\px \qde) \qun + 3 y (\qde+\px \qde) \qun .
\end{align*}
Second,
\begin{align*}
	& -12(\px \qun)(\px \qde) \\
	& = 
	6 (\qun-\px \qun) (\px \qde) - 6 (\px \qde + \qde) (\px \qun)
	- 6 \qun (\px \qde) + 6 \qde (\px \qun)\\
	& =6 (\qun-\px \qun) (\px \qde) - 6 (\px \qde + \qde) (\px \qun)		\\ &- 6 \qun (\qde+\px \qde) - 6 \qde (\qun - \px \qun)
	+12 \qun\qde.
\end{align*}
Gathering these computations, we obtain
\begin{align*}
		 & \px \left( - \px^2 ((\xun+\xde)\qun\qde)  + (\xun+\xde)\qun\qde - 4 (\qun^3 + \qde^3) (\xun+\xde)\qun\qde\right)  \\
		 &=2 \qun \qde+ 3 y  (\qun-\px \qun)\qde + 3 y (\qde+\px \qde) \qun\\&
		  +(-\qun^4 - 3 \xun \px(\qun^4)) \qde 
		 -\xun \qun^4 \px \qde 	
		 - 3\px(\qun^4) \xde \qde  -\qun^4 (\qde +\xde \px (\qde)) 	
		 \\
		 & +6 \xun (\qun-\px \qun)\qde+ 6 (\qun-\px \qun) (\px \qde)- 6 (\qun - \px \qun) \qde 
		  \\&+(-\qde^4 - 3 \xde \px(\qde^4)) \qun 
		 -\xde \qde^4 \px \qun 
		 - 3\px(\qde^4) \xun \qun  -  \qde^4(\qun +\xun \px (\qun)) 	 
\\&		 - 6 \xde(\qde+\px \qde) \qun  
		 - 6 (\px \qde + \qde) (\px \qun)- 6  (\qde+\px \qde) \qun 
		 \\
		 &=2 \qun \qde  +  3 y \left (\qun-\px \qun -  \px(\qun^4) \right )\qde
		  - y \qun^4 \px \qde
		 \\ 
		 & + 3 y  (\qde+\px \qde + \px(\qde^4)  )\qun
		 + y \qde^4 \px \qun
		 \\ & 
		  +(-\qun^4 - 6 \xun \px(\qun^4)) \qde 
		 - 2 \xun \qun^4 \px \qde 	
		  -\qun^4  \qde 	
		 \\
		 & +6 \xun (\qun-\px \qun)\qde+ 6 (\qun-\px \qun) (\px \qde)- 6 (\qun - \px \qun) \qde 
		  \\&+(-\qde^4 - 6 \xde \px(\qde^4)) \qun 
		 - 2 \xde \qde^4 \px \qun 
		    -  \qde^4 \qun  	 
\\&		 - 6 \xde(\qde+\px \qde) \qun  
		 - 6 (\px \qde + \qde) (\px \qun)- 6  (\qde+\px \qde) \qun,
\end{align*}
which  finishes the proof of Claim \ref{CL:A5}.
\end{proof}

\section{Modulation and monotonicity arguments}\label{AP:B}
 
\subsection{Proof of Lemma \ref{PR:de}}
Let 
$$
	\mathcal{V}(\omega_0,y_0) = \{ u \in H^1(\RR);
	\inf_{y_1-y_2>y_0}\| u - V(.;(0,0,y_1,y_2))\|_{H^1} \leq \omega_0\},
$$
where $V(x;\Gamma)$ is defined in Proposition \ref{PR:22}.

\begin{claim}[Time independent modulation]\label{le:huit}
There exist $\omega_0$, $\bar y_0>0$ and a unique $C^1$
map $\Gamma=(\mu_1,\mu_2,y_1,y_2): \mathcal{V}(\omega_0,\bar y_0)\to (0,\infty)^2\times \RR^2$
such that if $u\in \mathcal{V}(\omega,y_0)$ for $0<\omega\leq \omega_0$,
$y_0\geq \bar y_0$ and 
$$
	\varepsilon (x) = u(x) - V(x;\Gamma),
$$
then, for $j=1,2$,
$$
	\int \varepsilon  Q_{1+\mu_j}(.-y_j) =
	\int \varepsilon  Q_{1+\mu_j}'(.-y_j) = 0
$$
and
$$
	y_1-y_2 >y_0 - C \omega ,\quad 
	\|\varepsilon\|_{H^1} + |\mu_1 |+|\mu_2 |\leq C \omega .
$$
\end{claim}

\begin{proof}
	The proof, based on the implicit function theorem, is  similar to the one of Lemma 8 in \cite{MMT}, the only difference being that we perform modulation  around the map
$(\mu_1,\mu_2,y_1,y_2)\mapsto V(x;(\mu_1,\mu_2,y_1,y_2)) $ instead of the family of sums of two solitons. By the properties of $V$ (see  \eqref{eq:p20} and \eqref{eq:de1} below) and \eqref{eq:A13}, the nondegeneracy conditions are the same as in \cite{MMT}.
\end{proof}

The existence, uniqueness and continuity of $\Gamma(t)$ is a consequence of Claim \ref{le:huit} applied to $u(t)$ for all $t\in I$.

The equation of $\varepsilon(t)$ is easily deduced from {\eqref{eq:KDV}} and \eqref{eq:p21}.
Next, we prove the estimates on $\dot \Gamma(t)$, i.e. \eqref{eq:ga}, omitting standard regularization arguments to justify the computations. First, we expand $\frac d{dt}\int \varepsilon \qtun$. Using   \eqref{eq:ep}, we obtain 
\begin{equation*}
\begin{split}
	 0 = \frac d{dt} \int \varepsilon \qtun
	  &= \int \varepsilon   \partial_t \qtun + \int (\partial_x^2 \varepsilon - \varepsilon + 4 V^3\varepsilon) \partial_x \qtun  + \int \left((V+\varepsilon)^4 - V^4 -4 V^3 \varepsilon \right) \partial_x  \qtun
	   \\
	&  - \int E  \qtun - (\dot \mu_1 - {\cal M}_1 ) \int  \frac {\partial V}{\partial \mu_1}  \qtun -
	(\dot \mu_2 - {\cal M}_2) \int \frac {\partial V}{\partial \mu_2}  \qtun \\
	& + (\mu_1 - \dot y_1 - {\cal N}_1) \int \frac {\partial V}{\partial y_1}   \qtun
	+ (\mu_2 - \dot y_2 - {\cal N}_2) \int  \frac {\partial V}{\partial y_2}   \qtun.
\end{split}
\end{equation*}

We claim the following estimates.
\begin{claim}
	Assuming \eqref{eq:gga},
	\begin{equation}\label{eq:de2}
	\left|\int \qtun \qtde\right|\leq C (y+1) e^{-y}, 
	\end{equation}
\begin{equation}\label{eq:de1}
j=1,2,\quad 
   \left\|\frac {\partial V}{\partial \mu_j} - \Lambda \qtud \right\|_{H^1}+   \left\|\frac {\partial V}{\partial y_j} + \partial_x \qtud\right\|_{L^\infty} + \frac 1{\sqrt{y}} \left\|\frac {\partial V}{\partial y_j} + \partial_x \qtud\right\|_{H^1}\leq C   e^{-y}.
\end{equation}
\end{claim}
Indeed, under assumption \eqref{eq:gga}, \eqref{eq:de2} is a consequence of \eqref{eq:24} -- see also  proof of Claim \ref{CL:21}. Moreover, 
 \eqref{eq:de1} is a   consequence of the explicit expression of $V$ (see \eqref{eq:p9} and \eqref{eq:p19}) and the properties of $A_j$, $B_j$ and $D_j$ (see proof of Proposition \ref{PR:21}).

\medskip

By \eqref{eq:de2},  \eqref{eq:de1}, \eqref{eq:p19b}, $L Q'=0$ (\eqref{eq:A9}) and $\int  Q'Q =0$, we get
\begin{equation*}
\begin{split}
	0 & = \dot \mu_1 \int \varepsilon \Lambda \qtun +  
	(\mu_1 - \dot y_1) \int \varepsilon  \partial_x  \qtun + \|\varepsilon(t)\|_{L^2}O(ye^{-y})
	+ O\left(\|\varepsilon\|_{L^2}^2\right)\\
	&   - \int E  \qtun
	- (\dot \mu_1 - {\cal M}_1) \left(\int     \qtun  \Lambda \qtun +O(e^{-y})\right)+
	(\dot \mu_2 - {\cal M}_2) O(y^2 e^{- y})\\
	& 
	+  (\mu_1 - \dot y_1 - {\cal N}_1)O(e^{- y})
	+  (\mu_2 - \dot y_2 - {\cal N}_2) 
	O(y e^{-y}).
\end{split}
\end{equation*}
Hence, by $\left| \int \qtun  \Lambda \qtun \right|\geq c_0>0$ (see \eqref{eq:A13}),
for $y$ large and $\varepsilon$ small, we get
\begin{equation*}\begin{split} 
& |\dot \mu_1 - {\cal M}_1| \leq C \left(\|\varepsilon\|_{L^2}^2 +  y e^{-y} \|\varepsilon\|_{L^2} + \int |E \qtun| \right)\\
&   + C y^2 e^{-  y}  |\dot \mu_2 - {\cal M}_2| + C  e^{-y}   |\mu_1 - \dot y_1 - {\cal N}_1 | +  C y e^{- y} |\mu_2 - \dot y_2 - {\cal N}_2 | .
\end{split}\end{equation*}
Similarly, expanding $0=\frac d{dt}\int \varepsilon  \partial_x \qtun$, we obtain
\begin{equation*}\begin{split} 
& |\mu_1 - \dot y_1 - {\cal N}_1 | \leq C \left(\|\varepsilon\|_{L^2}+  \int | E \partial_x \qtun| \right)\\
&+ C  (\|\varepsilon\|_{L^2}+e^{-\frac 12 y}) |\dot \mu_1 - {\cal M}_1|   + C e^{- \frac 12 y}  |\dot \mu_2 - {\cal M}_2| + C y e^{-y} |\mu_2 - \dot y_2 - {\cal N}_2 | .
\end{split}\end{equation*}
Combining these two estimates, together with  similar estimates for 
$|\dot \mu_2 - {\cal M}_2|$ and $|\mu_2 - \dot y_2 - {\cal N}_2 |$
for $y_0$ large and $\omega_0$ small,   \eqref{eq:ga} is proved.

\subsection{Proof of Proposition \ref{PR:cFG}}

\textbf{Proof of \eqref{eq:FGcoer}.}
The proof of \eqref{eq:FGcoer} is standard under the orthogonality conditions \eqref{eq:or}, see for example Lemma 4 in \cite{MMT}. Recall that the proof is mainly based on coercivity property of the operator $L$ under orthogonality conditions, i.e. Lemma \ref{CL:A1} (v) and on localization arguments.
Indeed, we observe in particular that locally around each soliton $\qtud$, both functionals behave essentially as 
$$
\int (\px \varepsilon)^2 + (1+\mu_j) \varepsilon^2 - 4 \qtud^3 \varepsilon^2,
$$
which is a rescaled version of $(L\varepsilon,\varepsilon)$.

\medskip

\noindent\textbf{Proof of \eqref{eq:cF}.}
We start with the following preliminary estimates.
\begin{claim}\label{cl:4.1}
\begin{equation}\label{eq:de0}
	\left\| V - \qtun - \qtde \right\|_{L^\infty} \leq C e^{-y},
\end{equation}
\begin{equation}\label{eq:de3}
 \left\| {\partial_t V}- \ETE(V) + \sum_{j=1,2} \mu_j  \partial_x \qtud \right\|_{L^\infty}\leq
C e^{- y},
\end{equation}
\begin{equation}\label{eq:dc10}
	\left\|(\Phi - \mu_j) e^{-\frac 12 |x-y_j|} \right\|_{L^\infty} 
	+ \|V \partial_x \Phi\|_{L^\infty}\leq C (|\mu_1|+|\mu_2|) e^{-2 \rho y  },
\end{equation}
\begin{equation}\label{eq:dc2}
	\left\| \Phi \partial_x V  - \sum_{j=1,2} \mu_j \partial_x \qtud \right\|_{L^\infty}
	 \leq C (|\mu_1|+|\mu_2|) e^{- 2 \rho y  } + C e^{-y},
 \end{equation}
 \begin{equation}\label{eq:de10}\begin{split}
 & \Big\|\ETE(V) - \sud (\dot \mu_j - \mathcal{M}_j) \Lambda \qtud
+ \sud (\mu_j - \dot y_j - \mathcal{N}_j) \px \qtud\Big\|_{H^1} \\ &\leq C 
\left( \|\varepsilon\|_{L^2} + \int |E| (\qtun +\qtde)\right) \sqrt{y} e^{-y}.
\end{split} \end{equation}
\end{claim}
These estimates follow from \eqref{eq:p19b}, \eqref{eq:p9}, \eqref{eq:p19}, \eqref{eq:p22}, \eqref{eq:ph}, \eqref{eq:ga}, and \eqref{eq:de1}.

\medskip

Let
$$
\Theta=
\|\varepsilon\|_{L^2}^2 
\left[e^{- \frac 34  y}   + (|\mu_1|+|\mu_2|+\|\varepsilon\|_{L^2} )( e^{-2 \rho y} +  \|\varepsilon\|_{L^2} )  \right]
+ \|\varepsilon\|_{H^1} \|E\|_{H^1}.
$$
Now, we compute $\frac d{dt}{\cal F}_+(t)$:
\begin{equation*}\begin{split}
 \frac 12 \frac d{dt}{\cal F}_+(t) & = \int \partial_t \varepsilon \left(-\partial_x^2 \varepsilon + \varepsilon - ((V+\varepsilon )^4-V^4) +
\varepsilon \Phi\right) 
  \\& - \int \partial_t V \left((V+\varepsilon)^4 - V^4 - 4V^3\varepsilon \right)
  +\frac 12 \int \varepsilon^2  \partial_t \Phi =F_1+F_2+F_3.
\end{split}\end{equation*}
Observe that $\partial_x \Phi = (\mu_1-\mu_2) \varphi'\geq 0$ in the present situation.
We claim
$$
F_1+F_2 \leq C\Theta \quad\text{and} \quad F_3 \leq C \Theta.
$$

First, using the equation of $\varepsilon$ (i.e. \eqref{eq:ep}),
\begin{equation*}\begin{split} 
F_1 & = 
- \int \left(-\partial_x^2 \varepsilon + \varepsilon  - ((\varepsilon +V)^4-V^4)\right) \px (\Phi    \varepsilon)
\\ & -
 \int \ETE(V) \left(-\partial_x^2 \varepsilon + \varepsilon  + \varepsilon \Phi  - ((\varepsilon +V)^4-V^4)\right)
\\&-
 \int E \left(-\partial_x^2 \varepsilon + \varepsilon  + \varepsilon \Phi  - ((\varepsilon +V)^4-V^4)\right)= F_{1,1}+F_{1,2}+F_{1,3}.
\end{split}\end{equation*} 
Integrating by parts and using  $|\varphi'''|\leq (8\rho)^2 |\varphi'|\leq (1/16)|\varphi'| $,
\begin{equation*}\begin{split}
F_{1,1} & =
- \frac 32 \int (\partial_x\varepsilon)^2\partial_x \Phi  
- \frac 12 \int \varepsilon^2 \partial_x \Phi
+ \frac 12 \int \varepsilon^2 \partial_x^3\Phi 
+\int \left( (\varepsilon +V)^4-V^4\right) \px (\Phi    \varepsilon)\\
& \leq - \frac 32 \int (\partial_x\varepsilon)^2\partial_x \Phi  
- \frac 38 \int \varepsilon^2 \partial_x \Phi 
+\int \left( (\varepsilon +V)^4-V^4\right) \px (\Phi    \varepsilon).
\end{split}\end{equation*}

Then, we decompose $F_{1,2}$ as follows
 \begin{equation*}
F_{1,2 }   = - \int \ETE(V) \left(-\partial_x^2 \varepsilon + \varepsilon  + \varepsilon \Phi  - 4 V^3 \varepsilon\right)+ \int \ETE(V) \left((\varepsilon +V)^4-V^4- 4 V^3\varepsilon\right) .
\end{equation*}
First, by \eqref{eq:de10} and \eqref{eq:ga}, integating by parts and using Cauchy Schwarz inequality,
\begin{align*}
& \left| \int \Big(\ETE(V) - \sud (\dot \mu_j - \mathcal{M}_j) \Lambda \qtud
+ \sud (\mu_j - \dot y_j - \mathcal{N}_j) \px \qtud\Big)
 (-\partial_x^2 \varepsilon + \varepsilon  + \varepsilon \Phi  - 4 V^3 \varepsilon)\right|
 \\
& \leq C \|\varepsilon\|_{H^1} \left( \|\varepsilon\|_{L^2} + \int |E| (\qtun +\qtde)\right) \sqrt{y} e^{-y}\leq C \Theta.
\end{align*}
Second, by \eqref{eq:de0}, \eqref{eq:dc10}  and \eqref{eq:A9}, \eqref{eq:or}, \eqref{eq:ga},
\begin{align*}
& \left| \int \Big(  - \sud (\dot \mu_j - \mathcal{M}_j) \Lambda \qtud
+ \sud (\mu_j - \dot y_j - \mathcal{N}_j) \px \qtud\Big)
 (-\partial_x^2 \varepsilon + \varepsilon  + \varepsilon \Phi  - 4 V^3 \varepsilon)\right|
 \\
& \leq C \|\varepsilon\|_{H^1} \left( \|\varepsilon\|_{L^2} + \int |E| (\qtun +\qtde)\right) 
\left( C (|\mu_1|+|\mu_2|) e^{-2 \rho y  } +\sqrt{y} e^{-y}\right)\leq C \Theta.
\end{align*}
Thus,
$$
\left|\int \ETE(V) \left(-\partial_x^2 \varepsilon + \varepsilon  + \varepsilon \Phi  - 4 V^3 \varepsilon \right)\right|\leq C \Theta.
$$

Moreover, integating by parts,
\begin{equation*}
	|F_{1,3}|\leq C \|E\|_{H^1} \|\varepsilon\|_{H^1}\leq C \Theta.
\end{equation*}

Next, we combine
$F_2  = - \int \partial_t V \left((V+\varepsilon)^4 - V^4 - 4V^3\varepsilon \right)$ with
the remaining terms from $F_{1,1}$ and $F_{1,2}$.
Using \eqref{eq:dc2} and \eqref{eq:de3}, we have
\begin{equation*}\begin{split} 
& \int \left( (\varepsilon +V)^4-V^4\right) \px (\Phi    \varepsilon)
  +\int \ETE(V) ((\varepsilon +V)^4-V^4- 4 V^3\varepsilon) +F_2\\
  &= O\left(\|\varepsilon\|_{L^2}^2 (e^{-y}+e^{-2 \rho y} (|\mu_1|+|\mu_2|))\right)\\ 
  &+\int \left( (\varepsilon +V)^4-V^4\right) \px (\Phi    \varepsilon)
  +\int ((\varepsilon +V)^4-V^4- 4 V^3\varepsilon) \Phi \px V.
\end{split}\end{equation*} 
Moreover, integrating by parts,
\begin{align*}
 &   \int \Phi\left( (\varepsilon +V)^4-V^4\right)     \px \varepsilon
  +\int\Phi ((\varepsilon +V)^4-V^4- 4 V^3\varepsilon)  \px V
\\ & = \int   \Phi \px\left( \tfrac 15 ( \varepsilon +V)^5 - \tfrac 15 V^5 - V^4 \varepsilon\right) 
 =   - \int (\px \Phi) \left( \tfrac 15 ( \varepsilon +V)^5 - \tfrac 15 V^5 - V^4 \varepsilon\right) ,
\end{align*}
and so by \eqref{eq:dc10},
\begin{align}
&\left|\int \left( (\varepsilon +V)^4-V^4\right) \px (\Phi    \varepsilon)
  +\int ((\varepsilon +V)^4-V^4- 4 V^3\varepsilon) \Phi \px V\right|
\nonumber\\
&= \left|\int \partial_x \Phi \left( \varepsilon (\varepsilon+V)^4 -\tfrac 15 ( \varepsilon +V)^5  + \tfrac 15 V^5 \right)\right|
\leq  C \|V (\px \Phi)\|_{L^\infty} \|\varepsilon\|_{L^2}^2 
+ C \|\varepsilon\|_{H^1} \int \varepsilon^2 \px \Phi\nonumber\\
&\leq  C \Theta+ C \|\varepsilon\|_{H^1} \int \varepsilon^2 \px \Phi .\label{eq:ast}
\end{align}
We deduce
$ 
F_1+F_2 \leq C \Theta.
$ 

Finally
\begin{equation*}\begin{split} 
F_3=
& \frac 12 \int (\dot \mu_1 \varphi + \dot \mu_2 (1-\varphi)) \varepsilon^2,
\end{split}\end{equation*}
so that by $|{\cal M}_j|\leq C e^{-y}$ and \eqref{eq:ga},
\begin{equation}
|F_3|\leq  C (e^{-y}+ \sum_{j=1,2} |\dot \mu_j - {\cal M}_j|)  \|\varepsilon\|_{H^1}^2\leq C \Theta.
\end{equation}

\noindent\textbf{Proof of \eqref{eq:cG}.}
Since $\mu_2(t)\geq \mu_1(t)$ we have
$\frac 1{(1+\mu_1(t))^2}\geq  \frac {1}{(1+\mu_2(t))^2}$,
$\partial_x \Phi_1\geq 0$ and $\partial_x \Phi_2\leq 0$.
Note also that by explicit computations, for $\mu_j$ small enough:
\begin{equation}\label{eq:cP}
	\left|  \partial_x\Phi_1   + 2 \partial_x\Phi_2  \right|\leq C(|\mu_1|+|\mu_2|) \partial_x\Phi_1; 
\end{equation}
and similarly to Claim \ref{cl:4.1} (\eqref{eq:de3}--\eqref{eq:dc2}), we have
\begin{equation}\label{eq:dc23}\begin{split}
	&\left\| \Phi_1 (\partial_t V -\ETE(V)) + \Phi_2 \px V  \right\|_{L^\infty}
	+\|V (\px \Phi_1)\|_{L^\infty}  + \|V (\px \Phi_2)\|_{L^\infty}
	\\ & \leq C (|\mu_1|+|\mu_2|) e^{- 2 \rho y  } + C e^{-y}.
\end{split} \end{equation}
We compute $\frac d{dt}{\cal F}_-(t)$:
\begin{equation*}\begin{split}
\frac 12 \frac d{dt}{\cal F}_-(t) & 
=\frac 12 \int \left\{ \Big[ (\partial_x \varepsilon)^2 + \varepsilon^2 - \tfrac 25 \left((\varepsilon+V)^5 - V^5 - 5 V^4\varepsilon \right)\Big]\partial_t \Phi_1
+   \varepsilon^2 \partial_t \Phi_2\right\}\\ &+ \int \partial_t \varepsilon\left(-\partial_x^2 \varepsilon + \varepsilon - ((\varepsilon +V)^4-V^4)\right)\Phi_1 - \int \partial_t \varepsilon \partial_x \varepsilon \partial_x \Phi_1
+ \int  \partial_t \varepsilon  \,  \varepsilon \Phi_2  \\
& - \int \partial_t V \left((\varepsilon+V)^4 - V^4 - 4V^3\varepsilon \right) \Phi_1.
\end{split}\end{equation*}
The first term is treated as the term $F_3$ above.
By \eqref{eq:dc23}, we thus obtain
\begin{equation*}\begin{split}
& \frac 12 \frac d{dt}{\cal F}_-(t)   = \int \partial_t \varepsilon\left(-\partial_x^2 \varepsilon + \varepsilon - ((\varepsilon +V)^4-V^4)\right)\Phi_1 - \int \partial_t \varepsilon \partial_x \varepsilon \partial_x \Phi_1
+ \int  \partial_t \varepsilon  \,  \varepsilon \Phi_2  \\
& - \int \ETE(V) \left((\varepsilon+V)^4 - V^4 - 4V^3\varepsilon \right) \Phi_1
+\int \Phi_2 \px V \left((\varepsilon+V)^4 - V^4 - 4V^3\varepsilon \right)  
+O(\Theta).
\end{split}\end{equation*}

Using the equation of $\varepsilon$ and integrating by parts, arguing as in the proof of \eqref{eq:cF},  we get
\begin{align}
&\int \partial_t \varepsilon\left(-\partial_x^2 \varepsilon + \varepsilon - ((\varepsilon +V)^4-V^4)\right)\Phi_1 
  - \int \ETE(V) \left((\varepsilon+V)^4 - V^4 - 4V^3\varepsilon \right) \Phi_1\nonumber\\
& = -\frac 12 \int \left(-\partial_x^2 \varepsilon + \varepsilon - ((\varepsilon +V)^4-V^4)\right)^2
\px \Phi_1 -
\int \ETE(V) \left(-\partial_x^2 \varepsilon + \varepsilon - 4 V^3 \varepsilon\right)\Phi_1 \nonumber\\
& - \int E \left(-\partial_x^2 \varepsilon + \varepsilon - ((\varepsilon +V)^4-V^4)\right)\Phi_1
\nonumber\\ & \leq 
-\frac 12 \int (\px^2 \varepsilon)^2 \px \Phi_1 -\frac 78 \int (\px \varepsilon)^2 \px \Phi_1
- \frac 38 \int \varepsilon^2 \px \Phi_1+
 \int \ETE(V) \varepsilon \Phi_2 + C\Theta .\label{eq:num}
\end{align}
Next, by integration by parts,  \eqref{eq:ph} and \eqref{eq:dc23}, and arguing as in the proof of
\eqref{eq:ast},
\begin{align*}
 - \int \partial_t \varepsilon \partial_x \varepsilon \partial_x \Phi_1
&=  \int  \left(-\partial_x^2 \varepsilon + \varepsilon - ((\varepsilon +V)^4-V^4)\right)
\px(\px\varepsilon \partial_x \Phi_1) + ( \ETE(V) + E) \px\varepsilon \px \Phi_1\\
& \leq -\int (\px^2 \varepsilon)^2 \px \Phi_1   -\frac 78  \int (\px\varepsilon)^2 \px \Phi_1 
+ C \|\varepsilon\|_{H^1} \int \varepsilon^2 \px \Phi_1 + C\Theta. 
\end{align*}
Finally, again  by integration by parts,  \eqref{eq:ph} and \eqref{eq:ast},
\begin{align*}
& \int  \partial_t \varepsilon  \,  \varepsilon \Phi_2 + \int \ETE(V) \varepsilon \Phi_2
 +\int \Phi_2 \px V \left((\varepsilon+V)^4 - V^4 - 4V^3\varepsilon \right)\\
& = - \int  \left(-\partial_x^2 \varepsilon + \varepsilon - ((\varepsilon +V)^4-V^4)\right)
 \px ( \varepsilon \Phi_2 ) -\int  E \varepsilon   \Phi_2\\
& +\int \Phi_2 \px V \left((\varepsilon+V)^4 - V^4 - 4V^3\varepsilon \right)
  \leq \frac 32 \int (\px \varepsilon)^2 |\px \Phi_2| + \frac 34\int  \varepsilon^2 |\px \Phi_2| + C \Theta.
\end{align*}
Using \eqref{eq:cP}, we check that the two positive terms above are compensated 
by the term  \eqref{eq:num} up to terms of order $\Theta$, and the proof of \eqref{eq:cG} is now complete.

\subsection{Proof of Proposition \ref{PR:STAB}.}
By classical arguments (based on the implicit function theorem -- see e.g. Lemma \ref{PR:de}, Lemma \ref{le:huit} and \cite{MMT}), there exists $\omega_1>0$, $\bar y_0>1$ such that if
\begin{equation}\label{eq:close}
\text{$\inf_{y_1-y_2 >\bar y_0 }\|u(t)- Q_{1-\mu_0}(.-y_1)-  Q_{1+\mu_0}(.-y_2)\|_{H^1}\leq \omega_1   $}
\end{equation}
then  $u(t)$ can be decomposed as follows
\begin{equation}
\label{eq:d0}
	u(t,x)=\qb_1(t,x)+\qb_2(t,x)+{\bar \varepsilon}(t,x),
\end{equation}
where
\begin{equation}
\label{eq:pm}
	\qb_1(t,x)= Q_{1- \mu_0}(x-y_1(t)), \quad \qb_2(t,x)= Q_{1+ \mu_0}(x-y_2(t))
\end{equation}
and $y_j(t)$ are $C^1$ functions uniquely chosen so that
\begin{equation}
\label{eq:oh}
\int {\bar \varepsilon}(t,x) \partial_x \qb_j(t,x) dx = 0.
\end{equation}
Moreover, $\|\bar \varepsilon\|_{H^1} \leq C \omega_1$.
Note that this decomposition is similar to the one of Lemma \ref{PR:de}, except that for simplicity, we do not  adjust the scaling parameter by modulation (it is not required here since the variation of the scalings of the solitons is quadratic in $\bar \varepsilon$ in this regime -- see Claim \ref{eq:scg} where we control  $\int \bar \varepsilon \qb_j$  using the conservation laws). Moreover we use modulation around the family of sums of two solitons, not around a more sophisticated approximate solution such as $V$.

Let 
 $y(t)=y_1(t)-y_2(t).$
 For future reference, note that \begin{equation}\label{eq:inter}
	\int \qb_2(t) \qb_1(t) \leq C e^{-\frac 34 y(t)},
	\end{equation}
	(see the proof of a similar estimate in the proof of Claim \ref{CL:21}). Note that
one cannot obtain an estimate of the form $C e^{-y}$ in this situation.

By \eqref{eq:A1} and the equation of $u(t)$, the functions ${\bar \varepsilon}(t,x)$ and $y_j(t)$ satisfy the following equation
\begin{equation}\label{eq:et}\begin{split}
	&  \partial_t {\bar \varepsilon} 
	+ \partial_x\left(\partial_x^2 {\bar \varepsilon} - {\bar \varepsilon} +  (\qb_1+\qb_2 + {\bar \varepsilon})^4 - (\qb_1+\qb_2)^4\right)\\ & = 
	- \partial_x \left((\qb_1+ \qb_2)^4 - \qb_1^4 -\qb_2^4\right)  + (\dot y_1+\mu_0  )   \partial_x  \qb_1
+ (\dot y_2 - \mu_0 )  \partial_x    \qb_2.
\end{split}\end{equation}
By  \eqref{eq:inter}, as in the proof of  Lemma \ref{PR:de},  we obtain
\begin{equation}
\label{eq:eQp}
\left|  \mu_0  +  \dot y_1\right|+\left|  \mu_0 -  \dot y_2\right|\leq
C \left( \|{\bar \varepsilon}\|_{H^1} +  e^{ -\frac 34 y}\right).
\end{equation}

\noindent \emph{Proof of \eqref{eq:stab+}.}\quad 
For  $C_*>2$  to be chosen later, 
assume \eqref{eq:clo} and define 
\begin{equation*}\begin{split}
T^*  =  \sup \Big\{ & t_0<T<- (\rho\mu_0)^{-1} |\log \mu_0|
\hbox{ such that, for all $t_0<t<T$,
$u(t)$ satisfies \eqref{eq:close},} \\ & \text{ $\|{\bar \varepsilon}(t)\|_{H^1} \leq C_*  \omega\mu_0 + C_* e^{- 4 \rho {\mu_0}  |t|}$  and $y(t) >  \frac 32 \mu_0 |t|$.}\Big\}
\end{split}\end{equation*}
Note that for $C^*$ large enough, $T^*$ is well-defined by \eqref{eq:clo} and by continuity of $u(t)$ in $H^1$.

We prove that  $T^*=-  (\rho\mu_0)^{-1} |\log \mu_0|$, for $C^*$ large enough, assuming by contradiction 
that $T^* < - (\rho\mu_0)^{-1} |\log \mu_0|$ and working on the time interval $[t_0,T^*]$.

First, we claim the following control of the scaling directions of ${\bar \varepsilon}(t)$.

\begin{claim}\label{eq:scg}
For all $t\in [t_0,T^*]$,
	\begin{equation}\label{eq:tQ}
		\left|\int {\bar \varepsilon}(t,x) \qb_j(t,x) dx \right|\leq 
		C \left(\mu_0^{-1} \sup_{[t_0,t]}  \|{\bar \varepsilon}\|_{H^1}^2  + \sup_{[t_0,t]} e^{-\frac 12 y} + \mu_0 \omega\right).
	\end{equation}
\end{claim}
\begin{proof}[Proof of Claim \ref{eq:scg}]
We obtain \eqref{eq:scg}   by expanding $u(t)=\qb_1(t)+\qb_2(t)+{\bar \varepsilon}(t)$ in the conservation laws \eqref{eq:i5} and \eqref{eq:i6}using \eqref{eq:A1} and \eqref{eq:inter} 
\begin{align*}
   M(u(t_0))  & = M(\qb_1(t_0)) + M(\qb_2(t_0))+ 2 \int \bar \varepsilon(t_0) \qb_1(t_0)\\ &
+  2 \int \bar \varepsilon(t_0) \qb_2(t_0) +O(e^{-\frac 34 y(t_0)})+O(\|\bar \varepsilon(t_0)\|^2_{H^1})	 \\
&= M(u(t))  =  M(\qb_1(t)) + M(\qb_2(t))+ 2 \int \bar \varepsilon(t) \qb_1(t)\\ &
+  2 \int \bar \varepsilon(t)  \qb_2(t) +O(e^{-\frac 34 y(t)})+O(\|\bar \varepsilon(t)\|^2_{H^1});
\end{align*}
\begin{align*}
  \mathcal{E}(u(t_0)) &  = \mathcal{E}(\qb_1(t_0)) + \mathcal{E}(\qb_2(t_0))- 2 (1-\mu_0) \int \bar \varepsilon(t_0)  \qb_1(t_0) \\ &
-   2 (1+\mu_0)\int \bar \varepsilon(t_0)  \qb_2(t_0) +O(e^{-\frac 34 y(t_0)})+O(\|\bar \varepsilon(t_0)\|^2_{H^1})	 \\
& = \mathcal{E}(u(t))  =  \mathcal{E}(\qb_1(t)) + \mathcal{E}(\qb_2(t))- 2 (1-\mu_0) \int \bar \varepsilon(t)  \qb_1(t) \\ &
-  2 (1+\mu_0)\int \bar \varepsilon(t)  \qb_2(t) +O(e^{-\frac 34 y(t)})+O(\|\bar \varepsilon(t)\|^2_{H^1});	
\end{align*}
Using $M(\qb_j(t_0))=M(\qb_j(t))$, $\mathcal{E}(\qb_j(t_0))=\mathcal{E}(\qb_j(t))$, 
$\|\bar\varepsilon(t_0)\|_{H^1}\leq C \mu_0 \omega$, and combining the above estimates, we find \eqref{eq:tQ}.
\end{proof}

Now, we use a functional $\bar {\cal F}$ similar to $\mathcal{F}_-$ introduced in Proposition \ref{PR:cFG}. Let
\begin{equation*}
	{\bar {\cal F}}(t) =
		\int \left[(\partial_x {\bar \varepsilon})^2 + {\bar \varepsilon}^2 - \tfrac 25 \left(({\bar \varepsilon}+\qb_1+\qb_2)^5 - (\qb_1+\qb_2)^5 - 5 (\qb_1+\qb_2)^4{\bar \varepsilon} \right)\right]\bar \Phi_1 
		+\int   {\bar \varepsilon}^2 \bar\Phi_2,
\end{equation*}
where, $\varphi$ being defined in \eqref{eq:ph},  
$$\bar\Phi_1(x)=\frac {\varphi(x)}{(1-\mu_0)^2}+ \frac {1-\varphi(x)}{(1+\mu_0)^2} ,\quad 
\bar\Phi_2(x)=\frac {-\mu_0 \varphi(x)}{(1-\mu_0)^2}+ \frac {\mu_0(1-\varphi(x))}{(1+\mu_0)^2} .$$
We perform similar (and simpler) computations as the ones of Propositions \ref{PR:cFG} and \ref{PR:cFG}
(scaling parameters and $\Phi_{j}$ are time independent). We obtain
\begin{equation}\label{eq:ctG}
\frac d{dt}{\bar {\cal F}} (t) \leq C \|{\bar \varepsilon}\|_{L^2} \left( e^{- 2 \rho y } \|{\bar \varepsilon}\|_{L^2} + e^{-\frac 34 y}\right).
\end{equation}
In particular, we point out that the term $\|\varepsilon\|_{H^1}^4$ does not appear in this estimate,
since it was only due to the scaling modulation in the proof of Proposition \ref{PR:cFG}.

Note also that from \eqref{eq:clo}, we have $\|\bar \varepsilon(t_0)\|_{H^1}\leq C \omega \mu_0$ and thus $\bar{\cal F} (t_0)\leq C \omega^2 \mu_0^2$ .
Integrating 
  \eqref{eq:ctG} from $t_0$ to $t$ ($t_0<t<T^*$), using the definition of $T^*$, we obtain, for
$t_0\leq t\leq T^*$, for $\mu_0$ small enough (possibly depending on $C^*$)
\begin{equation*}\begin{split}
\bar {\cal F}(t) & \leq C \frac {C_*^2  } {\mu_0}  e^{- 3 \rho {\mu_0}    |t|} 
\left(e^{- 8 \rho {\mu_0}    |t|} + \omega^2\mu_0^2\right) + C\omega^2 \mu_0^2\\
& \leq C C_*^2   \mu_0^2 \, e^{- 8 \rho {\mu_0}    |t|}  + C\omega^2 \mu_0^2.
\end{split}\end{equation*}

Moreover, from  \eqref{eq:oh}, \eqref{eq:tQ} and standard arguments,
\begin{align}
\sup_{[t_0,T^*]}\|{\bar \varepsilon}\|_{H^1}^2\nonumber
&\leq C \sup_{[t_0,T^*]}{\bar {\cal F}}  + C^2  \sup_{[t_0,T^*]} \sum_{j=1,2}\left(\int {\bar \varepsilon}(t)\qb_{j}(t) \right)^2\\
& \leq C_1  C_*^2   \mu_0^2 \, e^{- 8 \rho {\mu_0}    |t|}  + C_1 \sup_{[t_0,T^*]} e^{- y} + C_1\omega^2 \mu_0^2 ,
\label{eq:ccG}
\end{align}
where $C_1>0$ is independent of $C_*$.
Choosing $C_*^2>4 C_1$ and then $\mu_0$ small enough, from \eqref{eq:ccG}, we get for all $t\in [t_0,T^*]$, $$\|{\bar \varepsilon}(t)\|_{H^1}^2\leq \frac 12 C_*^2 e^{- 8 \rho {\mu_0}    |t|}+ \frac 12 C_*^2 \omega^2 \mu_0^2.$$ Thus,   by a standard continuity argument, we have just contradicted the definition of $T^*$.

Finally, from \eqref{eq:eQp}, we have
$
	|\dot y + 2 \mu_0 |\leq C_1 \sup_{[t_0,T^*]}\|{\bar \varepsilon}\|_{H^1} + C_1 e^{ -\frac 34 y},
$ and thus  $y(t)> \frac 74 \mu_0 |t|$.

\medskip

\noindent\emph{Proof of \eqref{eq:stab-}.} The proof of \eqref{eq:stab-} is completely similar, replacing the  functional $\bar {\cal F}$ by a similar functional inspired by the functional ${\cal F}_+$ in \eqref{eq:en}.

\medskip

\noindent\emph{Proof of \eqref{eq:as2}.} The stability result being established for all $t<t_0$, the asymptotic stability is a consequence of \cite{MM1} and \cite{Ma}.

\subsection{Proof of Proposition \ref{PR:SHARP}}
The proof is similar to the one of Propositions  \ref{PR:cFG} and \ref{PR:STAB}.
We obtain a better result (no exponential error term), since Propositions  \ref{PR:cFG} and \ref{PR:STAB} compare an exact solution with  an approximate solution, whereas Proposition \ref{PR:SHARP} compares two exact solutions.

We  assume ${{\mathcal T}_1}\in [-T,T]$ (the case $|{\mathcal T}_1| > T$ is similar) and
we prove the stability result on $(-\infty,{{\mathcal T}_1}]$, the stability proof for $[{\mathcal T}_1,+\infty)$ following from similar arguments.

\medskip

 For $X_1,X_2 \in \RR$, let $U_{X_1,X_2}$ be the unique  solution of {\eqref{eq:KDV}} such that
	\begin{equation}\label{eq:X1X2}
		\lim_{t\to -\infty} \|U_{X_1,X_2}(t) - Q_{1-\mu_0}(x+ \mu_0 t-X_1  ) - Q_{1+\mu_0}(x - \mu_0 t -X_2)\|_{H^1} =0.
\end{equation}
Then, as a direct consequence of the uniqueness of the asymptotic $2$-soliton solution for given parameters (see \cite{Ma2}, Theorem 1), 
for any $Y_1,Y_2\in \RR$, one has
	\begin{equation}\label{eq:Y1Y2}\begin{split}
		& U_{Y_1,Y_2}(t,x) = U_{X_1,X_2}(t-T_0,x-X_0) \\ & \text{where} \quad X_0 = \frac 12 (Y_1-X_1) + \frac 12 (Y_2-X_2),
		\ T_0 = \frac 1{2\mu_0}\left( (Y_1-X_1) -  (Y_2-X_2)\right).
	\end{split}\end{equation}
	In particular, the map $(X_1,X_2)\mapsto U_{X_1,X_2}$ is smooth and
	\begin{equation}\label{eq:map}
		\frac {\partial U_{X_1,X_2}}{\partial X_j} = (-1)^{j}\frac 1{2 \mu_0} \frac {\partial U_{X_1,X_2}}{\partial t} 
		- \frac 12 \frac {\partial U_{X_1,X_2}}{\partial x}.
	\end{equation} 

For $C^*>2$ to be chosen, we define
\begin{equation}\label{eq:sh5}  
	T^* = \inf \{ t \leq {{\mathcal T}_1} \ ;   \text{ s.t. for all $t\leq t'\leq {{\mathcal T}_1}$, }  	\inf_{X_1,X_2}\|u(t')-U_{X_1,X_2}\|_{H^1}\leq C^* \omega \mu_0 \}.
	\end{equation}
By the assumption on $u({{\mathcal T}_1})$ and continuity of $u(t)$ in $H^1$, $T^*<{{\mathcal T}_1}$ is well-defined. We prove that $T^*=-\infty$ by contradiction~: we assume   $T^*>-\infty$ and we  strictly improve the estimate of $\inf_{X_1,X_2}\|u(t)-U_{X_1,X_2}\|_{H^1}$ on $[T^*,{{\mathcal T}_1}]$, which contradicts the definition of $T^*$.

By Proposition \ref{pr:st}, $U_{X_1,X_2}$ is close to the sum of two separated solitons. 
Thus, on $[T^*,{{\mathcal T}_1}]$, for $\omega$ small enough, we use modulation theory (as in the proof of Lemma \ref{PR:de}) to prove the existence of $X_1(t),$ $X_2(t)$ such that for ${\tilde U}(t,x)=U_{X_1(t),X_2(t)}(t,x)$,
\begin{equation}\label{eq:modW} 
		u(t,x)={\tilde U}(t,x)+ \tilde { \varepsilon}(t,x),	 \quad 	\int \tilde { \varepsilon} \frac {\partial \tilde U}{\partial X_j} =0. 
\end{equation}
(Note that this orthogonality condition is similar to $\int \varepsilon \partial_x \qtud=0$ in Lemma \ref{PR:de}.)
Moreover,  there exist $\tilde \mu_j(t)$, $\tilde y_j(t)$ such that
\begin{equation}\label{eq:sh6bis}
  \left\|\tilde U(t) -   Q_{1+\tilde \mu_1(t)}(.-\tilde y_1(t))
  - Q_{1+\tilde \mu_2(t)}(.-\tilde y_2(t))\right\|_{H^1}\leq C |\ln \mu_0|^{1/2} \mu_0^2.
\end{equation}
Next, $\|\tilde { \varepsilon}(t)\|_{H^1} \leq C C^* \omega \mu_0$, \begin{equation}\label{eq:sh8}
 \partial_t \tilde { \varepsilon} + \partial_x (\partial_x^2 \tilde { \varepsilon} -\tilde { \varepsilon} +   ({\tilde U}+ \tilde { \varepsilon})^4  -{\tilde U}^4)
	+ \sum_{j=1,2} \dot X_j   \frac {\partial \tilde U}{\partial X_j} = 0,
\end{equation}
\begin{equation}\label{eq:sh9}
	\text{and} \quad |\dot X_1|+|\dot X_2|\leq C \| \tilde { \varepsilon}\|_{H^1}.
\end{equation}
From Proposition \ref{pr:st}, there exists $t_0$ such that $\tilde \mu_{1}(t)>\tilde \mu_{2}(t)$ if $t>t_0$ and
$\tilde \mu_{1}(t)<\tilde \mu_{2}(t)$ if $t<t_0$. Assume   that $t_0 < T^*$.
In this case, to control $\tilde { \varepsilon}(t)$ on $[T^*,{{\mathcal T}_1}]$   we use the functional
$$
	\tilde{\mathcal{F}}(t)= \int \left[(\partial \tilde { \varepsilon})^2 + \tilde { \varepsilon}^2 -\tfrac 25 ((\tilde { \varepsilon}+{\tilde U})^5 - {\tilde U}^5 - 5{\tilde U}^4 \tilde { \varepsilon}) \right]\tilde \Phi_1
	+ \int   \tilde { \varepsilon}^2 \tilde \Phi_2,
$$
similar to $\mathcal{F}_-(t)$, 
for $\tilde \Phi_j$ defined from $\tilde \mu_j(t)$ as in Proposition \ref{PR:cFG}.
To treat the case $T^*<t_0$, one uses a functional similar to $\mathcal{F}_+(t)$.

\begin{claim}\label{CL:B5} For all $t\in [T^*,{{\mathcal T}_1}]$,
\begin{equation}\label{eq:cGW}
	\frac d{dt}\tilde{\mathcal{F}}(t)
	\geq - C \|\tilde { \varepsilon}(t)\|_{L^2}^2 e^{-(\frac 12 + \rho) \YYzz} \quad (C>0),
\end{equation}
\begin{equation}\label{eq:cGWW}
	\tilde{\mathcal{F}}(t)\geq \lambda \|\tilde { \varepsilon}(t)\|_{L^2}^2  -
	C \sum_{j=1,2} \left(\int  \tilde { \varepsilon} (t) Q_{1+ \tilde\mu_j(t)}(x-\tilde y_j(t))\right)^2.
\end{equation}
\begin{equation}\label{eq:cGWWW}
	\left| \int  \tilde { \varepsilon} (t) Q_{1+ \tilde\mu_j(t)}(x-\tilde y_j(t)) \right|\leq 
	C \omega \mu_0,
\end{equation}
where $\lambda >0$ and $C>0$ are independent of $C^*$.
\end{claim}
Assuming this claim, we integrate \eqref{eq:cGW} on $[T^*,\mathcal{T}_1]$, 
and then use a combination of \eqref{eq:cGWW} and \eqref{eq:cGWWW} to  contradict the defintion of $T^*$, for $C^*$ large enough and $\mu_0$ small enough.
Note that the estimates on $|\dot X|$ and $|\dot T|$ follow from \eqref{eq:sh9} and \eqref{eq:Y1Y2}.

\begin{proof}[Proof of Claim \ref{CL:B5}.]
To prove \eqref{eq:cGW}, we use the same argument as in the proof of Propositions \ref{PR:cFG} and \ref{PR:STAB}, except that in the present situation there is no error term $E(t,x)$, and no scaling parameter. Moreover, to replace \eqref{eq:dc23}, we use the following estimate
\begin{equation}\label{eq:sh7}
	\|\Phi_2 \partial_x {\tilde U} - \Phi_1  \partial_x (\partial_x^2 {\tilde U} - {\tilde U} + {\tilde U}^4)\|_{L^2}
	\leq C e^{-(\frac 12 + \rho)\YYzz}.
\end{equation}
Note that  \eqref{eq:cGWW} is a standard coercivity property, see Claim \ref{CL:A1}.

Finally, we prove \eqref{eq:cGWWW}.
On $[T^*,{\mathcal T}_1]$, we have from \eqref{eq:sh8} (see also proof of Lemma \ref{PR:de})
\begin{equation}\label{eq:ddt}\begin{split}
	\left|\frac d {dt} \int \tilde { \varepsilon} (t)   Q_{1+\tilde\mu_j(t)}(x-\tilde y_j(t))\right|
& \leq  C |\ln \mu_0| \mu_0^2 \|\tilde { \varepsilon}(t)\|_{H^1}+ C \|\tilde { \varepsilon}(t)\|_{H^1}^2
\\ &\leq C C^* \omega \mu_0^2 ( |\ln \mu_0| \mu_0 + C^* |\ln \mu_0|^{-2} ),
\end{split}\end{equation}
since $0<\omega < |\ln \mu_0|^{-2}$.
Integrating on $(t,{\mathcal T}_1])$, for $t\in [T^*,{\mathcal T}_1]$, using $0<T< C |\ln(\mu_0)| \mu_0^{-1}$, we find
\begin{equation*}
	\left| \int \tilde { \varepsilon} (t)   Q_{1+ \tilde\mu_j(t)}(x-\tilde y_j(t))\right|
\leq C \omega \mu_0 + C C^* \omega \mu_0( \mu_0 |\ln \mu_0|^2 + C^* |\ln \mu_0|^{-1})
  \leq 2 C  \omega \mu_0,
\end{equation*}
for $\mu_0$ small enough depending on $C^*$.
\end{proof}

\end{document}